\documentclass[12pt]{amsart}
\usepackage{amssymb, amsmath, amsfonts, amsthm, graphics, mathrsfs}
\usepackage[hmargin=1 in, vmargin = 1 in]{geometry}
\usepackage{hyperref}
\usepackage{tikz-cd}
\usepackage{mathabx}
\usetikzlibrary{calc}
\usepackage[all]{xy}
\usepackage{float}
\newcommand{\sslash}{\mathbin{/\mkern-6mu/}}

\newcommand{\newword}[1]{\textbf{#1}}
\newcommand{\defn}[1]{\textbf{#1}}
\newcommand{\defnintro}[1]{{\it #1}}

\DeclareMathOperator{\rank}{rank}
\DeclareMathOperator{\Hom}{Hom}
\DeclareMathOperator{\Ext}{Ext}

\DeclareMathOperator{\Ker}{Ker}

\DeclareMathOperator{\Spec}{Spec}

\DeclareMathOperator{\Res}{Res}
\DeclareMathOperator{\sgn}{sgn}
\DeclareMathOperator{\Frac}{Frac}

\renewcommand{\AA}{\mathbb{A}}

\newcommand{\CC}{\mathbb{C}}

\newcommand{\FF}{\mathbb{F}}
\newcommand{\GG}{\mathbb{G}}
\newcommand{\HH}{\mathbb{H}}

\newcommand{\PP}{\mathbb{P}}
\newcommand{\QQ}{\mathbb{Q}}
\newcommand{\RR}{\mathbb{R}}

\newcommand{\ZZ}{\mathbb{Z}}

\newcommand{\cA}{\mathcal{A}}

\newcommand{\cF}{\mathcal{F}}

\newcommand{\cO}{\mathcal{O}}

\newcommand{\fg}{\mathfrak{g}}

\newtheorem{theorem}{Theorem}

\newtheorem{thm}[theorem]{Theorem}
\newtheorem{Theorem}[theorem]{Theorem}
\newtheorem*{TheoremIntro}{Theorem}
\newtheorem{proposition}[theorem]{Proposition}
\newtheorem{prop}[theorem]{Proposition}

\newtheorem{cor}[theorem]{Corollary}
\newtheorem{conjecture}[theorem]{Conjecture}

\newtheorem{lemma}[theorem]{Lemma}
\newtheorem{lem}[theorem]{Lemma}

\newtheorem{remark}[theorem]{Remark}
\newtheorem{Remark}[theorem]{Remark}
\numberwithin{theorem}{section}
\newtheorem{example}[theorem]{Example}

\def\oF{\overline{\FF}}
\newcommand{\lcm}{\mathrm{lcm}}

\newcommand{\SL}{\mathrm{SL}}

\DeclareMathOperator{\dlog}{\mathrm{dlog}}
\DeclareMathOperator{\diag}{\mathrm{diag}}
\DeclareMathOperator{\Gr}{\mathrm{Gr}}

\def\tC{\tilde C}
\def\tB{\tilde B}
\def\wB{\widehat{B}}
\def\wC{\widehat{C}}
\def\x{\mathbf{x}}

\def\z{\mathbf{z}}

\def\Id{\mathrm{Id}}
\def\bgamma{\bar \gamma}
\def\oG{\vec{\Gamma}}
\def\Dir{\mathrm{Dir}}
\title{Cohomology of cluster varieties. I. \\ Locally acyclic case}
\author[Thomas Lam]{Thomas Lam}
\address{Department of Mathematics, 
University of Michigan, Ann Arbor, MI 48109, USA.}
\email{tfylam@umich.edu}
\thanks{T.L. was partially supported by NSF grants DMS-1160726, DMS-1464693, and a Simons Fellowship. D.E.S. was partially supported by NSF grant DMS-1600223.}
\author[David E Speyer]{David E Speyer}
\email{speyer@umich.edu}
\begin{document}
\begin{abstract}
We initiate a systematic study of the cohomology of cluster varieties.  We introduce the Louise property for cluster algebras that holds for all acyclic cluster algebras, and for most cluster algebras arising from marked surfaces.  For cluster varieties satisfying the Louise property and of full rank, we show that the cohomology satisfies the curious Lefschetz property of Hausel and Rodriguez-Villegas, and that the mixed Hodge structure is split over $\QQ$.  We give a complete description of the highest weight part of the mixed Hodge structure of these cluster varieties, and develop the notion of a standard differential form on a cluster variety.  We show that the point counts of these cluster varieties over finite fields can be expressed in terms of Dirichlet characters.  Under an additional integrality hypothesis, the point counts are shown to be polynomials in the order of the finite field.
\end{abstract}

\maketitle
\setcounter{tocdepth}{1}
\tableofcontents

\section{Introduction} 
Cluster algebras are certain commutative algebras introduced by Fomin and Zelevinsky \cite{CA1} that have found applications in many areas of mathematics.  The aim of this article is to initiate a systematic study of the singular cohomology of a \defnintro{cluster variety} $\cA$, the spectrum of a cluster algebra $A$ defined over the complex numbers.  By definition, a cluster algebra $A$ is a subalgebra of a rational function field generated by (possibly infinitely many) cluster variables.  We begin by discussing conditions guaranteeing that $\cA$ is a smooth complex affine algebraic variety, and in particular, a complex manifold.  The cohomology $H^*(\cA,\CC)$ can then be calculated via (algebraic) deRham cohomology.

\subsection{Local acyclicity and the Louise property}
Let $A$ be a rank $n$ skew-symmetric cluster algebra of geometric type with $m$ frozen variables, defined over $\CC$.  We note that our convention is to invert the frozen variables; see Section~\ref{sec:cluster background} for a full introduction to cluster algebras.
 Thus $A = A(\tB)$ is determined up to isomorphism by an $(n+m) \times n$ \defnintro{extended exchange matrix} $\tB$, an integer matrix whose top $n \times n$ submatrix is skew-symmetric.  \defnintro{Mutation} produces new exchange matrices $\tB'$ from $\tB$; two mutation-equivalent exchange matrices define the same cluster algebra.  

The \defnintro{quiver} of $\tB$ is the directed graph $\oG = \oG(\tB)$ with vertex set $1$, $2$, \dots, $n$, and an edge $i \to j$ if $\tB_{ij}>0$.  Associated to any subset $\oG' \subset [n]$ is a cluster algebra $ A(\tB_{\oG'})$, obtained by freezing the cluster variables $\{ x_i \}_{i \not \in \oG'}$, with extended exchange matrix $\tB_{\oG'}$ given by taking the columns of $\tB$ indexed by $\oG'$.  We say that $A(\tB)$ is \defnintro{acyclic} if $\oG(\tB')$ has no directed cycles, for some mutation $\tB'$ of $\tB$.  

Following Muller~\cite{Mul}, we define $i \to j$ to be a \defnintro{separating edge} of $\oG$ if there is no bi-infinite path through the edge $i \to j$. 

The Louise property for cluster algebras is defined recursively.  We say that a cluster algebra $A  = A(\tB)$ satisfies the \defnintro{Louise property} if either 
\begin{enumerate}
\item $\oG(\tB)$ has no edges, or  
\item for some $\tB'$ mutation-equivalent to $\tB$, the quiver $\oG = \oG(\tB')$ of $\tB'$ has a separating edge $i \to j$, such that the cluster algebras $A(\tB'_{\oG \setminus \{i\}}), A(\tB'_{\oG \setminus \{j\}})$, and $A(\tB'_{\oG \setminus \{i,j\}})$ all satisfy the Louise property.
\end{enumerate}

Any acyclic cluster variety satisfies the Louise property and any cluster variety satisfying the Louise property is \defnintro{locally acyclic}: it can be covered by finitely many acyclic cluster localizations.  Muller \cite{Mul} showed that locally acyclic cluster algebras enjoy many favorable properties; in particular, under a full rank condition locally acyclic cluster varieties are smooth.  In this paper, we establish general properties of cohomology of cluster algebras satisfying the Louise property or satisfying local acyclicity.  In the sequel \cite{LS}, we give more explicit results for acyclic cluster algebras.


\subsection{Mixed Hodge structure of cluster varieties}
Deligne \cite{Del} defined a mixed Hodge structure on the cohomology groups $H^\ast(X,\CC)$ for a complex algebraic variety $X$.  It induces a \defnintro{Deligne splitting} $H^k(X,\CC) = \bigoplus_{p,q} H^{k,(p,q)}(X)$.  We say that $H^\ast(X)$ is of \defnintro{mixed Tate type} (also known as \defnintro{Hodge-Tate type}) if $H^{k,(p,q)}(X) = 0$ for $p \neq q$.  We say that the mixed Hodge structure of $H^\ast(X)$ is \defnintro{split over $\QQ$} if each summand $H^{k,(p,q)}(X)$ has a basis coming from $H^\ast(X,\QQ)$.  (This implies that $H^\ast(X)$ is a direct sum of pure Hodge structures.)

We say that a cluster algebra $A(\tB)$, or the cluster variety $\cA =\Spec(A)$, is of \defnintro{full rank} if the matrix $\tB$ has full rank.

\begin{TheoremIntro}[proved in Section \ref{ssec:curious}]
Suppose that $\cA$ satisfies the Louise property and is of full rank.  Then $\cA$ is a smooth complex affine algebraic variety.  The mixed Hodge structure of $H^\ast(\cA,\CC)$ is of mixed Tate type, and is split over $\QQ$.
%
%
\end{TheoremIntro}

Let $(x_1, x_2, \ldots, x_{n+m})$ be a cluster with extended exchange matrix $\tilde{B}$.  Define a \defnintro{Gekhtman -- Shapiro--Vainshtein form}, or \defnintro{GSV form} for short, to be any $2$-form
\[
\gamma = \sum_{i,j=1}^{n+m}  \widehat{B}_{ij} \frac{d x_i}{x_i}  \wedge \frac{dx_j}{x_j}
\]
where $\widehat{B}$ is a $(n+m) \times (n+m)$ skew symmetric matrix extending $\tB$.  The importance of this closed $2$-form was first identified by Gekhtman, Shapiro and Vainshtein \cite{GSV}, and is independent of the choice of cluster.  It defines a cohomology class $[\gamma] \in H^{2,(2,2)}(\cA,\CC)$.  We say that $\gamma$ has \defnintro{full rank} if $\wB$ has full rank.


Suppose that $X$ is a smooth $2d$-dimensional affine variety and $[\gamma] \in H^{2,(2,2)}(X,\CC)$.  We say that the \defnintro{curious Lefschetz property} holds for the pair $(X,[\gamma])$, if $H^{\ast}(X,\CC)$ is of mixed Tate type and if, for all $s \geq 0$ and all $p \leq d$, the map 
\[ [\gamma]^{d -p} : H^{p+s, (p,p)}(X, \CC) \to H^{2d-p+s, (2d-p, 2d-p)}(X, \CC) \]
is an isomorphism.  The curious Lefschetz property was formalized by Hausel and Rodriguez-Villegas \cite{HR} in their study of the mixed Hodge polynomials of character varieties.

\begin{TheoremIntro}[proved in Section \ref{ssec:curious}]
Suppose $\cA$ is an even-dimensional cluster variety satisfying the Louise property and is of full rank.   Then the pair $(\cA,[\gamma])$ satisfies the curious Lefschetz property for any full rank GSV-form $\gamma$.
\end{TheoremIntro}

Even when $\cA$ is odd dimensional, the numerical consequences of the curious Lefschetz property hold:
\begin{TheoremIntro}[proved in Section \ref{ssec:curious}]
Suppose $\cA$ is an $e$-dimensional cluster variety satisfying the Louise property and is of full rank.  Then $\dim H^{p+s, (p,p)}(\cA, \CC) = \dim H^{e-p+s, (e-p, e-p)}(\cA, \CC)$.
\end{TheoremIntro}

The reader may like to turn now to Section~\ref{sec:examples} to see typical tables of $\dim H^{k,(p,p)}(\cA,\CC)$.

\subsection{Standard forms on cluster varieties}
Let $T_r \simeq (\CC^\ast)^r$ be an $r$-dimensional complex torus.  A cohomology class $[\omega] \in H^\ast(T_r,\CC)$ can be uniquely represented by a \defnintro{standard form} 
\[
\omega = \sum_{I = \{i_1,i_2,\ldots,i_k\}} a_I \;\frac{dx_{i_1}}{x_{i_1}} \wedge \frac{dx_{i_2}}{x_{i_2}} \wedge \cdots \wedge \frac{dx_{i_k}}{x_{i_k}}, \qquad a_I \in \CC.
\]
Each cluster $(x_1,\ldots,x_{n+m})$ of a cluster algebra $A$ has an associated cluster torus $T \subset \cA$.  A regular differential form on a cluster variety $\cA$ is called \defnintro{standard} if it restricts to a standard form on every cluster torus $T \subset \cA$.  It turns out that the space of standard forms on $\cA$ has many descriptions.

\begin{TheoremIntro}[proved in Sections \ref{ssec:standard} and \ref{ssec:neighbor}]
Suppose that $\cA$ is locally acyclic and of full rank.  Then the following vector spaces are isomorphic, by the natural isomorphisms:
\begin{enumerate}
\item The space of standard differential forms on $\cA$.
\item The space of regular differential forms on $\cA$ which are standard on some cluster torus.
\item For any cluster torus $T \subset \cA$, the space of standard differential forms on $T$ which extend to regular forms on its neighboring tori.
\end{enumerate}
\end{TheoremIntro}

Define the \defnintro{standard part} of the cohomology of $\cA$ by $H^\ast(\cA)_{st} := \bigoplus_k H^{k,(k,k)}(\cA)$; this is the highest weight part of cohomology.  
Like for the cohomology of a torus, the standard cohomology of a cluster variety is uniquely represented by standard forms.

\begin{TheoremIntro}[proved in Section \ref{ssec:standard}]
Suppose that $\cA$ is locally acyclic and of full rank. 
The natural map $\omega \mapsto [\omega] \in H^\ast(\cA)$ sends the space of standard forms on $\cA$ isomorphically onto the standard part $H^\ast(\cA)_{st}$.
\end{TheoremIntro}

Let $\Gamma$ denote the underlying undirected graph of $\oG(\tB)$.  Suppose the connected components of $\Gamma$ are $\Gamma_1,\ldots,\Gamma_r$.  We say that $\gamma_a$ is a GSV form for $\Gamma_a$ if it is a GSV form for $A(\tB_{\Gamma_a})$.  

\begin{TheoremIntro}[proved in Sections \ref{ssec:generators}--\ref{sec:proofstandardbasis}]
Suppose that $\cA$ is locally acyclic and of full rank.   Then $H^{\ast}(\cA)_{st}$ is generated by GSV 2-forms $\gamma_1,\ldots,\gamma_r$ for the connected components of $\Gamma$ and the 1-forms $\dlog y_1$, $\dlog y_2$, \dots, $\dlog y_m$, where $y_i$ are the frozen variables. If the connected components of $\Gamma$ have cardinalities $n_1$, $n_2$, \dots, $n_r$, then the Poincar\'{e} series of $H^{\ast}(\cA)_{st}$ is
\[ P(H^{\ast}(\cA)_{st},t) = (1+t)^{m-r} \prod_{i=1}^r \left( 1+ t+ \cdots + t^{n_i} + t^{n_i+1} \right) . \]
\end{TheoremIntro}  

\subsection{Point counts of cluster varieties}
Cluster algebras (and thus cluster varieties) can, and often are, defined over $\ZZ$ as the $\ZZ$-subalgebra of an ambient field generated by cluster variables. 
For a finite field $\FF_q$ with $q$ elements, we write $\cA(\FF_q)$ for the set of $\FF_q$-points of $\cA$.

The set of all functions $f: \{ \mbox{prime powers} \} \to \CC$ has a ring structure obtained by pointwise addition and multiplication.  Let $\Lambda$ denote the subring generated over $\ZZ$ by the functions $f(q) = q$ and the Dirichlet characters $\chi(q)$. 

\begin{TheoremIntro}[proved in Section \ref{sec:pointcounts}]
Suppose that $\cA$ satisfies the Louise property and is of full rank.  Then there is an element $g \in \Lambda$ such that for sufficiently large primes $p$, we have $g(q) = \#\cA(\FF_q)$ for any prime power $q = p^a$.  If, in addition, the rows of $\tB$ span $\ZZ^n$, then $\# \cA(\FF_q)$ is a polynomial in $q$ for all prime powers $q$.  
\end{TheoremIntro}

Furthermore, we show that the eigenvalues of Frobenius acting on the $\ell$-adic cohomology groups $H^\ast(\cA_{\oF_p},\QQ_\ell)$ are of the form $\chi(p)p^s$, where $\chi$ is a Dirichlet character and $s$ is an integer.  We remark that Chapoton \cite{Cha} has also studied point counts of varieties closely related to cluster varieties.

\subsection{Applications}
Due to the ubiquity of cluster structures, we expect a wide variety of applications.  We indicate a few here.
\subsubsection{Open Richardson varieties}

Let $G$ be a simple complex algebraic group with Lie algebra $\fg$, and let $B, B_-$ be opposite Borel subgroups.  Let $G/B$ be the flag variety with Schubert cells $\mathring X_w := B w B/B \subset G/B$ and opposite Schubert cells $\mathring X^w:=B_- wB/B \subset G/B$, where $w$ denotes an element of the Weyl group of $G$.

The \defnintro{open Richardson variety} is the intersection $R^u_v = \mathring X_v \cap \mathring X^u$.  It is a smooth affine algebraic variety, and is nonempty when $u \leq v$ in Bruhat order.  A folklore conjecture states that the coordinate ring $\cO(R^u_v)$ is a cluster algebra.  Leclerc \cite{Lec} has defined a candidate cluster algebra $A^u_v$ and proved that $A^u_v$ injects into $\cO(R^u_v)$.  Goodearl and Yakimov \cite{GY} have studied the closely related conjecture for double Bruhat cells.  Muller and Speyer \cite{MS} have studied the special case of open projected Richardson varieties in the Grassmannian. \footnote{Very recently, the mixed Hodge structure of ``open positroid varieties" were computed in \cite{GL2}, where our curious Lefschetz theorem is applied to settle unimodality and symmetry questions in $q,t$-Catalan theory.} 

We refine the conjecture as follows.
\begin{conjecture}\label{conj:Richardson}\footnote{For recent progress on this conjecture, see \cite{SSBW,GL,Ing}.}
The open Richardson variety $R^u_v$ is a cluster variety satisfying the Louise property and of full rank.
\end{conjecture}

In \cite{MS}, Muller and Speyer show that the Grassmannian is a cluster variety satisfying the Louise property.

The varieties $R^u_v$ were studied by Kazhdan and Lusztig.  The point counts $\#R^u_v(\FF_q)$ are given by their $R$-polynomials \cite{KL}.

The singular cohomology $H^*(R^u_v,\CC)$ of an open Richardson variety has an interpretation as a higher extension group of Verma modules.  Let $\cO$ denote the Category $\cO$ of the simple Lie algebra $\fg$.  For $w \in W$, let $M_w$ be the Verma module in the principal block of $\cO$ indexed by $w$.  The following folklore result follows from the localization theorem of Beilinson and Bernstein, see \cite[Proposition 4.2.1]{RSW}.

\begin{theorem} \label{thm:BB} For $u,v \in W$, we have an isomorphism $H^\ast(R^u_v,\CC) \cong \Ext^{\ell(v)-\ell(u)-\ast}_{\cO}(M_u,M_v)$. 
\end{theorem}

The computation of $\Ext^\ast_{\cO}(M_u,M_v)$ is an open problem.

The isomorphism of Theorem \ref{thm:BB} is compatible with mixed Hodge structures.  Namely, by a result of Soergel, the principal block $\cO_0$ is equivalent to the module category ${\rm Mod}(R)$ of some associative algebra $R$.  The algebra $R$ is naturally graded, and this induces a second grading on $\Ext^\ast_{\cO}(M_u,M_v)$ that matches with the weight filtration of $H^\ast(R^u_v,\CC)$, see \cite{RSW}.  

Thus Conjecture \ref{conj:Richardson} combined with our curious Lefschetz theorem (Theorem \ref{thm:curious}) would imply that the algebra $\Ext^\ast_{\cO}(M_u,M_v)$ satisfies the curious Lefschetz property.  As far as the authors are aware, such a structure is not predicted from representation theory.

\subsubsection{Character varieties}
Let $G$ be a complex reductive algebraic group and let $\Sigma$ be a topological surface.  By a {\it character variety}, we mean the space ${\rm Hom}(\pi_1(\Sigma),G)\sslash G$ or a variant thereof.  Cluster structures on character varieties are a fundamental ingredient in the work of Fock and Goncharov \cite{FG} on higher Teichm\"uller theory.  Fomin, Shapiro, and Thurston \cite{FST} gave an explicit construction of a cluster algebra for a triangulable marked surface.  On the other hand,  Hausel and Rodriguez-Villegas studied the mixed Hodge polynomials of certain character varieties in \cite{HR}, where they discovered the curious Lefschetz property.  We do not know the exact relationship between our cluster varieties and the character varieties of Hausel and Rodriguez-Villegas, but we do expect a deep relation.  As we now explain, our results apply to the cluster algebras constructed by Fomin, Shapiro, and Thurston.

Let $(\Sigma,M)$ be a triangulable marked surface.  Thus $\Sigma$ is a topological surface possibly with boundary, and $M$ is a finite subset of marked points in $\Sigma$.  The triangulable condition excludes a number of small exceptional cases; see \cite{FST} or \cite[Section 9]{Mul} for the full list.  Fomin, Shapiro, and Thurston \cite{FST} associate to $(\Sigma,M)$ a cluster algebra $A(\Sigma,M)$.  We have the following result, which has an identical proof to Muller's local acyclicity result \cite[Theorems 10.5 and 10.6]{Mul}.

\begin{theorem}\label{thm:surfaces}
Suppose that the triangulable surface $(\Sigma, M)$ is either
\begin{enumerate}
\item a union of triangulable disks, or 
\item $M$ contains at least two points in each component of $\Sigma$.
\end{enumerate}
Then $A(\Sigma,M)$ satisfies the Louise property.
\end{theorem}

The authors do not know whether or not the quivers associated to groups other than $SL_2$ will be Louise. 
Indeed, we find it quite mysterious whether or not a general cluster algebra should be expected to be Louise: The only cases which have been proved not to be Louise are certain algebras of finite mutation type, but it is unclear whether this is because most cluster algebras are Louise or because we have a lack of techniques for proving such results.

\subsection{Outline}
In Section \ref{sec:cohom}, we review basic facts concerning mixed Hodge structures, and formulate results concerning logarithmic forms on smooth complex varieties.  In Section \ref{sec:curious}, we discuss the curious Lefschetz property.  We show that the curious Lefschetz property implies the splitting of mixed Hodge structures and show that it is compatible with Mayer-Vietoris sequences.  In Section \ref{sec:cluster background}, we recall terminology on cluster algebras and some fundamental results concerning locally acyclic cluster algebras.  In Section \ref{sec:autom}, we discuss automorphisms and covering maps of cluster varieties.  We describe quotients of a cluster variety in terms of exchange matrices.  In Section \ref{sec:examples}, we give examples of mixed Hodge numbers of cluster varieties illustrating various results of the paper.  In Section \ref{sec:isolated}, we give a complete description of the mixed Hodge structure of isolated cluster varieties.  In Section \ref{sec:mixedhodge}, we prove our main results on the mixed Hodge structures of cluster varieties satisfying the Louise property.  In Section \ref{sec:standard}, we give multiple descriptions of standard forms on cluster varieties.  We describe generators for the standard cohomology and prove a formula for the Poincar\'{e} series of standard cohomology.  In Section \ref{sec:pointcounts}, we study point counts of locally acyclic cluster varieties over finite fields, and corresponding eigenvalues of Frobenius in $\ell$-adic cohomology.

\subsection{Acknowledgments}
We thank Tamas Hausel, Allen Knutson, Greg Muller, Nicholas Proudfoot, Vivek Shende, Dylan Thurston and Geordie Williamson for interesting discussions.
Computations leading to the results reported here were done with the \texttt{deRham} and \texttt{deRhamAll} commands from the \texttt{Macaulay2} package \texttt{Dmodules} (author Harrison Tsai); the authors thank ``Anton" from the Macaulay2 Google Group, and Bennet Fauber, for their help.

\section{Mixed Hodge structures and logarithmic forms}\label{sec:cohom}
\subsection{Deligne splitting}
General references for mixed Hodge structures are \cite{Del,PS}.  Let $X$ be a $d$-dimensional complex algebraic variety.  For each $k$, there is an increasing weight filtration
\[
0 = W_{-1} \subseteq W_0 \subseteq \cdots \subseteq W_{2k} = H^k(X,\QQ)
\]
and a decreasing Hodge filtration
\[
H^k(X,\CC)= F^0 \supseteq F^1 \supseteq \cdots \supseteq F^d \supseteq F^{d+1} = 0
\]
such that for every $0 \leq p \leq j$ one has
\[
\Gr_j^W (H^k(X,\CC)) = F^p \Gr_j^W (H^k(X,\CC)) \oplus \overline{ F^{j-p+1}\Gr_j^W (H^k(X,\CC))}
\]
where $\Gr_j^W (H^k(X,\CC))$ is the associated graded spaces of the weight filtration tensored with $\CC$. In other words, $F^p$ and $\overline{F^p}$ induce opposed filtrations on $\Gr_j^W(H^k(X,\CC))$.  The two filtrations together gives the mixed Hodge structure on $H^k(X,\CC)$.

\begin{lemma}[{\cite[Lemma-Definition 3.4]{PS}}] \label{L:bigrading}
Define the subspace $H^{k,(p,q)}$ of $H^k(X, \CC)$ by
\[
H^{k,(p,q)}:= F^p \cap W_{p+q} \cap\left(\overline{F^q} \cap W_{p+q} + \sum_{j \geq 2} \overline{F^{q-j+1}} \cap W_{p+q-j} \right).
\]
Then the $H^{k,(p,q)}$ satisfy
\begin{enumerate}
\item
$W_\ell(H^k(X,\CC)) = \bigoplus_{p+q \leq \ell} H^{k,(p,q)}$,
\item $ 
F^p(H^k(X,\CC)) = \bigoplus _{p' \geq p,\ 0 \leq q} H^{k,(p',q)}$,
\item $H^{k,(p,q)} = \overline{H^{k,(q,p)}} \mod W_{p+q-1}$.
\end{enumerate}
\end{lemma}

The decomposition of $H^k(X,\CC)$ into the subspaces $H^{k,(p,q)}$ is known as the \defn{Deligne splitting}.  The mixed Hodge numbers of $H^k(X)$ are $h^{k,(p,q)} := \dim H^{k,(p,q)}$. 
As Lemma~\ref{L:bigrading} shows, the data of the Deligne splitting is equivalent to that of the weight and Hodge filtrations; we will generally prefer to work with the Deligne splitting.

\begin{proposition}\label{prop:splitMV}
The Deligne splitting is functorial, and is respected by the boundary maps in Mayer-Vietoris sequences.
\end{proposition}



\begin{Theorem}[{see \cite[Theorem 5.39]{PS}}]
\label{T:weights}
If $X$ is compact, then $H^k(X,\QQ)$ has weights in $\{0,1,\ldots, k\}$; if $X$ is smooth, then $H^k(X,\QQ)$ has weights in $\{k,k+1,\ldots, 2k\}$.
\end{Theorem}
In particular, if $X$ is smooth and projective then $H^k(X,\QQ)$ is pure of weight $k$.  The compactly supported cohomology $H_c^\ast(X)$ can also be given a mixed Hodge structure.

\begin{theorem}[{see \cite[Theorem 6.23]{PS}}]\label{thm:Poincare}
For a $d$-dimensional smooth complex variety $X$ we have Poincar\'{e} duality, giving a perfect pairing
\[
H^k(X,\QQ) \times H_c^{2d-k}(X,\QQ) \to \QQ(-d)
\]
compatible with mixed Hodge structures.  Here $\QQ(-d)$ is the pure mixed Hodge structure on $\QQ$ with weight $2d$ and Hodge filtration $F^d=\QQ$ and $F^{d+1} = 0$.
\end{theorem}

\subsection{Mixed Tate type and split over $\QQ$}
We say that a mixed Hodge structure $(H,W,F)$ is of \defn{mixed Tate type} if $\Gr_j^W(H) = 0$ for $j$ odd, and $
\Gr_{2j}^W(H) = F^j \Gr_{2j}^W(H)$.
This condition is equivalent to the condition $H^{k,(p,q)}=0$ for $p \neq q$.  When $H^{\ast}(X)$ is of mixed Tate type, we have
\[ H^{k, (p,p)}(X) = W_{2p}(H^k(X,\CC))  \cap F^p(H^k(X)). \]

\begin{lemma}\label{lem:mixedTate}
Suppose that $X$ is covered by two open sets $U$ and $V$.  If $U$, $V$, and $U \cap V$ have cohomology of mixed Tate type, then so does $X$.
\end{lemma}
\begin{proof}
The compatibility of the Deligne splitting with the Mayer-Vietoris sequence shows that $H^{\ast,(p,q)} = 0$ for $p \neq q$.
\end{proof}


We say that the mixed Hodge structure of $H^\ast(X)$ is \defn{split over $\QQ$} if each summand $H^{k,(p,q)}(X)$ has a basis coming from $H^\ast(X,\QQ)$.  If $H^\ast(X)$ is split over $\QQ$ then $F^p(H^k(X,\CC))$ is invariant under complex conjugation, and it follows that $H^{\ast}(X)$ is of mixed Tate type.  Furthermore, in this case the mixed Hodge structure on $H^k(X)$ is split: it is the direct sum of pure Hodge structures on the subspaces $H^{k,(p,p)}(X)$ as $p$ varies.

%
%


\subsection{Log forms}

We introduce an easy way to recognize that a $k$-form is in $H^{k,(k,k)}$. Define the ring of \newword{log forms} on $X$ to be the sub-$\CC$-algebra of $\Omega^{\ast}(X)$ generated by the forms $d \log f : = df/f$, for $f$ a unit in $\cO(X)$. Note that log forms are automatically closed.

\begin{lemma} \label{lem:logImpliesTate}
Let $\omega$ be a log form of degree $k$ on a smooth variety $X$. Then the cohomology class $[\omega]$ represented by $\omega$ is in $H^{k,(k,k)}(X)$. 
\end{lemma}

\begin{proof}
Since wedge product respects the Deligne grading, it is enough to check the claim for $k=1$. Since a sum of forms in degree $(1,1)$ is in degree $(1,1)$, it is enough to check the claim for a form of the type $d \log f$. We can consider $f$ as a map from $X$ to $\CC^{\ast}$. Then $d \log f$ is $f^{\ast} \theta$, where $\theta$ is a generator for $H^1(\CC^{\ast}, \ZZ)$. The form $\theta$ is in Deligne degree $(1,1)$, so its pullback is as well.
\end{proof}

We say that a regular differential form $\eta$ on $X$ is a \newword{rational log form} if there is a dense Zariski open set $U$ of $X$ such that $\eta|_{U}$ is a log form on $U$. 
Since $U$ is dense in $X$, and $d \eta|_U=0$, rational log forms are also closed, and we can discuss their cohomology classes. Note that a rational log form is required to be defined on all of $X$.

We give the first example of a rational log form which is of importance to us. Let $\cA = \{ (x,x',y) : x x' = y+1 ,\ y \neq 0 \} \subset \CC^3$. Letting $T$ and $T'$ be the open sets $x \neq 0$ and $x' \neq 0$ in $\cA$, we note that $T \cup T'$ is all of $\cA$ except the single point $(x,x',y) = (0,0,-1)$. Let $\gamma$ be the differential form which is $\dlog x \wedge \dlog y$ on $T$ and $-\dlog x' \wedge \dlog y$ on $T'$.  Since $\cA \setminus (T \cup T')$ is a single point on the smooth surface $\cA$, the form $\gamma$ extends to $\cA$. (Explicitly, $\gamma = (dx \wedge dx')/y$.)  On $\cA$, we cannot express $\gamma$ as a log form, since the only units in $\cO(\cA)$ are $\CC^{\ast} \cdot y^{\ZZ}$. 
Nonetheless, we want to use the fact that $\gamma$ can be written as a log form on the tori $T$ and $T'$, so we introduce the notion of a rational log form.

We prove several lemmas on rational log forms. In particular, note that Proposition~\ref{prop:rationalLogImpliesTate} is strictly stronger than Lemma~\ref{lem:logImpliesTate}.
\begin{lemma}\label{lem:logform}
Let $\overline{X}$ be a smooth compactification of $X$, with $D:=\overline{X} \setminus X$ a simple normal crossings divisor, and let $\eta$ be a rational log form on $X$. Then $\eta$ extends to a section in $H^0(\overline{X}, \Omega^k(\log D))$.
\end{lemma}

\begin{proof}
By assumption, the only poles of $\eta$ are along the components of $D$. Choose a particular component $H$ of $D$ to focus on, and let $z$ be a local coordinate near $D$. Write $\eta$ as a polynomial in forms of the form $\dlog f_i$, for various meromorphic functions $f_i$, and let $f_i = z^{n_i} u_i$, for $u_i$ with neither zeroes nor poles on $H$. Then $\dlog f_i = n_i \dlog z + \dlog u_i$ and $\eta$ is easily seen to be of the form $\dlog z \wedge \eta_1 + \eta_2$, where neither $\eta_1$ nor $\eta_2$ has a pole along $H$, as desired.
\end{proof}

The following result can be thought of as a ``Hodge theorem" for log forms; giving a preferred de Rham representative for any class which can be represented as a rational log form. 

\begin{lemma} \label{lem:dlogzero} 
Let $X$ be a smooth complex variety. The ring of rational log forms injects into $H^{\ast}(X)$.
\end{lemma}

\begin{proof}
Let $\eta$ be a $p$-form on $X$, and assume that $\eta$ is a rational log form. Our goal is to show that, if $\eta$ is exact, then it is zero.

Choose a normal crossings compactification $\overline{X}$ of $X$ and let $D = \overline{X} \setminus X$. By Lemma \ref{lem:logform},  $\eta$ can be considered as an element of $H^0(\overline{X}, \Omega^p(\log D))$. 

There is a spectral sequence $H^q(\overline{X}, \Omega^p(\log D)) \implies H^{p+q}(X, \CC)$ which degenerates at $E_1$ \cite[Theorem 8.35]{Voisin}.
So $H^0(\overline{X}, \Omega^p(\log D))$ injects into $H^p(X, \CC)$. This map explicitly consists of considering $\eta$ as a closed $p$-form on $X$ and taking the cohomology class. So our hypothesis that $\eta$ is exact says that it maps to $0$ under this injection. We conclude that $\eta=0$.
\end{proof}

Lemma~\ref{lem:dlogzero} has a fascinating corollary for the structure of cluster algebras. The second author learned that this might be true from Paul Hacking, and understands it to have been developed in collaboration that lead to~\cite{GHKK}.

\begin{cor}
Let $A$ be a cluster algebra and let $(x_1, \ldots, x_{n+m})$ and $(z_1, \ldots, z_{n+m})$ be two clusters of $A$. Write $f = \sum c_I x^I = \sum d_J z^J$. Then $c_{0 \cdots 0} = d_{0 \cdots 0}$.
\end{cor}

\begin{proof}
Note that, if $(x_1, \ldots, x_{n+m})$ and $(x'_1, \ldots, x'_{n+m})$ are adjacent clusters, then $\dlog x_1 \wedge \cdots \wedge\dlog x_{n+m} = - \dlog x'_1 \wedge \cdots \wedge \dlog x'_{n+m}$. So $\dlog x_1 \wedge \cdots \wedge\dlog x_{n+m} = \pm \dlog z_1 \wedge \cdots \wedge\dlog z_{n+m}$; reorder the $z_i$ if necessary so that the sign is $+$ and define this $(n+m)$-form to be $\omega$. Let $T$ and $T'$ be the open tori corresponding to $(x_1, \ldots, x_{ n+m})$ and $(z_1, \ldots, z_{n+m})$.
 
Then $(f-c_{0 \cdots 0}) \omega$ is exact on $T$ and $(f-d_{0 \cdots 0 }) \omega$ is exact on $T'$. So both of these forms are exact on $T \cap T'$, and their difference $(c_{0 \cdots 0} - d_{0 \cdots 0}) \omega$ is exact on $T \cap T'$. But $(c_{0 \cdots 0} - d_{0 \cdots 0}) \omega$ is a rational log form, so this shows $(c_{0 \cdots 0} - d_{0 \cdots 0}) \omega=0$ and $c_{0 \cdots 0} = d_{0 \cdots 0}$.
\end{proof}

The following Lemma is merely a stepping stone to Proposition~\ref{prop:rationalLogImpliesTate}, which is strictly stronger than it.
\begin{lem} \label{lem:rationalLogTopHodge}
A rational log form of degree $k$ on $X$ represents a cohomology class in $F^k H^k(X)$.
\end{lem}

\begin{proof}
Choose a normal crossings compactification $\overline{X}$ of $X$. On $\overline{X}$, consider the log de Rham complex $\Omega_{\overline{X}}^\bullet(\log D) := (0 \to \cO(D) \to \Omega^1(\log D) \to \Omega^2(\log D) \to \cdots)$, and its $k$-th truncation $\Omega_{\overline{X}}^{\geq k}(\log D) :=(0 \to 0 \to \cdots \to 0 \to \Omega^k(\log D) \to \Omega^{k+1}(\log D) \to \cdots)$. By \cite[p. 208]{Voisin}, $F^k H^k(X) = {\rm Im}(\HH^k(\Omega_{\overline{X}}^{\geq k}(\log D)) \to \HH^k(\Omega_{\overline{X}}^\bullet(\log D)))$ is the image of the $k$-th hypercohomology of the truncated complex in the $k$-th hypercohomology of the log de Rham complex. 

Consider the \u{C}ech spectral sequence for the two complexes. On the $E_1$ page for the full log de Rham complex, we have $H^i(\overline{X}, \Omega^j(\log D))$ and for the truncated complex we have this in the columns where $j \geq k$. The maps on the $E_1$ page are $0$ -- by \cite[Theorem 8.35]{Voisin} for the full log de Rham complex, and by compatibility of the spectral sequences for the truncated one. So a class in $H^0(\overline{X}, \Omega^k(\log D))$ represents a class in the $k$-th hypercohomology of the truncated complex, and hence a class in $F^k H^k(X)$.
\end{proof}

\begin{lem} \label{ref:lemBlowUp}
Let $X$ be a smooth variety and let $X' \to X$ be a blow up of $X$ along a smooth center. Then the pull back map $F^k H^k(X) \to F^k H^k(X')$ is an isomorphism.
\end{lem}

When $X$ is proper, this is the well known fact that the Hodge groups $H^{k,(k, q)}(X) \cong H^q(X, \Omega^k)$ are birational invariants for proper models $X$. However, we will be using this result for $X$ which are not proper.

\begin{proof}
By Lemma~\ref{L:bigrading} and \cite[Proposition 4.20]{PS}, we have $F^k H^k(U) = \bigoplus_q H^{k,(k,q)}(U)$ when $U$ is smooth.
Let $V$ be the center of the blow up, and let $E$ be the exceptional divisor, which is a projective bundle over $V$. Then we have an exact sequence \cite[Corollary 5.37]{PS} 
\[ \cdots \to H^k(X) \to H^k(X') \oplus H^k(V) \to H^k(E) \to H^{k+1}(X) \to \cdots ,\]
respecting the Deligne splitting. Let $0 \leq q \leq k$. Since $H^{k-1}(E)$ has no terms in Deligne degree $(k,q)$, there is an exact sequence
\[ 0 \to  H^{k, (k,q)}(X) \to H^{k, (k,q)}(X') \oplus H^{k, (k,q)}(V) \to H^{k, (k,q)}(E) \to \cdots .\]
Since $E$ is a projective bundle over $V$, using the standard description of cohomology of a projective bundle, the map $H^{k, (k,q)}(V) \to H^{k, (k,q)}(E)$ is an isomorphism. So $H^{k, (k,q)}(X) \to H^{k, (k,q)}(X')$ is an isomorphism. 
\end{proof}

\begin{lem} \label{lem:topHodgeInjective}
Let $X$ be a smooth variety and let $U$ be a dense Zariski open subset of $X$. For every positive integer $k$, the map $F^k H^k(X) \to F^k H^k(U)$ is injective.
\end{lem}

\begin{proof}
By a sequence of blowups at smooth centers in $X \setminus U$, we can find $\pi: X' \to X$ such that $\pi^{-1}(U) \to U$ is an isomorphism and $X' \setminus U$ is a normal crossing divisor, $X' \setminus U = D_1 \cup D_2 \cup \cdots \cup D_r$. By Lemma~\ref{ref:lemBlowUp}, $F^k H^k(X) \to F^k H^k(X')$ is an isomorphism, so it is enough to show that $F^k H^k(X') \to F^k H^k(U)$ is injective. We can decompose this map as $F^k H^k(X') \to F^k H^k(X' \setminus D_1)  \to F^k H^k(X' \setminus ( D_1 \cup D_2) ) \to \cdots \to   F^k H^k(X' \setminus \bigcup_{i=1}^r D_i ) = F^k H^k(U)$. So we need to show that, if $Y$ is a smooth variety and $D$ a smooth hypersurface, then $F^k H^k(Y) \to F^k H^k(Y \setminus D)$ is injective.

We use the Gysin long exact sequence
\begin{equation}\label{eq:Gysin} \cdots \to H^{k-2}(D) \to  H^k(Y) \to H^k(Y \setminus D) \to H^{k-1}(D) \to \cdots . \end{equation}
This sequence is more commonly stated for compactly supported cohomology (see \cite[p.138]{PS}); in our situation, $Y,Y \setminus D$, and $D$ are all smooth so we may apply Poincar\'{e} duality (Theorem \ref{thm:Poincare}) to obtain \eqref{eq:Gysin}.

The Deligne splitting on $H^{k-2}(D)$ lives in weights $(p,q)$ for $0 \leq p,q \leq k-2$. The boundary map in the Gysin sequence sends $H^{k-2, (p,q)}(D)$ to $H^{k, (p+1, q+1)}(Y)$. 
In particular, for all $p \geq k$, we have an exact sequence
\[ 0 \to H^{k, (p,q)}(Y) \to H^{k, (p,q)}(Y \setminus D)  \to \cdots \]
so $H^{k, (k,q)}(Y) \to H^{k, (k,q)}(Y \setminus D)$ is injective as required.
\end{proof}

Finally, we can prove:
\begin{prop} \label{prop:rationalLogImpliesTate} 
Let $X$ be a smooth variety and $\eta$ a rational log $k$-form. Then $\eta$ represents a cohomology class in $H^{k,(k,k)}(X)$.
\end{prop}

\begin{proof}
By Lemma~\ref{lem:rationalLogTopHodge}, $\eta$ represents a class in $F^k H^k(X) = \bigoplus_{q=0}^k H^{k, (k,q)}(X)$. Let the class $[\eta]$ of $\eta$ be $[\eta] = \sum_{q=0}^k \theta_q$, where $\theta_q \in H^{k,(k,q)}(X)$.
Let $U$ be a dense Zariski open set on which $\eta$ becomes a log form. Then $[\eta|_U] = \sum_{q=0}^k \theta_q|_U$. The restriction $\theta_q|_U$ lies in $H^{k,(k,q)}(U)$. But, by Lemma~\ref{lem:logImpliesTate}, $[\eta|_U] \in H^{k,(k,k)}(U)$. So, for $0 \leq k < q$, we have $\theta_k|_U=0$.

But, by Lemma~\ref{lem:topHodgeInjective}, $H^{k, (k,q)}(X)$ injects into $H^{k,(k,q)}(U)$. So, for $0 \leq k < q$, the cohomology class $\theta_k$ is zero. In other words, $[\eta] = \theta_k$. We have shown that $\eta$ represents a class in $H^{k, (k,k)}(X)$, as desired.
\end{proof}

\section{The curious Lefschetz property}
\label{sec:curious}

Suppose that $X$ is a smooth $2d$-dimensional affine variety and let $\gamma \in H^{2,(2,2)}(X,\CC)$.  

We say that the \defn{curious Lefschetz property} holds for the pair $(X,\gamma)$, if $H^{\ast}(X)$ is of mixed Tate type (meaning $H^{r,(p,q)}(X)=0$ for $p \neq q$) and if, for all $s \geq 0$ and all $p \leq d$, the map 
\[ \gamma^{d -p} : H^{p+s, (p,p)}(X, \CC) \to H^{2d-p+s, (2d-p, 2d-p)}(X, \CC) \]
is an isomorphism.

\subsection{The curious Lefschetz property and splitting of Hodge structure}

We begin with some basic linear algebra observations.
Let $V^0$, $V^1$, $V^2$, \dots, $V^{2d}$ be a sequence of vector spaces with maps $\gamma: V^i \to V^{i+2}$. Suppose that, for all $k \leq d$, the map $\gamma^{d-k} : V^k \to V^{2d-k}$ is an isomorphism. For $0 \leq i \leq \ell \leq d$, let $W(\ell,i)$ be the subspace of $V^{d-\ell+2i}$ defined by
\[ W(\ell, i) := \gamma^i V^{d-\ell} \cap \mathrm{Ker} \left(\gamma^{\ell-i+1} : V^{d-\ell+2i} \to V^{d+\ell+2} \right) .\]
We leave the following easy lemma to the reader:

\begin{lemma}
The maps $W(\ell,0) \overset{\gamma}{\longrightarrow} W(\ell,1) \overset{\gamma}{\longrightarrow} W(\ell,2) \overset{\gamma}{\longrightarrow} \cdots \overset{\gamma}{\longrightarrow} W(\ell,\ell)$ are isomorphisms, and $\gamma W(\ell,\ell)=0$. The vector space $V^{k}$ decomposes as the direct sum 
\[ V^k = \bigoplus_{\substack{ d-\ell+2i=k \\ 0 \leq i \leq \ell \leq d}} W(\ell,i). \]
\end{lemma}

The reader familiar with the ordinary Lefschetz decomposition of cohomology should recognize $W(\ell,0)$ as the ``primitive cohomology"; the reader familiar with the representation theory of $SL_2$ should recognize $W(\ell,0)$ as the analogue of the highest weight vectors.

Suppose $(X, \gamma)$ has the curious Lefschetz property. Fix a positive integer $s$. We can apply this analysis to the vector spaces $H^{p+s, (p, p)}(X)$, for $0 \leq p \leq 2d$. (Only the spaces with $s \leq p \leq 2d-s$ can actually be nonzero, due to Theorem~\ref{T:weights}.) 

Temporarily write $H(s,p)$ for $H^{s+p, (p,p)}(X)$.  We make the corresponding definition
\begin{equation} W(s, \ell, i) := \gamma^i H(s,d-\ell) \cap \mathrm{Ker} \left(\gamma^{\ell-i+1} : H(s,d-\ell+2i) \to H(s,d+\ell+2) \right) .\label{primitiveH}
\end{equation}

The Deligne summands $H^{s+p, (p,p)}(X)$ are complex subspaces of the complex vector space $H^{s+p}(X, \CC)$. However, the curious Lefschetz property provides us with far more rigidity.
\begin{lem}\label{lem:Ws}
Suppose that $(X, \gamma)$ has the curious Lefschetz property. Suppose also that a nonzero scalar multiple of $\gamma$ is in  $H^2(X, \QQ)$. Then all of the $W(s, \ell, i)$ have bases in $H^\ast(X, \QQ)$.
\end{lem}

(We ask that  a multiple of $\gamma$ be rational, rather than requiring that $\gamma$ itself be, because it is more convenient not to build a factor of $(2 \pi i)^2$ into the definition of the GSV form.)

\begin{proof}
It is enough to prove the claim for $W(s,\ell,0)$, as $W(s, \ell,i) = \gamma^i W(s,\ell,0)$.
Note that $W(s, \ell,0) \subseteq H^{s+d-\ell}(X, \CC)$; we abbreviate $j:=s+d-\ell$.

Let $K$ be the kernel of $\gamma^{\ell+1} : H^{j}(X,\CC) \to H^{j+2\ell + 2}(X, \CC)$. Then $K = \CC \otimes \mathrm{Ker}(\gamma^{\ell+1} : H^{j}(X,\QQ) \to H^{j+2 \ell + 2}(X, \QQ))$, so $K$ has a basis in $H^{\ast}(X, \QQ)$.
Computation also shows that
\[ K = \bigoplus_{\substack{ s' \leq s \\ \ell-s+s' \leq  \ell' \leq \ell+s-s' \\ \ell'-\ell \equiv s'-s \bmod 2}} W\left(s', \ell', \tfrac{\ell'-\ell+s-s'}{2} \right) . \] 
So $K \cap \bigoplus_{r \geq s} H^{j, (j-r,j-r)}(X, \CC) = W(s, \ell, 0)$.  But by Lemma \ref{L:bigrading}(1), $\bigoplus_{r \geq s} H^{j, (j-r,j-r)}(X, \CC) $ is one of the pieces of the weight filtration on $H^j$, so it also has a basis in $H^{\ast}(X, \QQ)$, and thus their intersection has a basis in $H^{\ast}(X, \QQ)$.
\end{proof}

As a corollary, we obtain
\begin{theorem}\label{thm:splitQ}
Suppose that $(X, \gamma)$ has the curious Lefschetz property. Suppose also that a nonzero scalar multiple of $\gamma$ is in  $H^2(X, \QQ)$. Then each Deligne summand $H^{p+s, (p,p)}(X)$ has a basis in $H^{\ast}(X, \QQ)$. In other words, the mixed Hodge structure of $H^\ast(X)$ is split over $\QQ$.
\end{theorem}

\begin{proof}
The result follows from Lemma \ref{lem:Ws} and the equality
\begin{align*} H^{p+s, (p,p)}(X) &= \bigoplus_{d-\ell+2i=p} W(s, \ell,i). 
\hfill \qedhere
\end{align*}
\end{proof}

\subsection{The curious Lefschetz property in products}
%
%

The following lemma has been proved several times in the context of the more standard hard Lefschetz theorem.
Either of the proofs of \cite[Theorem 2.2]{McDaniel} or \cite[Proposition~3.4]{HW} applies in our setting, when combined with the fact that the K\"unneth isomorphism respects the Deligne splitting.

\begin{proposition}\label{prop:curiousproduct}
Suppose that the curious Lefschetz property holds for $(X,\gamma)$ and $(X',\gamma')$.  Then it holds for $(X \times X', \pi_1^*(\gamma) + \pi_2^*(\gamma'))$ where $\pi_1: X \times X' \to X$ and $\pi_2: X \times X' \to X'$ denote the two projections.
\end{proposition}

\subsection{The curious Lefschetz property in Mayer-Vietoris sequences}

The aim of this section is to prove the following theorem. 

\begin{theorem} \label{CuriousMV}
Let $X$ be a smooth $2d$-dimensional affine variety and let $\gamma \in H^{2,(2,2)}(X,\CC)$.  Let $U$ and $V$ be open affine subvarieties of $X$ (so $U \cap V$ is also affine). Suppose that $(U, \gamma)$, $(V, \gamma)$ and $(U \cap V, \gamma)$ all have the curious Lefschetz property. Then $(X, \gamma)$ has the curious Lefschetz property.
\end{theorem}

Our proof will be through the Mayer-Vietoris sequence 
\[ 0 \to H^0(X) \to H^0(U) \oplus H^0(V) \to H^0(U \cap V) \overset{\delta}{\longrightarrow} H^1(X) \to \cdots \]

The maps preserving cohomological degree are ordinary restriction in cohomology, and hence preserve the Deligne splitting and commute with cup product. 
The boundary map $\delta$ also preserves the Deligne splitting (Proposition \ref{prop:splitMV}), sending $H^{r,(p,q)}(U \cap V)$ to $H^{r+1, (p,q)}(X)$. 
Our next lemma discusses the relationship of $\delta$ to the cup product.

\begin{lemma} \label{MayerMult}
Let $X$ be a manifold covered with open sets $U$ and $V$. Let $\delta$ denote the boundary map in the corresponding Mayer-Vietoris sequence. Let $\alpha \in H^s(X)$ and $\beta \in H^t(U \cap V)$. Then $\alpha \cup \delta(\beta) = (-1)^s \delta(\alpha \cup \beta)$, as classes in $H^{s+t+1}(X)$.
\end{lemma}

In applications of this lemma, $s$ will always be $2$, so the $(-1)^s$ term will drop out.

\begin{proof}
It is easiest to verify this in the deRham model for cohomology.  Let $\sigma_U+\sigma_V=1$ be a partition of unity subordinate to the cover $(U, V)$. Represent $\beta$ by a closed $t$-form $\tilde{\beta}$. So $\sigma_U \tilde{\beta}$ is a $t$-form on all of $X$. Up to sign conventions, $\delta(\beta)$ is represented by the form $d(\sigma_U \tilde{\beta})$. 

Lift $\alpha$ to a closed $s$-form $\tilde{\alpha}$. Then $\alpha \cup \delta(\beta)$ is represented by $\tilde{\alpha} \wedge d(\sigma_U \tilde{\beta})$ and $\delta(\alpha \cup \beta)$ is represented by $d(\tilde{\alpha} \wedge \sigma_U \tilde{\beta})$. Since $\tilde{\alpha}$ is closed, $d(\tilde{\alpha} \wedge \sigma_U \tilde{\beta}) = (-1)^s \tilde{\alpha} \wedge d(\sigma_U \tilde{\beta})$.
\end{proof}

\begin{proof}[Proof of Theorem~\ref{CuriousMV}]
First, we check that the mixed Hodge structure on $X$ is of mixed Tate type. Let $p \neq q$. Since the maps of the Mayer-Vietoris sequence respect the Deligne splitting, we have an exact sequence $H^{r-1,(p,q)}(U \cap V) \to H^{r,(p,q)}(X) \to H^{r, (p,q)}(U) \oplus H^{r, (p,q)}(V)$. The outside terms are assumed to be $0$, so the middle term is as well.

Now, consider the diagram
\[ \xymatrix@C=40 pt{
{\begin{matrix}H^{r-1,(p,p)}(U)  \\ \oplus  H^{r-1,(p,p)}(V) \end{matrix}} \ar[d] \ar[r]^-{\gamma^{d-p} \oplus \gamma^{d-p}} &  {\begin{matrix} H^{2d-2p+r-1,(2d-p,2d-p)}(U) \\  \oplus  H^{2d-2p+r-1,(2d-p,2d-p)}(V) \end{matrix}}  \ar[d] \\ 
H^{r-1,(p,p)}(U \cap V) \ar[d]^{\delta} \ar[r]^-{\gamma^{d-p}}& H^{2d-2p+r-1,(2d-p,2d-p)}(U \cap V)  \ar[d]^{\delta} \\ 
H^{r,(p,p)}(X) \ar[d] \ar[r]^-{\gamma^{d-p}}& H^{2d-2p+r,(2d-p,2d-p)}(X)  \ar[d] \\ 
 {\begin{matrix}  H^{r,(p,p)}(U)  \\ \oplus  H^{r,(p,p)}(V) \end{matrix} }\ar[d] \ar[r]^-{\gamma^{d-p} \oplus \gamma^{d-p}}  &  { \begin{matrix} H^{2d-2p+r,(2d-p,2d-p)}(U)  \\ \oplus  H^{2d-2p+r,(2d-p,2d-p)}(V) \end{matrix}} \ar[d] \\ 
 H^{r,(p,p)}(U \cap V)  \ar[r]^-{\gamma^{d-p}}&  H^{2d-2p+r,(2d-p,2d-p)}(U \cap V)\\
}\]

The diagram commutes, using Lemma~\ref{MayerMult} for the square that include $\delta$ and the fact that restriction of cohomology classes is a map of rings for the others.
Our induction hypothesis is that the first, second, fourth and fifth rows are isomorphisms. So the five lemma shows that the third row is an isomorphism, as desired.
\end{proof}

\section{Cluster algebras}
\label{sec:cluster background}
\subsection{Seeds}
Let $\cF$ be a rational function field over $\CC$.  An extended exchange matrix is a $(n+m) \times n$ matrix $\tB = (\tB_{ij})$ such that the principal part, the top $n \times n$ square submatrix $B$, is skew-symmetric. (We anticipate no difficulty in extending our results to the skew-symmetrizable case, but restrict to the skew-symmetric case for convenience.)  A seed $t = (\x,\tB)$ consists of $n+m$ algebraically independent elements $x_1$, $x_2$, \ldots, $x_{n+m} \in \cF$ and an extended exchange matrix $\tB$.  
We say that $(\x,\tB)$ has rank $n$.

Mutation of seeds is defined as follows: given an index $k \in \{1,2,\ldots,n\}$, one defines a new extended exchange matrix $\tB' = \mu_k(\tB)$, the mutation of $\tB$ in the direction $k$, by the formula 
\[ \tB'_{ij} = \begin{cases} - \tB_{ij} & \mbox{if $i=k$ or $j=k$,}\\ 
\tB_{ij} + [\tB_{ik}]_{+} [\tB_{kj}]_{+} - [\tB_{ik}]_{-} [\tB_{kj}]_{-} & \mbox{otherwise,} \\ \end{cases} \]
where $[x]_+ = \max(x,0)$ and $[x]_- = \min(x,0)$.
One creates a new seed $t' = (\x', \tB') = \mu_k(t)$ by the formulae
\begin{align}\label{eq:mutation}
x'_k &= \frac{\prod_{i} x_i^{[\tB_{ik}]_+} \ + \ \prod_i x_i^{[-\tB_{ik}]_+}}{x_k} \\ 
x'_i &= x_i  & \mbox{if $i \neq k$.} 
\end{align}
%
We continue mutating on all possible indices, producing new seeds.
The $(n+m)$-tuples $(x_1, \ldots, x_{n+m})$ produced in this manner are called clusters, and the individual $x_i$ are called cluster variables.  The cluster variables $x_1,x_2 \ldots,x_n$ are called \defn{mutable}.  The cluster variables $x_{n+1}, x_{n+2}, \dots, x_{n+m}$ are the same in every cluster and are called \defn{frozen}; we shall often denote them by $y_1 = x_{n+1}, y_2 = x_{n+2}, \ldots, y_m = x_{n+m}$ as well.  
The sub-$\CC$-algebra of $\cF$ generated by all the cluster variables, and the reciprocals of the frozen variables, is the cluster algebra $A(\x,\tB)$, or simply $A(\tB)$ or $A$.  We write $\cA$ or $\cA(\tB)$ or $\cA(\x, \tB)$ for the cluster variety $\Spec A$.  We say that $A(\tB)$ has \defn{full rank} if the matrix $\tB$ has full rank (that is, rank $n$).  It is known \cite[Lemma 3.2]{CA3} that the rank of the matrix $\tB$ is mutation-invariant; see Corollary \ref{cor:rank}.  We say that $\tB$ has \defn{principal coefficients} if $\tB =\left( \begin{smallmatrix} B \\ \mathrm{Id}_n \end{smallmatrix} \right)$ and $d$-principal coefficients if $\tB =\left( \begin{smallmatrix} B \\ d\,\mathrm{Id}_n \end{smallmatrix} \right)$.
\begin{Remark} Cluster algebras are naturally defined over $\ZZ$.  In Section~\ref{sec:pointcounts}, we will work with ground fields other than $\CC$. \end{Remark}

For any cluster $(x_1, \ldots, x_n, x_{n+1}, \ldots, x_{n+m})$, we have the Laurent phenomenon \cite{CA1}: the cluster algebra $A$ is contained in the Laurent polynomial ring $\CC[x_1^{\pm 1}, \ldots, x_{n+m}^{\pm 1}]$.  The open subset $\Spec(\CC[x_1^{\pm 1}, \ldots, x_{n+m}^{\pm 1}]) \subset \cA$ is called a \defn{cluster torus} of $\cA$.

The upper cluster algebra $U \subset \cF$ is defined to be the intersection $\bigcap \CC[x_1^{\pm 1}, \ldots, x_{n+m}^{\pm }]$ of Laurent polynomial rings where the intersection is taken over all seeds.  Thus the Laurent phenomenon states that $A \subseteq U$.

%

\subsection{Separating edges and the Louise algorithm}
Let $t = (\x,\tB)$ be a seed.  Let $\oG(t) = \oG(\tB)$ be the directed graph with vertices $\{1,2,\ldots,n\}$ and a directed edge $i \to j$ whenever $\tB_{ij} > 0$.   We also let $\Gamma(t)$ denote the underlying undirected graph of $\oG$.
We say that a seed $(\x,\tB)$ is \defn{acyclic} if $\oG(t)$ has no oriented cycles; see \cite{CA3}.  We say that $\cA$ is acyclic if some seed of $\cA$ is acyclic.

For a subset $S \subset [n]$, let $\tB_S$ denote the $|S| \times (n + m)$ matrix consisting of the columns of $\tB$ indexed by $S$.  We say that the cluster algebra $A(\tB_S)$ is obtained from $A(\tB)$ by freezing the cluster variables $\{x_i \mid i \in [n] \setminus S\}$.  If, in addition, $A(\tB_S) \cong A[x^{-1}_i \mid i \in [n] \setminus S]$ then we call $A(\tB_S)$ a cluster localization, and the subvariety $\cA(\tB_S) \subset \cA(\tB)$ is called a cluster chart of $\cA$.  We say that a cluster algebra $A(\tB)$ is \defn{locally acyclic} if $\cA$ can be covered by finitely many acyclic cluster charts; see \cite{Mul}.

We define an edge $e = i \to j$ of $\oG$ to be a \defn{separating edge} if there is no bi-infinite path through $e$.  Edges belonging to directed cycles cannot be separating, but this is not a sufficient condition to be a separating edge.
\begin{proposition}[\cite{Mul}]\label{prop:covering}
Suppose $i \to j$ is a separating edge.  Then $x_i$ and $x_j$ are not simultaneously zero.  In other words, $\cA$ is covered by the two open subsets $\Spec A[x_i^{-1}]$ and $\Spec A[x_j^{-1}]$.
\end{proposition}

We define the Louise property for cluster algebras recursively.  We say that a cluster algebra $A $ satisfies the \defnintro{Louise property} if either 
\begin{enumerate}
\item for some seed $t$ of $A$, the quiver $\oG(t)$ has no edges, or  
\item for some seed $t = (\x,\tB)$ of $A$, the quiver $\oG(t)$ has a separating edge $i \to j$, such that the cluster algebras $A(\tB_{\oG \setminus \{i\}}), A(\tB_{\oG \setminus \{j\}})$, and $A(\tB_{\oG \setminus \{i,j\}})$ all satisfy the Louise property.
\end{enumerate}
Suppose that $A(\tB)$ satisfies the Louise property.  Then it follows from the definition that all the cluster algebras encountered in this recursion also satisfy the Louise property.  We remark that in the base case (1) of the recursion, we can instead require the weaker property that $\oG(t)$ is acyclic.  Indeed, every acyclic cluster algebra satisfies the Louise property.  As remarked in Theorem \ref{thm:surfaces}, cluster algebras defined from marked surfaces are usually locally acyclic.

\begin{example}
Let 
\[
\tB = \begin{pmatrix} 0 &  -1 & -1 & -1 \\
1 & 0 & 1 & -1 \\
1 & -1 & 0 & 1\\
1 & 1 & -1 & 0 \end{pmatrix} \qquad \oG(\tB) =  \begin{tikzpicture}[every node/.style={draw=blue,thick,circle,inner sep=0pt},scale = 0.6,baseline={([yshift=-.5ex]current bounding box.center)}]
\node[draw,circle] (v1) at (0,0) {$1$};
\node[draw,circle] (v2) at (120:2) {$2$};
\node[draw,circle] (v3) at (240:2) {$3$};
\node[draw,circle] (v4) at (0:2) {$4$};
\draw[->](v2) -- (v1);
\draw[->](v2) -- (v3);
\draw[->](v3) -- (v1);
\draw[->](v3) -- (v4);
\draw[->](v4) -- (v2);
\draw[->](v4) -- (v1);
\end{tikzpicture}
\]
We will check that $\tB$ satisfies the Louise property.   The edge $2 \to 1$ is a separating edge.  The exchange matrices $\tB_{[4] \setminus \{1\}}$, $\tB_{[4] \setminus \{2\}}$, $\tB_{[4] \setminus \{1,2\}}$ have the following three quivers respectively:
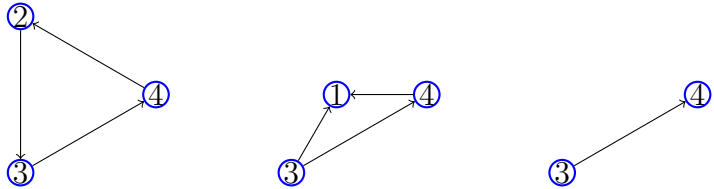
\[
\begin{tikzpicture}[every node/.style={draw=blue,thick,circle,inner sep=0pt},scale = 0.6,baseline={([yshift=-.5ex]current bounding box.center)}]
\node[draw,circle] (v2) at (120:2) {$2$};
\node[draw,circle] (v3) at (240:2) {$3$};
\node[draw,circle] (v4) at (0:2) {$4$};
\draw[->](v2) -- (v3);
\draw[->](v3) -- (v4);
\draw[->](v4) -- (v2);
\begin{scope}[shift={(6,0)}]
\node[draw,circle] (v1) at (0,0) {$1$};
\node[draw,circle] (v3) at (240:2) {$3$};
\node[draw,circle] (v4) at (0:2) {$4$};
\draw[->](v3) -- (v1);
\draw[->](v3) -- (v4);
\draw[->](v4) -- (v1);
\end{scope}
\begin{scope}[shift={(12,0)}]
\node[draw,circle] (v3) at (240:2) {$3$};
\node[draw,circle] (v4) at (0:2) {$4$};
\draw[->](v3) -- (v4);
\end{scope}
\end{tikzpicture}
\]
The right two quivers are acyclic and thus Louise.  The left quiver has no separating edge, but after mutating at 2, we obtain the quiver:
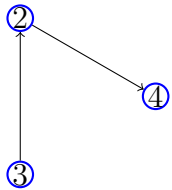
\[
\begin{tikzpicture}[every node/.style={draw=blue,thick,circle,inner sep=0pt},scale = 0.6,baseline={([yshift=-.5ex]current bounding box.center)}]
\node[draw,circle] (v2) at (120:2) {$2$};
\node[draw,circle] (v3) at (240:2) {$3$};
\node[draw,circle] (v4) at (0:2) {$4$};
\draw[->](v3) -- (v2);
\draw[->](v2) -- (v4);
\end{tikzpicture}
\]
which is acyclic and thus Louise.  It follows that the original exchange matrix $\tB$ is Louise.
\end{example}

If the Louise property holds for $A$, then $A$ is locally acyclic.  This follows from Proposition~\ref{prop:covering}, the recursive definition of Louise, and the following result; see also \cite{MS}.

\begin{prop}[{\cite[Proposition 3.1 and Lemma 3.4]{Mul}}] \label{prop:localization} Let $A(\tB)$ be a cluster algebra. Let $S \subset [n]$.  Suppose that the freezing $A(\tB_S)$ is locally acyclic.  Then $A(\tB_S) \cong A[x^{-1}_i \mid i \in [n] \setminus S]$.
\end{prop}

Locally acyclic cluster algebras enjoy a number of favorable properties.

\begin{thm}[{\cite[Theorems 4.1, 4.2, and 7.7]{Mul}}]\label{thm:Muller}  Suppose that $A$ is locally acyclic.  Then $A$ is finitely generated over $\CC$, and is equal to the upper cluster algebra $U$.  If, in addition, $A$ is of full rank, then $\cA$ is smooth. \end{thm}

%

\begin{remark}\label{rem:fullrank}
Suppose that $\tB$ is of full rank, and $\tB_S$ is the modified exchange matrix where we delete some of the columns of $\tB$.  Then $\tB'$ is also of full rank.  This will be useful in our inductive arguments.
\end{remark}

\section{Automorphisms and quotients of cluster varieties}\label{sec:autom}
\def\Aut{\mathrm{Aut}}
\def\mult{\mathrm{mult}}

\subsection{Automorphisms of cluster algebras}

Let $A$ be a cluster algebra. Define $\Aut(A)$ to be the group of algebra automorphisms $\phi$ of $A$ such that, for each cluster variable $x$, we have $\phi(x) = \zeta(x) x$ for some $\zeta (x) \in \CC^{\ast}$. 

For any $k \times \ell$ integer matrix $M$, we write $\mult(M)$ for the map $(\CC^{\ast})^{\ell} \to (\CC^{\ast})^{k}$ given by $\mult(M)(\zeta_1, \ldots, \zeta_{\ell})_i = \prod_{j=1}^{\ell} \zeta_j^{M_{ij}}$, so $\mult(M_1 M_2) = \mult(M_1) \circ \mult(M_2)$. 

\begin{prop} \label{cluster aut}
Let $(\x, \tB)$ be any seed of $A$. An element of $\Aut(A)$ is determined by its action on $(x_1, \ldots, x_{n+m})$. Given $(\zeta_1, \ldots, \zeta_{n+m}) \in (\CC^{\ast})^{n+m}$, there is an automorphism $\phi$ with $\phi(x_i) = \zeta_i x_i$ if and only if $\zeta \in \Ker(\mult(\tilde{B}^T))$. Thus, $\Aut(A) \cong \mathrm{Hom}(\ZZ^{n+m}/\tilde{B} \ZZ^n, \CC^{\ast})$.
\end{prop}

\begin{proof}
The formulation $\zeta \in \Ker(\mult(\tB^T))$ is essentially \cite[Lemma 2.3]{GSV}. Rewriting $ \Ker(\mult(\tB^T))$ as $\mathrm{Hom}(\ZZ^{n+m}/\tB \ZZ^n, \CC^{\ast})$ is elementary.
\end{proof}

\begin{cor} \label{AutCharacter}
The character group $\Hom(\Aut(A), \CC^{\ast})$ is naturally isomorphic to $\ZZ^{n+m}/\tB \ZZ^n$.
\end{cor}

\begin{proof}
We have $\Hom(\Hom(\ZZ^{n+m}/\tB \ZZ^n, \CC^{\ast}),\CC^{\ast}) \cong \ZZ^{n+m}/\tB \ZZ^n$, since $\ZZ^{n+m}/\tB \ZZ^n$ is a finitely generated abelian group.  
\end{proof}

It follows that the isomorphism class of the abelian group $\ZZ^{n+m}/\tB \ZZ^n$ is a mutation invariant.  Corollary~\ref{AutCharacter} describes the character group of $\Aut(A)$; its torsion part is the group of characters which are constant on the connected components of $\Aut(A)$. We record the result:

\begin{cor} \label{LocConstCharacter}
The group of locally constant characters of $\Aut(A)$ is naturally isomorphic to $\tB \QQ^n \cap \ZZ^{n+m} / \tB \ZZ^n$.
\end{cor}
%


\begin{cor}\label{cor:rank}
If $\tilde{B}$ has rank $n$, then so do all mutations of $\tilde{B}$. Similarly, if the rows of $\tilde{B}$ span $\ZZ^n$ over $\ZZ$, then so do all mutations of $\tilde{B}$.
\end{cor}

\begin{proof}
The former is equivalent to requiring that $\QQ \otimes \ZZ^{n+m}/\tilde{B} \ZZ^n \cong \QQ^m$.  The latter is equivalent to requiring that $\ZZ^{n+m} / \tilde{B} \ZZ^n \cong \ZZ^m$.
\end{proof}

We will say that $\tilde{B}$ has \defn{full rank} if $\tilde{B}$ has rank $n$, and has \defn{really full rank} if $\tilde{B}^T \ZZ^{n+m} = \ZZ^n$. 
By Corollary \ref{cor:rank}, both conditions are mutation invariants, so we also apply the terms full rank and really full rank to $A$ and $\cA$.

\subsection{Covering maps of cluster varieties}
Let $\cA_1$ and $\cA_2$ be cluster varieties with the same rank and extended exchange matrices of the same size.  A \defn{cluster morphism} from $\cA_2$ to $\cA_1$ is a triple $(\Psi, \Phi, \{R_t\})$, where
\begin{enumerate}
\item
$\Psi: \cA_2 \to \cA_1$ is a morphism of algebraic varieties, induced by a ring homomorphism $\Psi^*: A_1 \to A_2$;
\item $\Phi$ is a map taking seeds of $A_1$ to seeds of $A_2$, commuting with mutation.  More precisely, for $k = 1,2,\ldots,n$ we have $\mu_k(\Phi(t)) = \Phi(\mu_k(t))$ for all seeds $t$ of $A_1$;
\item $\{R_t\}$ is a collection of $(n\times m) \times (n\times m)$ integer matrices, with block diagonal form $R_t = \left( \begin{smallmatrix} \mathrm{Id}_n & 0 \\ P & Q \end{smallmatrix} \right)$, indexed by seeds $t = (\z,\tB_1)$ of $A_1$, satisfying the following condition: if $\Phi(t) = (\x,\tB_2)$ then
\[\begin{array}{lcl}
\tB_2 &=& R_t\tB_1 \ \mbox{and} \\
 \Psi^*(z_j) &=& \prod_{i=1}^{n+m} x_i^{R_{ij}} = \begin{cases} x_j \prod_{i=n+1}^{n+m} x_i^{R_{ij}} & \mbox{if $j \in \{1,2,\ldots,n\}$} \\
\prod_{i=n+1}^{n+m} x_i^{R_{ij}} & \mbox{if $j \in \{n+1,n+2\ldots,n+m\}$.}
\end{cases} \end{array}
\]
In other words, $\Psi$ sends the cluster torus of $\Phi(t)$ to the cluster torus of $t$ via $\mult(R^T)$.
\end{enumerate}

Our notion of cluster morphism is related to Fraser's notion of a quasi-homomorphism of cluster algebras which is defined in a more general setting; see \cite[Corollary 4.5]{Fra}. 

The following statement is immediate from (3) of the definition.
\begin{lemma}
A cluster morphism $(\Psi, \Phi, \{R_t\}): \cA_2 \to \cA_1$ exists only if there are seeds $(\z,\tB_1)$ of $\cA_1$ and $(\x,\tB_2)$ of $\cA_2$ such that the principal parts of $\tB_1$ and $\tB_2$ are the same.
\end{lemma}

Since frozen variables do not mutate and the map $\Psi^*$ does not depend on choice of seed, the formulae $z_j = \prod_{i=1}^{n+m} x_i^{R_{ij}}$ depends only on the cluster variable $z_j$ and not on the seed $t$.

Suppose that $R$ is a $(n+m) \times (n+m)$ integer matrix with block diagonal form $R = \left( \begin{smallmatrix} \mathrm{Id}_n & 0 \\ P & Q \end{smallmatrix} \right)$, and $\tB$ is an extended exchange matrix.  Let $\tC = R\tB$.  For $k = 1,2,\ldots,n$, define $R' =\mu_k(R) = \mu_k^{\tB}(R)$ by the formulae
\begin{equation} \begin{array}{rcll}
R'_{ik} &=& -R_{ik} -[\tC_{ik}]_+ + \sum_{\ell=1}^{n+m}R_{i\ell}[B_{\ell k}]_+ & i>n \\
&=& -R_{ik} +[\tC_{ik}]_- - \sum_{\ell=1}^{n+m}R_{i\ell}[B_{\ell k}]_- & \\
R'_{ij} &=& R_{ij} & j \neq k \ \mbox{or}\ i\leq n \\
\end{array} . \label{eq:mutateR}
\end{equation}

%
Note that in block form $R' = \left( \begin{smallmatrix} \mathrm{Id}_n & 0 \\ P' & Q \end{smallmatrix} \right)$, so that only the $P$ block changes.

\begin{lemma}\label{lem:Rmutation}
If $\tB' = \mu_k(\tB)$ and $\tC' = \mu_k(\tC)$ then we have $\tC' = R'\tB'$, where $R' = \mu_k^{\tB}(R)$.
\end{lemma}
\begin{proof}
Write $S:= R'\tB'$.  It is immediate from the definitions that $S_{i\ell} = \tC'_{i\ell}$ when $i \in \{1,2,\ldots,n\}$, so we assume that $i > n$.  Suppose first that $\ell = k$.  Then
\[
S_{ik}
 = \sum_{j} R'_{ij}\tB'_{jk} =- \sum_j R_{ij} \tB_{jk} = - \tC_{ik} = \tC'_{ik}
\]
where we have used that $\tB_{kk}=0$ in the second equality.

Suppose now that $\ell \neq k$.  Then
\begin{align*}
S_{i\ell}&=\sum_{j=1}^{n+m}  R'_{ij}\tB'_{j\ell}  
= R'_{ik}\tB'_{k\ell}+\sum_{j \neq k} R_{ij}\tB'{j\ell} \\
&= \left(R_{ik} + [\tC_{ik}]_+ -\sum_{s=1}^{n+m} R_{is}[\tB_{sk}]_+\right)\tB_{k\ell} + \sum_{j \neq k} R_{ij} \left(\tB_{j\ell} + \sgn(\tB_{jk})[\tB_{jk}\tB_{kl}]_+\right)\\
&= \sum_{j=1}^{n+m} R_{ij}\tB_{j\ell} + [\tC_{ik}]_+\tB_{k\ell} + \sum_{j=1}^{n+m} R_{ij}\left(\sgn(\tB_{jk})[\tB_{jk}\tB_{k\ell}]_+ - [\tB_{jk}]_+ \tB_{k\ell} \right)
\end{align*}
Now observe that
\[
\left(\sgn(\tB_{jk})[\tB_{jk}\tB_{k\ell}]_+ - [\tB_{jk}]_+ \tB_{k\ell} \right) = \begin{cases} 0 & \mbox{if $\tB_{k\ell} \geq 0$,} \\
-\tB_{jk}\tB_{k\ell} & \mbox{if $\tB_{k\ell} \leq 0$.}
\end{cases}
\]
Thus 
\[
S_{i\ell} - \tC_{i\ell} = \begin{cases} [\tC_{ik} \tB_{k\ell}]_+ & \mbox{if $\tB_{k\ell} \geq 0$,}\\
 [\tC_{ik}]_+\tB_{k\ell}-\left(\sum_{j=1}^{n+m}R_{ij}\tB_{jk} \right)\tB_{k\ell}  = - [\tC_{ik}B_{k\ell}]_+& \mbox{if $\tB_{k\ell} \leq 0$.}
\end{cases}
\]
In both cases we obtain $S_{i\ell} = \tC'_{i\ell}$.
\end{proof}

\begin{theorem}
A cluster morphism $(\Psi, \Phi, \{R_t\}): \cA_2 \to \cA_1$ is completely determined by $R = R_{t_0}$ for the initial seed $t_0 = (\z,\tB)$ (or any other seed).  If $\Phi(\z,\tB) = (\x,\tC)$ and $\mu_k(t_0) = t_1 = (\z',\tB')$ then $R_{t_1} = \mu_k^{\tB}(R)$.  Conversely, any integer matrix $R$ with block form $\left( \begin{smallmatrix} \mathrm{Id}_n & 0 \\ P & Q \end{smallmatrix} \right)$ induces a cluster morphism.
\end{theorem}
\begin{proof}
Suppose we are given a cluster morphism $(\Psi, \Phi, \{R_t\}): \cA_2 \to \cA_1$.  Substituting $z_j= \prod_{i=1}^{n+m} x_i^{R_{ij}}$ into the exchange relation
\[
z_k z_k' = \prod_{j} z_j^{[\tB_{jk}]_+} \ + \ \prod_j z_j^{[-\tB_{jk}]_+}
\]
gives
\begin{align*}
x_kx'_k&=\prod_{i=1}^n x_i^{[\tB_{ik}]_+} \prod_{i=n+1}^{n+m} x_i^{-r_i + \sum_j R_{ij}[\tB_{jk}]_+} \ + \ \prod_{i=1}^n x_i^{-[\tB_{ik}]_-} \prod_{i=n+1}^{n+m} x_i^{-r_i-
\sum_j R_{ij}[\tB_{jk}]_-}
\end{align*}
where $r_i = R_{ik} + R'_{ik}$.  This is the exchange relation for $(\x,\tC)$ if and only if $R'_{ik}$ is given by \eqref{eq:mutateR}.  Thus all the matrices $R_t$ can be obtained by iterated mutation of $R = R_{t_0}$.

Conversely, given $R$ satisfying $\tC = R\tB$, we can define $R_t$ iteratively via mutation.  That we obtain a cluster morphism follows from the above calculation, together with Lemma \ref{lem:Rmutation}.  Note that $\Psi^*: A_1 \to A_2$, if it exists, uniquely determines all the $R_t$ via the formula $z_j= \prod_{i=1}^{n+m} x_i^{R_{ij}}$.  Thus the exchange relation calculation above implies that the matrices $R_t$ are well-defined via mutation (and do not depend on the mutation path taken from $t_0$ to $t$).
\end{proof}

If an initial seed $t_0$ has been chosen, we may thus denote a cluster morphism by $(\Psi,\Phi,R)$, where $R = R_{t_0}$.

We say that a cluster morphism $(\Psi, \Phi, \{R_t\})$ is a \defn{covering map} of cluster varieties if $\Psi: \cA_2 \to \cA_1$ identifies $\cA_1$ with $\cA_2 \sslash H \cong \Spec(A_2^H)$ for a finite subgroup $H \subset \Aut(A_2)$ which acts freely on $\cA_2$.  We shall only consider covering maps in the case that $\cA_2$ is a smooth complex affine algebraic variety, so the geometric invariant theory quotient $\cA_2 \sslash H$ coincides with the quotient as complex manifolds.


%


\begin{prop}\label{prop:coveringmap}
Suppose that $\cA_1$ and $\cA_2$ are locally acyclic cluster varieties of full rank.
Suppose that $(\Psi, \Phi, \{R_t\}): \cA_2 \to \cA_1$ is a covering map.  Then $R_t$ has full rank for any $t$.  Conversely, if $(\Psi, \Phi, \{R_t\}): \cA_2 \to \cA_1$ is a cluster morphism determined by a full rank integer matrix $R = R_{t_0}$, then it is a covering map of cluster varieties.

The group $H$ and the matrix $R$ are related by 
\[ H = \Ker(\mult(Q^T)) \cong \Ker(\mult(R^T)) \subset \Aut(A_2). \]
\end{prop}
\begin{proof}
The map $\mult(R_t^T)$ sends the cluster torus of the seed $\Phi(t)$ of $\cA_2$ to the cluster torus of the seed $t$ of $\cA_1$.  This map preserves dimension if and only if $R$ has full rank.

Now, suppose that $R = R_{t_0}$ is of full rank.  Since $\tB_2^T = \tB_1^T R^T$, we see that $\Ker(\mult(R^T)) \subseteq \Ker(\mult(\tB_2^T)) = \Aut(A_2)$. Moreover, projection onto the frozen coordinates gives an isomorphism $\Ker(\mult(R^T)) \cong \Ker(\mult(Q^T))$.  Since $R$ has full rank, so does $Q$, and thus $H := \Ker(\mult(Q^T))$ is a finite abelian subgroup of $\Aut(A_2)$.

We first check that the action of $H$ on $\cA_2$ is free.  Let $h = (\zeta_1, \ldots, \zeta_{n+m})$ be a non-identity element of $\Ker(\mult(R^T))$. We must show that $h$ does not fix any point of $\cA_2$.
The block form  $R^T = \left( \begin{smallmatrix} \mathrm{Id}_n & P^T \\ 0 & Q^T \end{smallmatrix} \right)$ shows that $(\zeta_{n+1}, \ldots, \zeta_{n+m}) \neq (1,1,1,\ldots, 1)$. 
Choose a coordinate $i$ for which $\zeta_{n+i} \neq 1$.  Let $p$ be any point of $\cA_2$. The frozen variables $(x_{n+1}, \ldots, x_{n+m})$ are defined and nonzero everywhere on $\cA_2$, and we have $x_{n+i}(h(p)) = \zeta_{n+i} x_{n+i}(p)$. In particular, $x_{n+i}(h(p)) \neq x_{n+i}(p)$. So $h(p) \neq p$.

We now check that $\Psi: \cA_2 \to \cA_1$ is the quotient map by $H$.  It follows from the formula $\Psi^*(z_j) = \prod_{i=1}^{n+m} x_i^{R_{ij}}$ that $\Psi^*(A_1) \subseteq A_2^H$, and we have that $\Psi^*: A_1 \to A_2$ is an injection since $\Frac(A_1) \hookrightarrow \Frac(A_2)$ is an injection.  It remains to show that any element $f \in A_2^H$ is in the image of $\Psi^*$.

For a seed $t$ of $A_2$, we may write $f$ as a Laurent polynomial $f'/M$ in the cluster variables of $t$, and we may further assume that both $f'$ and $M$ are $H$-invariant.  Thus $f = \Psi^*(g)$, where $g \in \Frac(A_1)$ is a Laurent polynomial in the cluster variables of the seed $\Phi(t)$ of $A_1$.  Such a Laurent polynomial expression for $g$ holds for all seeds of $A_1$, and so $g$ lies in the upper cluster algebra of $A_1$.  By Theorem \ref{thm:Muller} and the assumption that $A_1$ is locally acyclic, we have $g \in A_1$. 
%
\end{proof}

\subsection{Constructing covering maps}
\begin{lemma} \label{lem:smith}
Let $\tB$ be an $(n+m) \times n$ extended exchange matrix of full rank, where $m \geq n$. Let $d$ be large enough that $d \ZZ^n \subseteq \tB^T \ZZ^{n+m}$. Then there exists a full rank integer matrix $R = \left( \begin{smallmatrix} \mathrm{Id}_n & 0 \\ P & Q \end{smallmatrix} \right)$ such that
$
R \tB = \left( \begin{smallmatrix} B \\ d\,\Id_n \\ 0 \end{smallmatrix} \right).
$
\end{lemma}

\begin{proof}
Write $\tB$ in Smith normal form: $\tB = U D V$ where $U$ and $V$ are in $\SL_{n+m}(\ZZ)$ and $\SL_{n}(\ZZ)$, and $D = \left( \begin{smallmatrix} \diag(d_1, d_2, \ldots, d_n) \\ 0 \end{smallmatrix} \right)$. (It is usual to further require that $d_1 | d_2 | \cdots | d_n$, but we will not need this.) Since $\tB$ has full rank, the $d_i$ are all nonzero. The condition $d \ZZ^n \subseteq V^T D^T U^T \ZZ^{n+m}$ is equivalent to $d \ZZ^n \subseteq D \ZZ^{n+m}$, since $(V^T)^{-1} \ZZ^n = \ZZ^n$ and $U^T \ZZ^{n+m} = \ZZ^{n+m}$. In other words, all of the $d_i$ divide $d$.

Let $U ' = \begin{pmatrix}V^{-1} & 0 \\ 0 & \Id_{m} \end{pmatrix}\diag(d/d_1,d/d_2,\ldots,d/d_n,1,1,\ldots,1) U^{-1}$, where the diagonal matrix $ \diag(d/d_1,d/d_2,\ldots,d/d_n,1,1,\ldots,1)$ has $m$ trailing $1$'s.  Then $U'$ is a full rank matrix.  Define a $m \times (n+m)$ matrix $(P \; Q)$ where $P$ is $m \times n$ and $Q$ is $m \times m$ by
\[
\begin{pmatrix}P & Q \end{pmatrix} = \begin{pmatrix}\Id_{n} & \ast_{n \times m} \\ 0 & \ast_{m \times m}  \end{pmatrix} U',
\]
where the $\ast$'s are chosen so that $Q$ has full rank.  (To see that this is possible, let $v_1,\ldots, v_{n+m}$ be the rows of the last $m$ columns of $U'$, and let $S = {\rm span}(v_{n+1},\ldots,v_{n+m})$.  Note that it is always possible to pick a basis $\beta_1,\ldots,\beta_m$ of ${\rm span}(v_{1},\ldots,v_{n+m})$ such that $\beta_1 - v_1, \ldots, \beta_n - v_n, \beta_{n+1},\ldots,\beta_m \in S$.)

We then calculate that
\begin{align*}
\begin{pmatrix}P & Q \end{pmatrix}  \tB &= \begin{pmatrix}\Id_{n} & \ast_{n \times m} \\ 0 & \ast_{m \times m}  \end{pmatrix} \begin{pmatrix}V^{-1} & 0 \\ 0 & \Id_{m} \end{pmatrix}\diag(d/d_1,d/d_2,\ldots,d/d_n,1,1,\ldots,1) D V\\
&= \begin{pmatrix}\Id_{n} & \ast_{n \times m} \\ 0 & \ast_{m \times m}  \end{pmatrix} \begin{pmatrix}V^{-1} & 0 \\ 0 & \Id_{m} \end{pmatrix}\begin{pmatrix}d \Id_n \\ 0_{m \times m} \end{pmatrix} V \\
&=\begin{pmatrix}\Id_{n} & \ast_{n \times m} \\ 0 & \ast_{m \times m}  \end{pmatrix} \begin{pmatrix}d \Id_n \\ 0_{m \times m}\end{pmatrix} \\
&= \begin{pmatrix}d \Id_n \\ 0_{(m -n) \times (m-n)}\end{pmatrix}
\end{align*}
%
%
%
%
Thus $R = \left(\begin{smallmatrix} \Id_{n \times n} & 0 \\ P & Q \end{smallmatrix}\right)$ has the required property.
\end{proof}

\begin{prop}\label{prop:coverprincipal}
Suppose $\cA(\tB)$ is a locally acyclic cluster algebra of full rank, and $\tB$ is $(n+m) \times n$ where $m \geq n$.  Then there is a covering map $\cA(\tilde{B}') \to \cA(\tB)$ of cluster algebras, where $\tB'$ has the form
$
\tB' = \left( \begin{smallmatrix} B \\ d\,\Id_n \\ 0 \end{smallmatrix} \right).
$
\end{prop}
\begin{proof}
Follows immediately from Lemma \ref{lem:smith} and Proposition \ref{prop:coveringmap}.
\end{proof}

\subsection{Isomorphic cluster varieties}
In some cases, a covering map of cluster varieties is in fact an isomorphism.
\begin{prop} \label{prop:rowspanenough}
Let $\tB$ and $\tB'$ be $(n+m) \times n$ and $(n+m') \times n$ exchange matrices which have the same top $n$ rows and whose rows have the same integer span. Let $\cA = \cA(\tB)$ and $\cA' = \cA(\tB')$ be the corresponding cluster varieties. Then $\cA \times (\CC^{\ast})^{m'} \cong \cA' \times (\CC^{\ast})^m$.  If $m = m'$, then we have natural isomorphisms $H^{k,(p,q)}(\cA) \cong H^{k, (p,q)}(\cA')$.
\end{prop}
\begin{proof}
Let $\tB = \left( \begin{smallmatrix} B \\ C \end{smallmatrix} \right)$ and $\tB' = \left( \begin{smallmatrix} B \\ C' \end{smallmatrix} \right)$. Let $\tB'' = \left( \begin{smallmatrix} B \\ C \\ C' \end{smallmatrix} \right)$ with corresponding cluster variety $\cA''$.  Since every row of $C'$ is in the span of $\tB$, we can find $R = \left( \begin{smallmatrix} \Id_n & 0 \\ P & Q \end{smallmatrix} \right)$, where $Q$ is a lower-triangular unipotent matrix, so that $R \left( \begin{smallmatrix} B \\ C \\ 0 \end{smallmatrix} \right) = \left( \begin{smallmatrix} B \\ C\\ C' \end{smallmatrix} \right)$.  We have $\Ker(\mult(Q^T)) = \{1\}$, so by Proposition \ref{prop:coveringmap} we have $\cA \times (\CC^{\ast})^{m'} \cong \cA''$.  Similarly, $\cA' \times (\CC^{\ast})^m \cong \cA''$.

If $m = m'$, then we can find $R$ and $R'$ such that $R \tB = \tB'$ and $R' \tB' = \tB$.  (See also \cite[Corollary~4.5]{Fra}.)  So there are finite coverings $\cA \to \cA'$ and $\cA' \to \cA$. Whenever $X$ is a finite cover of $Y$, we obtain an injection $H^{\ast}(X, \QQ) \to H^{\ast}(Y, \QQ)$ so, in this case, we have injections $H^{\ast}(\cA, \QQ) \to H^{\ast}(\cA', \QQ) \to H^{\ast}(\cA, \QQ)$ which must both be isomorphisms. Moreover, these isomorphisms pass to the mixed Hodge structure.
\end{proof}

%

%
%

Even if $\tB$ and $\tB'$ are both $(n+m) \times n$ extended exchange matrices satisfying the assumption  of Proposition \ref{prop:rowspanenough}, it is not necessarily true that $\cA \cong \cA'$ as cluster varieties.  Consider the exchange matrices:
\[ \tB = \begin{pmatrix} 0 & 5 \\ -5 & 0 \\ 1 & 0 \\ 0 & 1 \end{pmatrix} \qquad \text{and} \qquad \tB' = \begin{pmatrix} 0 & 5 \\ -5 & 0 \\ 1 & 0 \\ 0 & 2 \end{pmatrix}. \]
Suppose we had $\left( \begin{smallmatrix} \mathrm{Id} & 0 \\ P & Q \end{smallmatrix} \right) \tB = \tB'$. Then $Q \equiv \left(\begin{smallmatrix} 1 & 0 \\ 0 & 2 \end{smallmatrix}\right) \bmod 5$ and we deduce that $\det Q \equiv 2 \bmod 5$. But, if $Q$ is invertible over the integers, then $\det Q = \pm 1$, a contradiction. 

We do not know whether or not $\cA \cong \cA'$ as abstract varieties. (In particular, we do not know this for the above example.)
In general, $X \times \CC^{\ast} \cong Y \times \CC^{\ast}$ does not imply $X \cong Y$~\cite{Dubouloz}. 
It is possible that examining cluster varieties might build additional examples of failure of cancellation for products with tori.

\subsection{GSV forms} \label{sec:GSV}
Let $\tB$ be a $n \times (n+m)$ extended exchange matrix, and $\cA = \cA(\tB)$ be the corresponding cluster variety.  We denote the cluster variables in the initial seed by $x_1$, $x_2$, \dots, $x_{n+m}$.  A \defn{Gehktman-Shapiro-Vainshtein form}, or \defn{GSV form} for $\cA$ is a 2-form
\begin{equation}\label{eq:GSV} \gamma = \sum_{i,j=1}^{n+m} \widehat{B}_{ij} \frac{d x_i \wedge d x_j}{x_i x_j}, 
\end{equation}
where $\widehat{B}$ is a $(n+m) \times (n+m)$  skew symmetric matrix whose first $n$ columns are equal to $\tB$.  We say that a GSV-form $\gamma$ is \defn{full rank} if the matrix $\widehat{B}$ is full rank.  Note that full rank GSV-forms only exist if $n+m$ is even.

Suppose that $\cA$ satisfies the Louise property and is of full rank.  Then any GSV form $\gamma$ extends to a closed $2$-form on all of $\cA$ (cf. \cite{Mul3} and Corollary \ref{cor:LinAlg}).  By Lemma \ref{lem:logImpliesTate}, the class $[\gamma] \in H^*(\cA)$ represented by $\gamma$ lies in $H^{2,(2,2)}(\cA)$.

A cluster algebra typically has many GSV forms.

\begin{lemma}\label{lem:GSVfullrank}
Suppose that $n+m$ is even.  Then a full rank GSV form exists.
\end{lemma}
\begin{proof}
We have to show that a full rank, integer $(n+m) \times n$ matrix $\tB$ with skew-symmetric principal part $B$ can be extended to a full rank, integer $(n+m) \times (n+m)$ skew-symmetric matrix $\wB$.  
In other words, we must find a way to choose $\widehat{B}_{ij}$, for $n+1 \leq i,j \leq n+m$, to give a skew symmetric matrix with nonzero determinant. Since this determinant is a polynomial in these $\widehat{B}_{ij}$, it is enough to see this polynomial is nonzero and therefore it is enough to show that we can find such an extension with rational entries.

Suppose $R$ is a non-singular $(n+m) \times (n+m)$ matrix.  Then the matrix $\wB$ is full-rank if and only if the matrix $R\wB R^T$ is.  Using matrices of the form $R = \left(\begin{smallmatrix} \ast & 0 \\ 0 & \Id_{m} \end{smallmatrix}\right)$, we may reduce to the case that the principal part of $\tB$ has the form $ \left(\begin{smallmatrix} B' & 0 \\ 0 & 0 \end{smallmatrix}\right)$, where $B'$ is a skew-symmetric non-singular $\rank(B) \times \rank(B)$ matrix.  Next, by using matrices of the form $R = \left(\begin{smallmatrix} \Id_{n} & 0 \\ \ast & \Id_{m} \end{smallmatrix}\right)$, we may reduce to the case that
\[
\tB = \begin{pmatrix} B' & 0 \\ 0 & 0 \\ 0 & C \\ 0 & 0
\end{pmatrix},
\]
where the four blocks have heights $\rank(B), n-\rank(B), n-\rank(B), m-n+\rank(B)$, and $C'$ is a full-rank square matrix of size $n-\rank(B)$.  It is clear that a full rank extension is given by
\[
\wB = \begin{pmatrix} B' & 0 & 0 & 0 \\ 0 & 0 & - C^T \\ 0 & C & 0 &0 \\ 0 & 0 & 0 & A
\end{pmatrix},
\]
for any full rank skew-symmetric $(m-n+\rank(B)) \times (m-n+\rank(B))$ matrix $A$.
\end{proof}

\begin{proposition}\label{prop:GSVpullback}
Suppose that $(\Psi,\Phi,R): \cA(\tC) \to \cA(\tB)$ is a covering map of cluster varieties.  Let $\gamma$ be a GSV-form for $\cA(\tB)$, given by the matrix $\wB$.  Then the pullback form $\Psi^*(\gamma)$ is a GSV-form for $\cA(\tC)$, given by the matrix $\wC = R\wB R^T$.  If $\gamma$ has full rank, then so does $\Psi^*(\gamma)$.
\end{proposition}
\begin{proof}
Denote the cluster variables of the initial seed of $\cA(\tB)$ (resp. $\cA(\tC)$) by $z_1,z_2,\ldots,z_{n+m}$ (resp. $x_1,x_2,\ldots,x_{n+m}$).  Since $\Psi^*(z_j) = \prod_{i=1}^{n+m} x_i^{R_{ij}}$, we have $\Psi^*(\dlog z_j) = \sum_{i=1}^{n+m} R_{ij} \dlog x_i$.  The formula for the pullback form $\Psi^*(\gamma)$ follows.  Suppose that $\wB$ has full rank.  The matrix $R$ is full rank (Proposition \ref{prop:coveringmap}) and we have $\tC = R\tB$, so the block form of $R$ implies that $\wC$ is a full rank skew-symmetric matrix extending $\tC$.  This gives the last statement.
\end{proof}

\section{Examples} \label{sec:examples}
In this section, we give mixed Hodge tables for various cluster varieties $\cA(\tB)$, to exhibit the phenomena which our later theorems will describe.  These were generally computed via ad hoc Mayer-Vietoris sequences arising from Proposition~\ref{prop:covering}, and using Proposition~\ref{prop:isolatedrank1} and Theorem \ref{thm:standard}. 
In the sequel~\cite{LS}, we will give more systematic ways of computing the cohomology of acyclic cluster varieties.

We first remark on a few ways in which we simplify our presentation. 
All of our examples are locally acyclic (in fact, acyclic) and of full rank. 
So, by Theorem~\ref{thm:curious}, all the cohomology is in degrees $(p,p)$ for the Deligne splitting, which is why we only display the dimensions of the $H^{k,(p,p)}$. 

If $\Gamma(\tB)$ is a tree, then any orientation of $\Gamma$ is mutation equivalent to any other orientation.  So, in examples involving trees, we do not specify how $\Gamma$ is oriented.

Proposition~\ref{prop:rowspanenough} means that we only need to specify $B$ and the integer row span of $\tB$ to give an example. Thus, in our examples, it is not necessary to give a specific bottom part to the $B$-matrix.

\subsection{Type $A_1$}
Table~\ref{tab:isolated} lists the mixed Hodge numbers $\dim H^{k,(p,p)}$ for the isolated cluster variety with $\tB = \left( \begin{smallmatrix} 0 \\ b \end{smallmatrix} \right)$, which we will compute in Proposition~\ref{prop:isolatedrank1}. 

The horizontal coordinate in the table is $k$ and the vertical coordinate in the table is $k-p$.  Thus the entry $b-1$ in the lower right indicates that $\dim H^{2,(1,1)}(\cA) = b-1$.  
Multiplication by the two-form $\gamma$ (see Proposition~\ref{prop:isolatedrank1}) moves two steps to the right in the table. 
The curious Lefschetz theorem (Theorem \ref{thm:curious}) states that every row of this table is palindromic, with centers on the diagonal line $p=(n+m)/2$. We have marked this diagonal line with \framebox{boxed text} in Table~\ref{tab:isolated}.

\begin{table}[!htbp] 
\begin{tabular}{|r|ccc|}
\hline
& $H^0$ & $H^1$ & $H^2$ \\
\hline
\rule{0pt}{1.5em} $k-p=0$ & $1$ & \framebox{$1$} & $1$ \\
$1$ & & & \framebox{$b-1$} \\[0.5 em]
\hline
\end{tabular}
\caption[Cohomology of an isolated cluster algebra]{The cohomology for $\tB = \left( \begin{smallmatrix} 0 \\ b \end{smallmatrix} \right)$. \label{tab:isolated}}
\end{table}

\subsection{Type $A_n$}
Suppose $\Gamma(\tB)$ is a path with $n$ vertices, and all nonzero entries of $\tB$ are $\pm 1$. If $n$ is even, then this exchange matrix is of really full rank even when $m = 0$.  The cohomology $H^\ast(\cA,\QQ)$ is computed in Proposition~\ref{prop:An} to be $\QQ[\gamma]/\gamma^{n/2+1}$.
So the Betti numbers are $(1,0,1,0,\ldots,0,1,0,1)$, and is entirely standard, meaning that it is concentrated entirely in Deligne summands $H^{k, (k,k)}$.

If $n$ is odd, then we must add a frozen row in order to make the matrix full rank. 
If we choose this frozen row so that $\tB$ is really full rank, then all the cohomology is in Deligne summands $H^{k,(k,k)}$ again (Proposition~\ref{prop:An}), with Betti numbers $(1,1,\ldots,1)$.

When $n$ is even,  these Betti numbers might suggest that $\cA$ is homotopy equivalent to $\CC \PP^{n/2}$.
However, this is wrong.  Computing with integer coefficients, we find that $\gamma^{n/2} = \left( \tfrac{n}{2} \right)!  \omega$ where $\gamma$ and $\omega$ are the integral generators of $H^2(\cA,\ZZ)$ and $H^{2n}(\cA,\ZZ)$ respectively. 

\subsection{Type $\tilde A_2$}
We next present the simplest examples of cluster varieties of really full rank where the cohomology is not entirely in $\bigoplus_k H^{k, (k,k)}$,  which we will term the standard part (Section~\ref{sec:standard}).

Suppose $\oG(\tB)$ is an acyclically oriented triangle, with all off-diagonal entries of $B$ being $\pm 1$.
(If we orient the triangle cyclically, it is mutation equivalent to $A_3$, a path.)
We must have $m \geq 1$ for $\tB$ to be full rank; we assume that $m = 1$ and that $\tB$ is really full rank.
The resulting cohomology is shown on the left half of Table~\ref{tab:triangle}.

Later in this section, we will want to consider the rank $4$ cluster variety obtained by adding another mutable variable to the triangle. We show the result on the right half of Table~\ref{tab:triangle}.  The top rows of these tables illustrate Theorem \ref{thm:standard}.

\begin{table}[!htbp] 
\[\begin{array}{|r|ccccc|}
\hline
& H^0 & H^1 & H^2 & H^3 & H^4 \\
\hline
k-p=0 & 1 & 1 & 1 & 1 & 1 \\
1 & & & & 1 &  \\
\hline
\end{array}
\qquad \qquad
\begin{array}{|r|ccccc|}
\hline
& H^0 & H^1 & H^2 & H^3 & H^4 \\
\hline
k-p=0 & 1 & 0  & 1 & 0 & 1 \\
1 & & & & 1 &  \\
\hline
\end{array}
 \]
\caption[Cohomology of a triangle]{The cohomology for $\tB = \left( \begin{smallmatrix} 0 & 1 & 1 \\ -1 & 0 & 1 \\ -1 & -1 & 0 \\ \hline 0 & 0 & -1 \end{smallmatrix} \right)$ and $\tB = \left( \begin{smallmatrix} 0 & 1 & 1 & 0 \\ -1 & 0 & 1 &0 \\ -1 & -1 & 0 & 1 \\ 0 & 0 & -1& 0 \end{smallmatrix} \right)$ \label{tab:triangle}}
\end{table}
In~\cite{LS} we will show that, if $\tB$ is acyclic and of really full rank, then $H^{2,(1,1)}(\cA) \cong 0$ and $H^{3,(2,2)}(\cA) \cong H^1(\Gamma, \CC)$.
So the nontrivial term in $H^{3,(2,2)}$ of Table~\ref{tab:triangle} reflects that $\Gamma$ is not a tree.

\subsection{An example with $k - p > 1$}
We give an example where $H^{k,(p,p)} \neq 0$ for $k-p >1$, where $n = 8$ and $m = 0$.
\begin{equation}\label{2Tri} 
\tB = \begin{pmatrix}
0 & 1 & 1 & 0 & 0 & 0 & 0 & 0 \\
-1 & 0 & 1 & 0 & 0 & 0 & 0 & 0 \\
-1 & -1 & 0 & 1 & 0 & 0 & 0 & 0 \\
0 & 0  & -1 & 0 & 1 & 0 & 0 & 0  \\
0 & 0  & 0 & -1 & 0 & 1 & 0 & 0  \\
0 & 0 & 0  & 0 & -1 & 0 & 1 & 1   \\
0 & 0 & 0 & 0  & 0 & -1 & 0 & 1   \\
0 & 0 & 0 & 0 & 0  & -1 & -1 & 0    \\
\end{pmatrix} 
\end{equation}
\begin{equation}
\vec{\Gamma}=
\xymatrix{
2 \ar[r] & 3 \ar[r] & 4 \ar[r] & 5 \ar[r] & 6 \ar[r] \ar[dr] & 7 \ar[d] \\
1 \ar[u] \ar[ur] & & & & & 8 \\
}
 \end{equation}
We use the Mayer-Vietoris sequence derived via Proposition~\ref{prop:covering} from the edge between the fourth and fifth vertex. All open sets are products of the examples we discussed in Table~\ref{tab:triangle} (sometimes with different frozen parts, but the same up to row span), so we can reuse those computations and compute the cohomology table in Table~\ref{tab:2Tri}.

The term in $H^{6, (4,4)}$ is the cup product of the two terms in $H^{3,(2,2)}$. As mentioned before, in~\cite{LS} we will show that the $H^{3,(2,2)}$ terms come from the cycles in $\Gamma$ so, roughly, we needed to make two cycles which were ``far enough apart" to have nontrivial cup product.

\begin{table}[!htbp] 
\[\begin{array}{|r|ccccccccc|}
\hline
& H^0 & H^1 & H^2 & H^3 & H^4 & H^5 & H^6 & H^7 & H^8 \\
\hline
k-p=0 & 1 & 0 & 1 & 0 & 1 &  0 & 1 & 0 & 1 \\
1 & & & & 2 &0 & 2 & 0 & 2 &   \\
2 & & & &  & &  & 1 &  &   \\
\hline
\end{array}\]
\caption[Cohomology with $k-p=2$]{The cohomology of $\cA(\tB)$ with $\tB$ given in \eqref{2Tri}. \label{tab:2Tri}}
\end{table}

\subsection{Types $D_4, E_6, E_7, E_8$}
Even when $\Gamma(\tB)$ is a tree, the cohomology need not be entirely standard.

Suppose $\Gamma(\tB)$ is the $D_4$ Dynkin diagram, that is, a star with three leaves.  For $\tB$ to be full rank, we must have $m \geq 2$.  We assume that $m = 2$ and that $\tB$ is really full rank.  The mixed Hodge numbers are shown in Table~\ref{tab:D4}. 
\begin{table}[!htbp] 
\[\begin{array}{|r|ccccccc|}
\hline
& H^0 & H^1 & H^2 & H^3 & H^4 & H^5 & H^6 \\
\hline
k-p=0 & 1 & 2 & 2 & 2 & 2 & 2 & 1 \\
1 & & & & &  1 & &  \\
\hline
\end{array} \]
\caption[Cohomology of $D4$]{The cohomology of really full rank $D_4$ with $m = 2$. \label{tab:D4}}
\end{table}

We also exhibit the cohomology of the cluster varieties of types $E_6$, $E_7$ and $E_8$ in Table~\ref{tab:E678}. The first and last of these is already of really full rank without adding any frozen variables; for $E_7$ we must add a frozen variable, and we choose to do so in order to make $\tB$ really full rank.

\begin{table}[!htbp] 
\[E_6: \begin{array}{|r|ccccccc|}
\hline
& H^0 & H^1 & H^2 & H^3 & H^4 & H^5 & H^6 \\
\hline
k-p=0 & 1 & 0 & 1 & 0 & 1 & 0 & 1 \\
1 & & & & &  1 & &  \\
\hline
\end{array} \]
\[E_7:\begin{array}{|r|ccccccccc|}
\hline
& H^0 & H^1 & H^2 & H^3 & H^4 & H^5 & H^6 & H^7 & H^8\\
\hline
k-p=0 & 1 & 1 & 1 & 1 & 1 & 1 & 1 & 1 & 1\\
1 & & & & &  1 & 1 & 1 & &  \\
\hline
\end{array} \]
\[E_8:\begin{array}{|r|ccccccccc|}
\hline
& H^0 & H^1 & H^2 & H^3 & H^4 & H^5 & H^6 & H^7 & H^8\\
\hline
k-p=0 & 1 & 0 & 1 & 0 & 1 & 0 & 1 & 0 & 1\\
1 & & & & &  1 & 0 & 1 & &  \\
\hline
\end{array} \]
\caption[Cohomology of $E_6$, $E_7$ and $E_8$]{The cohomology of $E_6$ (no frozen variables), $E_7$ (one frozen variable, really full rank), and $E_8$ (no frozen variables) cluster varieties. \label{tab:E678}}
\end{table}

\subsection{Grassmannians}
The {\it open positroid variety} $\mathring \Pi(r,d)$ in the Grassmannian $\Gr(r,d)$ of $r$-planes in $\CC^d$ is obtained by removing the $n$ hyperplanes $\{\Delta_{i,i+1,\ldots,i+r-1}  = 0\}$ from $\Gr(r,d)$ \cite{KLS}, where $\Delta_I$ denotes a Pl\"ucker coordinate.  The open positroid variety is a cluster variety \cite{Scott}.

The open positroid variety $\mathring \Pi(2,d) \subset \Gr(2,d)$ turns out to be a cluster variety $\cA(\tB)$ of really full rank, where $n= d-3$, and $m = d-1$, and $\Gamma(\tB)$ is a path with $n-1$ edges.  Thus Proposition \ref{prop:An} states that the Poincar\'{e} series of $H^\ast(\mathring \Pi(2,d))$ is given by
\[
P(H^\ast(\mathring \Pi(2,d)),t) = 
(1+t)^{d-2} (1+ t+ \cdots + t^{d-3} + t^{d-2} ).
\]

The open positroid cell $\mathring \Pi(3,6) \subset \Gr(3,6)$ has type $D_4$ and is of really full rank with $m=5$, so its cohomology is obtained by multiplying the cohomology of Table~\ref{tab:D4} by that of $(\CC^{\ast})^3$.
Similarly, $\mathring \Pi(3,7)$ and $\mathring \Pi(3,8)$ are of types $E_6$ and $E_8$, of really full rank with $m=6$ and $7$, so their cohomologies are the top and bottom tables of 
Table~\ref{tab:E678}, times that of $(\CC^{\ast})^6$ and $(\CC^{\ast})^7$.  We have repeatedly applied Proposition \ref{prop:rowspanenough}.

Larger Grassmannians are no longer acyclic, although they are locally acyclic~\cite{MS}. One could use the methods of~\cite{MS} to recursively compute $H^{\ast}(\mathring \Pi(d,n))$, but the computations quickly become impractical.  


\begin{remark}
The above examples give evidence for the following statement, which will be established in~\cite{LS}.  Let $\oG(\tB)$ be acyclic and $\tB$ be of really full rank, and suppose that $H^{k, (p,p)} \neq 0$. Then $p \geq \max(\tfrac{2}{3} k,2k-m-n)$.

This is stronger than the bound $\tfrac{1}{2} k \leq p $ from Theorem \ref{T:weights}.
\end{remark}
%


\section{Isolated cluster varieties}\label{sec:isolated}
A seed $t = (\x,\tB)$ is \defn{isolated} if $\Gamma(t)$ has no edges. We also use this term for the cluster algebra $A(\tB)$ and the cluster variety $\cA(\tB)$.

\subsection{Rank 1}\label{sec:rankone}

Let $\tilde{B} = \left( \begin{smallmatrix} 0 \\ b \end{smallmatrix} \right)$. So $\cA = \{(x,x',y) \mid x x' = y^b+1 \text{ and } y \neq 0\} \subset \CC^3$.  The aim of this section is to describe the cohomology of $\cA$.  Let $\gamma$ be the $2$-form $\dlog x \wedge \dlog y$. Note that $\dlog x + \dlog x'= b \ y^{r-1} dy/(y^b+1)$, so $\dlog x \wedge \dlog y = - \dlog x' \wedge \dlog y$. Thus, $\gamma$ extends to a form which is defined whenever $x$ or $x'$ is nonzero. The locus where both $x$ and $x'$ vanish is codimension $2$ in the smooth variety $\cA$, so $\gamma$ extends to a regular $2$-form on all of $\cA$, which we also denote by $\gamma$.


\begin{prop}\label{prop:isolatedrank1}
The variety $\cA$ retracts onto the locus $|x|=|x'|$, $|y|=1$, which is $b$ two-dimensional spheres glued in a circle. The Betti numbers are $(1,1,b)$. In $H^0$ and $H^1$, representative classes are $1$ and $\dlog y$, living in $(2 \pi i)^k H^k(\cA, \QQ)$, sitting in the Deligne summands $H^{0,(0,0)}$ and $H^{1,(1,1)}$ respectively.

A basis for $H^2(\cA, \CC)$ is given by the set of $b$ two-forms $\gamma$, $y \gamma$, $y^2 \gamma$, \dots, $y^{b-1} \gamma$. The form $\gamma$ is in $(2 \pi i)^2 H^2(\cA, \QQ)$ and spans the Deligne summand $H^{2,(2,2)}$. The forms $y \gamma$, $y^2 \gamma$, \dots, $y^{b-1} \gamma$ are in $(2 \pi i) H^2(\cA, \QQ(\exp(\pi i/b)))$, the vector space they span has a basis in $H^2(\cA, \QQ)$, and their span is the Deligne summand $H^{2,(1,1)}$.  The mixed Hodge structure on $H^2(\cA)$ is split over $\QQ$.
\end{prop}

When $b=1$, this result specializes to say the following:
\begin{cor} \label{cor:isolatedreallyfullrank}
Let $\tilde{B} = \left( \begin{smallmatrix} 0 \\ 1 \end{smallmatrix} \right)$ with cluster variety
$
\{(x,x',y) \mid xx' = y+1 \text{ and } y \neq 0\} \subset \CC^3.
$ 
This retracts onto the locus $|x|=|x'|$, $|y|=1$, which is a pinched torus. The Betti numbers are $(1,1,1)$ with representative classes $(1, \dlog y, \dlog x \wedge \dlog y)$. The classes $\dlog y$ and $\dlog x \wedge \dlog y$ live in $(2 \pi i) H^1(\cA, \QQ)$ and $(2 \pi i)^2 H^2(\cA, \QQ)$, and are in Deligne summands $H^{1,(1,1)}$ and $H^{2, (2,2)}$ respectively.
\end{cor}

\begin{proof}[Proof of Proposition~\ref{prop:isolatedrank1}]

We begin by retracting $\cA$ onto a necklace of $b$ two-dimensional spheres.

Let $h(t,r):[0,1] \times \RR_{>0} \to \RR_{>0}$ be a smooth function such that
\begin{itemize}
\item $h(1,r)=r$ and $\lim_{t \to 0} h(t,r)=1$ uniformly on compact subsets of $\RR_{>0}$ 
\item $h(t,1)=1$ and $\tfrac{\partial h}{\partial r}(t,1)=1$.
\end{itemize}
For an explicit example, one may take $h(t,r) = \exp \left( \tfrac{t}{1-t} \tanh \tfrac{(1-t) \log r}{t} \right)$ for $0<t<1$, with $h(1,r)=r$ and $h(0,r)=1$.

The map $\cA \times [0,1] \to \cA$
\[
(x,x', r e^{i \phi}) \mapsto \left(x,\frac{1+h(t,r)^b e^{i b\phi}}{1+r^b e^{i b \phi}} x', h(t,r) e^{i \phi} \right)
\]
is a deformation retract of $\cA$ onto the locus $|y|=1$.  (The condition $\tfrac{\partial h}{\partial r}(t,1)=1$ means that we can extend the fraction continuously to $1$ at $r=1$, $e^{i b \phi}=-1$.)

Write $x=r e^{i \theta}$ and $x' = r' e^{i \theta'}$. The map $\{|y|=1\} \times [0,1] \to \{|y|=1\}$
\[
(r e^{i \theta}, r' e^{i \theta'} ,y) \times t \longmapsto \left( r^{1-t/2} (r')^{t/2} e^{i \theta},  r^{t/2} (r')^{1-t/2} e^{i \theta'}    , y\right)
\]
is a further deformation retract onto the locus $|x|=|x'|$.  Thus $\cA$ deformation retracts to the locus $Y$ where $|x|=|x'|$ and $|y|=1$. 

 The fiber of $Y$ over some $y \in S^1$ is a circle if $y^b+1 \neq 0$, and is a point if $y ^b+ 1 = 0$.  Thus $\cA$ is homotopy equivalent to the space $Y$ consisting of $b$ two-dimensional spheres $S_1$, $S_2$, \ldots, $S_b$ glued in a circle. Specifically, we write $S_j$ for the sphere where $y = \exp(i \phi)$ for $\phi \in [(2 j-1) \pi/b, (2j+1) \pi /b]$.

It is easy to check that the Betti numbers of $Y$, and thus $\cA$, are $(1,1,b)$.

The first homology $H_1(Y,\ZZ)$ is spanned by the class of a continuous curve $\gamma:[0,1] \to Y$ where $\gamma(t) =(x(t),x'(t),e^{2\pi i t})$.  It is thus clear that $\dlog y$ spans $H^1(\cA,\CC)$, and the integral
\[
\int_\gamma \frac{dy}{y} = 2\pi i
\]
shows that $\dlog y$ lies in $(2 \pi i) H^1(\cA, \QQ)$.  
Map $\cA$ to $\mathbb{C}^{\ast}$ by $(x,x',y) \mapsto y$. We see that the induced map $H^1(\CC^{\ast}, \CC) \to H^1(\cA, \CC)$ is an isomorphism. Since pullbacks respect the Deligne splitting, this shows that the one-dimensional space $H^1(\cA, \CC)$ is equal to the Deligne summand $H^{1,(1,1)}(\cA)$.

We next work with $H^2(\cA)$.
We will parametrize the sphere $S_j$ with coordinates $\lambda$ (longitude) and $\phi$ (latitude), so that $y = e^{i \phi}$ with $\phi \in [(2 j-1) \pi/b, (2j+1) \pi /b]$ and $x = r(\phi) e^{i \lambda}$ where $r(\phi) = \sqrt{|1+e^{i b \phi}|}$.
In this parametrization, 
\[ \gamma = \frac{dx}{x} \wedge \frac{dy}{y} = (r'(\phi)/r(\phi) d \phi + i d \lambda) \wedge (i d \phi) = d \phi \wedge d \lambda. \]

Thus, 
\[ \int_{S_j} \gamma = \int_{\phi = (2j-1) \pi/b}^{(2j+1) \pi/b} \int_{\lambda = 0}^{2 \pi} d \phi \ d \lambda = \frac{(2 \pi i)^2}{b} .\]
So $\frac{1}{(2 \pi i)^2} \gamma$ represents a class in $H^2(\cA, \QQ)$.
We remark that, although this computation was easy, it forms the base case of our eventual computation that the mixed Hodge structures of cluster varieties are split over $\QQ$.

%

For $c \neq 0$, the integral of $y^c \gamma$ over the 2-sphere $S_j$ can be computed as
\begin{multline}
\int_{S_j} y^c \gamma = \int_{\phi = (2j-1) \pi/b}^{(2j+1) \pi/b} \int_{\lambda = 0}^{2 \pi} e^{i c \phi} d \phi \ d \lambda = \\ \frac{2\pi i}{c}\left(\exp((2j+1) c \pi i/b)- \exp((2j-1) c \pi i/b) \right)
= \frac{-4\pi}{c}\sin(\pi c /b)  e^{jc 2\pi i/b} \label{keyIntegral}  \end{multline}
Thus we have $\frac{1}{ 2\pi i } y^c \gamma \in H^2(\cA,\QQ(e^{\pi i /b}))$.

In order to check that $\gamma$, $y \gamma$, $y^2 \gamma$, \dots, $y^{b-1} \gamma$ is a basis for $H^2(\cA, \CC)$, we must check that it pairs perfectly against the $b$ spheres, which are a basis for $H_2(\cA, \QQ)$. In other words, we must check that the $b \times b$ matrix whose $(c,j)$ entry is $\int_{S_j} y^c \gamma$ is invertible. Rescaling the $c=0$ row by $(2 \pi i)^2/b$ and the other rows by $\frac{-4\pi}{c}\sin(\pi c /b)$, we obtain the Vandermonde matrix $\left((e^{2 \pi i/b})^{jc} \right)$, which is clearly invertible.

We abbreviate $\mathrm{Span}(y \gamma, y^2 \gamma, \dots, y^{b-1} \gamma)$ to $V$.
Note that $V$ is the kernel of $\sum_j \int_{S_j} ( \ )$, so it is defined over $\QQ$ as claimed.

We now compute the Deligne splitting on $H^2(\cA)$.

Let $U$ be the open subset of $\cA$ where $x$ is nonzero, and let $U'$ be the open locus where $x'$ is nonzero. So $U \cup U'$ is all of $\cA$ except for the $m$ points $(0,0,e^{\pi i (2j+1)/b})$.  By an application of the Mayer-Vietoris sequence, removing finitely many points from a four-dimensional manifold does not change $H^2$.  Thus the restriction map $H^2(\cA) \to H^2(U \cup U')$ is an isomorphism, and we may compute the mixed Hodge structure on $H^2(U \cup U')$ instead. 
For this, we have the Mayer-Vietoris sequence:
\[
H^1(U \cap U') \to H^2(U \cup U') \to H^2(U) \oplus H^2(U') .
\]
We claim that $V$ is the kernel of $H^2(U \cup U', \CC) \to H^2(U, \CC) \oplus H^2(U', \CC)$. 
This is easy to check: On $U$, we have $y^c \gamma = \frac{1}{c} d ( y^c dx/x) $ and, on $U'$, we have $y^c \gamma = - \frac{1}{c} d ( y^c dx'/x') $. 
So $y^c \gamma$ becomes exact on $U$ and $U'$ and is in the kernel. 
The image of $\gamma$ in $H^2(U)$ is nonzero by Lemma~\ref{lem:dlogzero}, so $V$ is precisely the kernel.

Projection onto the $x$ and $y$ coordinates shows that $U \cap U' \cong \mathbb{C}^{\ast} \times (\mathbb{C} \setminus \{ e^{\pi i (2j+1)/b} \} )$. 
It is easy to check that $H^1( \mathbb{C}^{\ast} \times (\mathbb{C} \setminus \{ e^{\pi i (2j+1)/b} \} ))$ is entirely in the Deligne summand $H^{1,(1,1)}$.  So the image of this cohomology in $H^2(U \cup U')$ must be in degree $H^{2,(1,1)}$ and we see that $V \subseteq H^{2, (1,1)}(\cA)$.

Also, by Lemma~\ref{prop:rationalLogImpliesTate},  $\mathrm{Span}(\gamma) \subseteq H^{2,(2,2)}(\cA)$.

So far, we have shown that $V \subseteq H^{2, (1,1)}(\cA)$ and  $\mathrm{Span}(\gamma) \subseteq H^{2,(2,2)}(\cA)$.
 Since $H^2 = \mathrm{Span}(\gamma) \oplus V$, we have shown that $ \mathrm{Span}(\gamma) = H^{2,(2,2)}$ and $V = H^{2,(1,1)}$.

We computed that $(2 \pi i)^{-2} \gamma$ is in $H^2(\cA, \QQ)$ and showed that $(2 \pi i)^{-2} \gamma$ spans the $(2,2)$ classes.
And we computed that $V$ has a basis in $H^2(\cA, \QQ)$ and spans the $(1,1)$ classes. 
This shows that the mixed Hodge structure on $H^2(\cA)$ is split over $\QQ$ and concludes the proof of Proposition~\ref{prop:isolatedrank1}. \end{proof}

The cyclic group $\{ \zeta \in \CC^{\ast} : \zeta^b=1 \}$ acts on $Y$ by $(x,x',y) \mapsto (x,x',\zeta y)$. We now decompose $H^{\ast}(Y, \CC)$ into isotypic components for this action.

\begin{cor} \label{BasicIsotypics}
Let $r \in \ZZ/b \ZZ$ and consider the isotypic component of $H^{\ast}(Y, \CC)$ where $\zeta$ acts by $\zeta^r$ for each $b$-th root of unity $\zeta$. If $r=0$, this isotypic component occurs in Deligne degrees $(0,(0,0))$, $(1,(1,1))$ and $(2,(2,2))$, with dimension $1$ each. If $r \neq 0$, then this isotypic component is one-dimensional, in degree $(2,(1,1))$.
\end{cor}

\begin{remark}
The first author and Pylyavskyy have generalized cluster algebras to Laurent phenomenon algebras \cite{LP}. The Laurent phenomenon algebra which plays the role analogous to $x x' = y^b+1$ (with $y \neq 0$) is $x x' = f(y)$ (with $y \neq 0$) where $f$ is a polynomial of degree $b$. Call this variety $Z$. Let $f(y) = \prod_{i=1}^b (y - \alpha_i)$ and let us assume that the $\alpha_i$ are distinct. 

In this case, it is still true that $Z$ retracts onto a wedge of spheres. One can still define $\gamma = \dlog x \wedge \dlog y$ and show that $\gamma$, $y \gamma$, \dots, $y^{b-1} \gamma$ is a basis for $H^2$. It is still true that $\mathrm{Span}(y \gamma, y^2 \gamma, \dots, y^{b-1} \gamma)$ has a basis in $H^2(Z, \QQ)$. However, integrating $\gamma$ over a basis for $H_2(Z, \QQ)$ now produces terms of the form $\log (\alpha_r/\alpha_s)/(2 \pi i)$. So the mixed Hodge structure is not split over $\QQ$ in this case. 
It is interesting to speculate that the cluster algebras are in some sense the center of the moduli space of Laurent phenomenon algebras and that, at this center, the mixed Hodge structure is split.
\end{remark}

\subsection{The general isolated case}

Let $\tB$ be an $(n+m) \times n$ extended exchange matrix of the form $\left( \begin{smallmatrix} 0 \\ C \end{smallmatrix} \right)$. We make the standing assumption that $\tB$ has full rank. Let $\cA$ be the corresponding cluster variety.
In this section, we will give a complete description of $H^{\ast}(\cA)$.

Let $C_j$ denote the $j$-th column of $C$.  The group $\Aut(A)$ acts on $\cA$, and hence on $H^{\ast}(\cA)$. The connected component of the identity must act trivially on $H^{\ast}(\cA)$, so we can decompose $H^{\ast}(\cA)$ into isotypic components for the locally constant characters of $\Aut(\cA)$.
In Corollary~\ref{LocConstCharacter}, we saw that the locally constant characters of $\Aut(A)$ are given by $G:= (\tB \QQ^n \cap \ZZ^{n+m})/\tB \ZZ^n$. We note for future reference that, due to the block form of $\tB$, we could also write $G = (C \QQ^n \cap \ZZ^m) / C \ZZ^n$.

For $g \in G$, we can lift $g$ to a representative $\tilde{g} = \sum_{j=1}^n q_j C_j$ where the $q_j$ are in $\QQ$ and are well defined modulo $\ZZ$. 
We define $J(g) = \{ j \in [n] : q_j \not\in \ZZ \}$, and denote by $H^k(\cA)^g$ the $g$-isotypic component of $H^k(\cA)$.

\begin{theorem} \label{thm:Isolated Main}
The mixed Hodge structure on $H^{\ast}(\cA)$ is of mixed Tate type and split over $\QQ$. Fix $g \in G$. The $g$-isotypic component of $H^{k}(\cA)$ lies completely in the Deligne summand $H^{k,(k-|J(g)|, k-|J(g)|)}(\cA)$, and the dimensions of $H^k(\cA)^g$ are given by
\[ \sum_k \dim H^k(\cA)^g q^k = q^{2 |J(g)|} (1+q+q^2)^{n-|J(g)|} (1+q)^{m-n} . \]
\end{theorem}

\begin{proof}
By Proposition \ref{prop:coverprincipal}, $\cA$ is covered by a cluster variety $\bar{\cA}$, with extended exchange matrix $\bar{\tB} = R\tB = \left( \begin{smallmatrix} 0 \\ d \, \mathrm{Id}_n \\ 0  \end{smallmatrix} \right)$ where the integer $d$ is given in Proposition \ref{prop:coverprincipal}.  Let $\bar{C} = \left( \begin{smallmatrix} d \, \mathrm{Id}_n \\ 0 \end{smallmatrix} \right)$ and $R = \left( \begin{smallmatrix} \mathrm{Id}_n & 0 \\ 0 & Q \end{smallmatrix} \right)$.  Then the covering map $\bar{\cA} \to \cA$ has deck group $H = \Ker(\mult(Q^T))$, and we have $\bar C = QC$.

So $H^{\ast}(\cA, \CC)$ is the $H$-invariant part of $H^{\ast}(\bar{\cA}, \CC)$. 
We will describe the $g$-isotypic component of $H^{\ast}(\cA, \CC)$ as a subspace of $H^{\ast}(\bar{\cA}, \CC)$. 
The character $g$ of $\Aut(\cA)$ pulls back to the character $Q g$ of $\Aut(\bar{\cA})$; here we are thinking of $g$ as an element of $(C \QQ^n \cap \ZZ^m) / C \ZZ^n$.
So the $g$-isotypic component of $H^{\ast}(\cA, \CC)$ is the $Q g$ isotypic component of $H^{\ast}(\bar{\cA}, \CC)$. 

Choose a lift $\tilde{g} = \sum_{j=1}^n q_j C_j$ of $g$ as above. Since $Q$ maps the columns of $C$ to $d$ times the first $n$ coordinate vectors, we have $Q \tilde{g} = (d q_1, d q_2, \ldots, d q_n, 0,0,\ldots,0)$. Also, since $g \in \ZZ^m$ and $Q$ has integer entries, note that the $d q_j$ are integers. We set $r_j =d q_j$.

Thus, our goal is to find the $(r_1,\ldots, r_n,0,0,\ldots,0)$ isotypic component of $H^{\ast}(\bar{\cA}, \CC)$.
Let $Y$ be the cluster variety $\{ (w,w',z) : w w' = z^{d}+1,\ z \neq 0 \}$. Then $\bar{\cA} \cong Y^n \times (\CC^{\ast})^{m-n}$ so, by the K\"unneth isomorphism,
\[ H^{\ast}(\bar{\cA}, \CC) \cong H^{\ast}(Y, \CC)^{\otimes n} \otimes H^{\ast}(\CC^{\ast}, \CC)^{\otimes (n-m)}. \] 
In order to obtain the $(r_1,\ldots, r_n,0,0,\ldots,0)$-isotypic component of $H^{\ast}(\bar{\cA}, \CC)$, we tensor together the $r_j$ component of $H^{\ast}(Y, \CC)$ for each $j$, and all of $H^{\ast}(\CC^{\ast}, \CC)$. 
Corollary~\ref{BasicIsotypics} says that the $j$-th $H^{\ast}(Y, \CC)$ contributes in degrees $(0,(0,0))$, $(1,(1,1))$ and $(2,(2,2))$ if $r_j \equiv 0 \bmod d$, and in degree $(2,(1,1))$ if $r_j \not \equiv 0 \bmod d$. The cohomology of $\CC^{\ast}$ is in Deligne degrees $(0,(0,0))$ and $(1,(1,1))$, with dimension $1$ in each degree.

We note that $J(g) = \#\{j \in [n] \mid r_j \not \equiv 0 \bmod d\}$.  Both statements of the theorem now follow.
\end{proof}

\subsection{Union of cluster tori}
For  later use, we note the following result.
\begin{proposition}\label{prop:uniontori}
Suppose that $\cA$ is locally acyclic.  Let $\bigcup T$ denote the union of all the cluster tori $T \subset \cA$.  Then
$\cA \setminus \bigcup T$ has codimension at least two in $\cA$.
\end{proposition}
\begin{proof}
We first establish the result in the case that $\cA$ is isolated.  By Proposition \ref{prop:coverprincipal}, $\cA$ has a finite covering by a cluster variety $\bar{\cA}$, with extended exchange matrix $\bar{\tB} = \left( \begin{smallmatrix} 0 \\ d \, \mathrm{Id}_n \\ 0  \end{smallmatrix} \right)$.  The inverse image of $\bigcup T$ in $\bar{\cA}$ is the union of cluster tori $\bigcup \bar T$ in $\bar{\cA}$.  The desired statement for $\cA$ follows from the statement for $\bar \cA$.  Let $x_1,x_2,\ldots,x_n$ be the initial cluster variables, $x'_1,\ldots,x'_n$ the mutated cluster variables, and $y_1,\ldots,y_{m}$ be the frozen variables of $\bar \cA$.  The cluster variety $\bar \cA$ has $2^n$ cluster charts, each given by the non-vanishing of either $x_i$ or $x'_i$, for each $i \in [n]$.  Let $Y_i = \{x_i = x'_i = 0\} \subset \cA$.  Then $\bar \cA \setminus \bigcup \bar T = \bigcup Y_i$.  But if $x_i = x'_i = 0$ then the equation $x_ix'_i = y_i^d + 1$ forces $y_i$ to take one of finitely many values.  Thus $Y_i$ is isomorphic to finitely many copies of a cluster variety of codimension two.  It follows that $\bar \cA \setminus \bigcup \bar T$ has codimension two in $\bar \cA$.

Now suppose $\cA$ is an arbitrary locally acyclic cluster variety.  Let $Y = \cA \setminus \bigcup T$.  
Suppose $Y$ contains an irreducible component $D$ of codimension one.  Since $\cA$ can be covered by cluster charts that are isolated, there exists some isolated cluster chart $\cA' \subset \cA$ containing the generic point of $D$.  But each cluster torus $T'$ in $\cA'$ is also a cluster torus in $\cA$.  So this contradicts the claim for isolated cluster varieties, and we conclude that the statement holds for all locally acyclic cluster varieties.
\end{proof}

\section{Mixed Hodge structure and curious Lefschetz for cluster varieties}
\label{sec:mixedhodge}

In this section, we prove the curious Lefschetz theorem for cluster varieties satisfying the Louise property and of full rank.  

\subsection{Tori}
\begin{prop}\label{prop:curioustorus}
Let $T = T_{2r} = \Spec(\CC[x_1^{\pm 1}, \ldots, x_{2r}^{\pm_1}])$ be the $2r$-dimensional complex torus.  Let $B$ be a $2r \times 2r$ skew-symmetric matrix of full rank.  Then the curious Lefschetz property holds for the pair $(T,\gamma)$, where $\gamma = \sum_{i,j=1}^{2r} B_{ij} \dlog x_i \wedge \dlog x_j$.
\end{prop}
\begin{proof}
The forms $\theta_j =\dlog x_j$ for $j = 1,2,\ldots, 2r$ form a basis of $H^1(T_{2r},\QQ)$.  In another $\QQ$-basis $\{\theta'_1,\ldots,\theta'_{2r}\}$ of $H^1(T_{2r},\QQ)$, the 2-form $\gamma$ can be written as $\gamma = \sum_{i=1}^r a_i \theta'_{i} \wedge \theta'_{i+r}$, for non-zero rational numbers $a_i$.  By Proposition \ref{prop:curiousproduct}, it suffices to prove the result for $r =1$, that is, for the two-torus.  For the two-torus, the result is immediate.
\end{proof}

\subsection{Isolated cluster varieties}


\begin{proposition} \label{Isolated Curious}
Suppose $\cA$ is an isolated cluster variety of full rank and even dimension, that is, $n+m$ is even.  Then the pair $(\cA,\gamma)$ satisfies the curious Lefschetz property for any full rank GSV-form $\gamma$.
\end{proposition}

\begin{proof}
We suppose that $\cA = \cA(\tB)$ where $\tB =  \left( \begin{smallmatrix} 0 \\ C \end{smallmatrix} \right)$, and $C$ is an $m \times n$ matrix of rank $n$, where $n+m$ is even.  

As in the proof of Theorem~\ref{thm:Isolated Main}, we have a covering map $(\Psi,\Phi,R): \bar{\cA} \to \cA$, with deck group $H$. By Proposition \ref{prop:GSVpullback}, the form $\gamma$ pulls back to an $H$-invariant GSV form $\bar{\gamma} := \Psi^{\ast}(\gamma)$ on $\bar{\cA}$.  The cohomology $H^{\ast}(\cA)$ (with its Deligne splitting) is identified with the $H$-invariant part of $H^{\ast}(\bar{\cA})$. 
If the Theorem holds for $H^{\ast}(\bar{\cA})$ and $[\bar \gamma]$, then it will in particular hold for the $H$-invariant part of $H^{\ast}(\bar{\cA})$.

So we are reduced to studying the case where $C$ is of the form $\left( \begin{smallmatrix} d \mathrm{Id}_n \\ 0 \end{smallmatrix} \right)$ . We now assume $C$ is of this form, so we can use the notations $\cA$, $\gamma$, etcetera rather than $\bar{\cA}$, $\bar{\gamma}$, etcetera.

The integer matrix $\wB$ has the form
\[
\wB = \begin{pmatrix} 0 & -d\, \Id_n & 0\\
d\,\Id_n & S & -A^T \\
0 & A & M
\end{pmatrix}
\]
where the $A$ is an arbitrary $n \times (m-n)$ matrix, $S$ is a skew-symmetric $n \times n$ matrix, and $M$ is a skew-symmetric full rank $(m-n) \times (m-n)$ matrix.

We first reduce to the case that all the entries of $S, A, M$ are divisible by $d$.  Applying Proposition \ref{prop:GSVpullback} with $R = \left(\begin{smallmatrix}\Id_n & 0 \\ 0 & d \, \Id_n \end{smallmatrix}\right)$, we compute that
\[
R\wB R^T = \begin{pmatrix}\Id_n & 0 \\ 0 & d\, \Id_n \end{pmatrix}\wB \begin{pmatrix}\Id_n & 0 \\ 0 & d\, \Id_n \end{pmatrix}= \begin{pmatrix} 0 & -d^2\, \Id_n & 0\\
d^2\,\Id_n & d^2\,S & -d^2 \,A^T \\
0 & d^2\, A & d^2\, M.
\end{pmatrix}
\]
So after replacing $d$ with $d^2$, and $\cA$ with a cover, we may assume that all the entries of $S, A, M$ are divisible by $d$.  Now let us pick
\[
R = \begin{pmatrix}
\Id_n & 0 & 0 \\ 
X & \Id_n & 0 \\ 
\frac{1}{d}\,A & 0 & \Id_{m-n}
\end{pmatrix}
\]
where $X$ is a $n \times n$ integer matrix satisfying $X - X^T = S/d$.  Then 
\begin{equation}\label{eq:finalwB}
R \wB R^T = \begin{pmatrix} 0 & -d\, \Id_n & 0\\
d\,\Id_n & 0 &0 \\
0 & 0 & M
\end{pmatrix}.
\end{equation}
So again by applying Proposition \ref{prop:GSVpullback}, we may and will assume that $\wB$ has the form given by the RHS of \eqref{eq:finalwB}.  Denote $\theta_j = \dlog x_j$.  We need to consider the two-form
\[
\gamma = 2 d \sum_{j=1}^n \theta_j \wedge \theta_{n+j} + \sum_{k,l=2n+1}^{n+m} M_{k,\ell}\, \theta_k \wedge \theta_\ell \] 
on $\cA$, where the matrix $M$ is skew-symmetric and full rank.

In this case, $\tB = \left(\begin{smallmatrix} 0 \\ d\Id_n \\ 0\end{smallmatrix}\right)$, so $\cA \cong Y_d^n \times T_{m-n}$, where $Y_d$ is the cluster variety with extended exchange matrix $\left(\begin{smallmatrix} 0 \\ d\end{smallmatrix}\right)$ studied in Section \ref{sec:rankone}, and $T_{m-n} \cong (\CC^\ast)^{m-n}$ is the $(m-n)$-dimensional torus.

Denote the cluster variable of $Y_d$ by $x$ and the frozen variable by $y$.  By Proposition \ref{prop:curiousproduct}, it suffices to show that the curious Lefschetz property holds for the pair $(Y_d, 2d \, \dlog x \wedge \dlog y)$ and for the pair $(T_{(n-m)}, \sum_{k,l=2n+1}^{n+m} M_{k,\ell} \dlog x_{k-2n} \wedge \dlog x_{\ell-2n})$.  The former follows from the computations in Section \ref{sec:rankone}.  The latter follows from Proposition \ref{prop:curioustorus}.
\end{proof}

\subsection{General case}\label{ssec:curious}
We now prove the curious Lefschetz property in general.  We remark that most of the important consequences also hold in the case where $\cA$ is odd dimensional.

\begin{theorem}\label{thm:curious}
Let $\cA$ be a cluster variety of rank $n$ with $m$ frozen variables, satisfying the Louise property and of full rank.  Then the mixed Hodge structure of $H^\ast(\cA)$ is of mixed Tate type and is split over $\QQ$.  
\begin{enumerate}
\item
 For any integers $s$ and $p$, we have $\dim H^{p+s, (p,p)}(\cA) = \dim H^{n+m-p+s, (n+m-p, n+m-p)}(\cA)$.
\item
If $n+m$ is even, and $\gamma$ is a full rank GSV-form, then $(\cA, \gamma)$ satisfies the curious Lefschetz property.
\end{enumerate}
\end{theorem}

\begin{proof}
We first consider the case that $n+m$ is even, so a full rank GSV-form $\gamma$ exists by Lemma \ref{lem:GSVfullrank}.
We note that all the cluster varieties $\cA'$ appearing in the Louise algorithm for $\cA$ are of full rank (see Remark \ref{rem:fullrank}), and that $\gamma|_{\cA'}$ is a GSV-form for each such $\cA'$.

In Proposition \ref{Isolated Curious} we showed that $H^{\ast}(\cA')$ satisfies the conclusions of the theorem when $\cA'$ is isolated and of full rank.  This is the base case of the recursion in the Louise algorithm.  It then follows from Lemma \ref{lem:mixedTate} that the mixed Hodge structure of $H^\ast(\cA)$ is of mixed Tate type, and from Theorem \ref{CuriousMV} that $(\cA,\gamma)$ satisfies the curious Lefschetz property.  By Theorem \ref{thm:splitQ}, the mixed Hodge structure of $H^\ast(\cA)$ is split over $\QQ$.

We now consider the case where $n+m$ is odd. Let $\tB'$ be the extended exchange matrix where we add one more frozen row to $\tB$, with the additional row being $(0,0,\ldots,0)$. Letting $\cA'$ be the cluster variety with exchange matrix $\tB'$, we have $\cA' \cong \cA \times \CC^{\ast}$. The desired properties of $\cA$ then follow from the theorem for $\cA'$ and the K\"unneth theorem.
\end{proof}

\section{Standard forms on cluster varieties}\label{sec:standard}
Let $X$ be a smooth complex affine algebraic variety whose cohomology is of mixed Tate type.  We define the \defn{standard part} $H^\ast(X)_{st}$ of $H^\ast(X)$ to be the subring
\[
H^\ast(X)_{st} := \bigoplus_k H^{k,(k,k)}(X)
\]
and the nonstandard part to be
\[ 
H^{\ast}(X)_{ns} = \bigoplus_k \bigoplus_{q<k} H^{k, (q,q)}(X).
\]

In this section, we will engage in a detailed study of $H^{\ast}(\cA)_{st}$, which appears to be the most tractable part of $H^{\ast}(\cA)$. 
We summarize our main results:
Proposition~\ref{prop:kernel} provides an alternate description of $H^{\ast}(\cA)_{ns}$ for cluster varieties.
Theorem~\ref{thm:standardequiv} provide several useful theoretical characterizations of $H^{\ast}(\cA)_{st}$.
In Corollary~\ref{cor:LinAlg}, we give a practical method for computing the standard cohomology.
In Theorem~\ref{thm:standard} and Proposition~\ref{prop:standardbasis}, we give explicit ring theoretic generators and a vector space basis for the standard part.

\subsection{Restriction to tori}

For a torus $T$, we have $H^\ast(T)_{st} = H^\ast(T)$.

\begin{prop}\label{prop:kernel}
Suppose that the cluster variety $\cA$ is a smooth complex algebraic variety, and let $T$ be a cluster torus of $\cA$.  Then $H^\ast(X)_{st}$ injects into $H^\ast(T)$ and $H^{\ast}(X)_{ns} = \Ker(H^\ast(X) \to H^\ast(T))$.  In particular, the kernel of the restriction map of cohomology to a cluster torus $T$ does not depend on the choice of $T$.
\end{prop}

%
%

\begin{proof}
Since all of $H^k(T)$ is in degree $(k,k)$, it is clear that $H^k(X)_{ns}$ is in the kernel of restriction to $T$. On the other hand, $H^{k, (k,k)}(X) \to H^{k, (k,k)}(T)$ is injective by Lemma~\ref{lem:topHodgeInjective}.
\end{proof}

\subsection{Standard forms}\label{ssec:standard}
We say that a differential form $\omega$ of degree $k$ on a torus $T = \Spec \CC[x_1^{\pm 1},\ldots, x_r^{\pm_1}]$ is \defn{standard} if it can be written as
\[
\omega = \sum_{I = \{i_1,i_2,\ldots,i_k\}} a_I \;\frac{dx_{i_1}}{x_{i_1}} \wedge \frac{dx_{i_2}}{x_{i_2}} \wedge \cdots \wedge \frac{dx_{i_k}}{x_{i_k}}, \qquad a_I \in \CC.
\]
Every cohomology class in $H^{\ast}(T)$ has a unique standard representative: the vector space of standard forms maps isomorphically onto $H^\ast(T)$.

\begin{lemma}\label{lem:standardrestrict} Let $Z$ be a smooth complex algebraic variety. Let $T$ and $S$
be two dense tori in $Z$. Let $\eta$ be a class in $H^{\ast}(Z)$ and let $\alpha$ and let $\beta$ be the standard representatives of $\eta|_T \in H^{\ast}(T)$ and $\eta|_S \in H^{\ast}(S)$
respectively. Then $\alpha|_{T \cap S} = \beta|_{T \cap S}$.
\end{lemma}
\begin{proof}
Let $t_i$ and $s_i$ be the coordinates on the tori $T$ and $S$. Then the form $\alpha|_{T \cap S} - \beta|_{T \cap S}$ is in the ring generated
by the $\dlog t_i$ and $\dlog s_i$. It represents the cohomology class $\eta|_{T \cap S} - \eta|_{T \cap S}$, which is zero, so the form $\alpha|_{T \cap S} - \beta|_{T \cap S}$ is exact. By Lemma \ref{lem:dlogzero}, $\alpha|_{T
\cap S} - \beta|_{T \cap S} =0$.
\end{proof}

A differential form on $\cA$ is called \defn{standard} if it restricts to a standard form on every cluster torus $T \subset \cA$.  Our next result says that classes in $H^\ast(\cA)_{st}$ can be represented uniquely by standard forms.

\begin{thm}\label{thm:standardequiv}
Suppose that the cluster variety $\cA$ is locally acyclic and of full rank.  Then the following vector spaces are isomorphic, by the natural isomorphisms:
\begin{enumerate}
\item The subring $H^\ast(\cA)_{st}$.
\item The space of standard differential forms on $\cA$.
\item The space of differential forms on $\cA$ which are standard on some cluster torus.
\end{enumerate}
\end{thm}

\begin{proof}
We begin by showing that, if $\alpha$ is a class in $H^\ast(\cA)_{st}$, then $\alpha$ can be represented by a standard form. 
Let $T_1$ and $T_2$ be two cluster tori, and let $\eta_i$ be the standard form on $T_i$ representing the cohomology class $\alpha|_{T_i}$. 
Then, by Lemma~\ref{lem:standardrestrict}, we have $\eta_1|_{T_1 \cap T_2} = \eta_2|_{T_1 \cap T_2}$. 
So there is a single form $\eta$ on $\bigcup T_i$ which represents $\alpha$. Set $U = \bigcup T_i$.
By Proposition \ref{prop:uniontori}, the complement of $U$ is codimension $\geq 2$ in $\cA$, so $\eta$ extends to a differential form on all of $\cA$, which we will also denote by $\eta$. 
We must see that $\eta$ represents the cohomology class $\alpha$. At the moment, we only know that $\eta|_U$ represents $\alpha|_U$.

From Proposition~\ref{prop:rationalLogImpliesTate}, the form $\eta$ represents a class in $H^{k, (k,k)}(\cA)$; call that class $\alpha'$, and we know that $\alpha|_U = \alpha'|_U$.
But, by Lemma~\ref{lem:topHodgeInjective}, $H^{k, (k,k)}(\cA)$ injects into $H^{k, (k,k)}(U)$, so  $\alpha|_U = \alpha'|_U$ implies that $\alpha=\alpha'$, as desired.

Clearly, a standard form is, in particular, standard on some torus. And, by Proposition~\ref{prop:rationalLogImpliesTate}, a form on $\cA$ which is standard on some torus represents a standard class. This shows the equivalence of (1), (2) and~(3). 
\end{proof}
In the sequel, we will identify $H^*(\cA)_{st}$ with the vector space of standard forms on $\cA$.

We will later prove an additional characterization of standard cohomology in Proposition~\ref{prop:neighborTori}.

\subsection{Extension of standard forms to neighboring tori}\label{ssec:neighbor}

Let $T$ be a cluster torus in the cluster variety $\cA$, with mutable variables $x_1$, $x_2$, \dots, $x_n$ and frozen variables $x_{n+1}$, $x_{n+2}$, \dots, $x_{n+m}$. Let $T_1$, $T_2$, \dots, $T_n$ be the neighboring tori. 
In this section, we will develop a criterion for a standard form $\omega$ on $T$ to extend to a differential form on the $T_i$. 
Once we prove Proposition~\ref{prop:neighborTori}, this will give a simple way to compute $H^{\ast}(\cA)_{st}$ by a direct linear algebra computation.
We remark that the results of this section have no hypotheses of local acyclicity; we could simplify some of our arguments if we added them, but we want to emphasize how direct and computational our proofs are.

\begin{proposition}\label{prop:residues}
With the above notations, for $1 \leq r \leq n$,  write $\omega = \omega_1 + \omega_2 \wedge \dlog x_r$, where $\dlog x_r$ does not appear in $\omega_1$ or $\omega_2$. Then $\omega$ extends to $T_r$ if and only if
\begin{equation}\label{eq:residue}
\omega_2 \wedge \sum_{s=1}^{m+n} \tB_{sr} \frac{d x_s}{x_s} =0.
\end{equation}

\end{proposition}

\begin{proof}
Consider the two tori $T$ and $T_r$, glued together along the relation
\[
x_r x'_r = m^+ + m^-,
\]
where $m^{\pm}  = \prod_{s=1}^{m+n} x_s^{[\pm \tB_{sr}]_+}$.
Let $f = m^+ + m^-$.  
On $T_i$, the form $\omega$ can be written as the (possibly meromorphic) form
\[
\omega' = \omega_1 + \omega_2 \wedge \dlog (f /x'_r) = \omega_1 + \omega_2 \wedge  \dlog f - \omega_2 \wedge   \dlog x'_r.
\]

If $\omega$ extends to $T_r$, then the residue $\Res_f \omega' = \omega_2|_{f=0}$ must be zero.  Each component $Z$ of $\{ f = 0 \}$ is a torus, and the restriction map from standard forms on $T_r$ to standard forms on $Z$ has the same kernel as the map $\psi \mapsto \psi \wedge  \dlog (m^{+}/m^{-})$. We deduce that $\omega_2 \wedge \dlog (m^{+}/m^{-})=0$, and we compute $\omega_2 \wedge \dlog (m^{+}/m^{-})=\omega_2 \wedge (\sum_{s=1}^{m+n} \tB_{sr} \dlog x_s) = 0$. 

We also must verify that, in this case, $\omega$ is standard at $T'$. To this end, we note that $f = m_- (1+m_+/m_-)$, so $\dlog f = \dlog m_- + \dlog(m_+/m_-)/(m_-/m_+ +1)$. Thus, under our assumption that $\omega_2 \wedge  \dlog(m_+/m_-)=0$, we have $\omega_2 \wedge \dlog f = \omega_2  \wedge \dlog m_-$ and thus
\begin{equation}
\omega' =  \omega_1 + \omega_2 \wedge  (\dlog m_-  -    \dlog x'_r). \label{mutateStandard}
\end{equation}
This last formula is clearly standard.

We can reverse the argument: the only potential poles of $\omega'$ are simple poles along the zero locus of $f$ and, if $\omega_2 \wedge (\sum_{s=1}^{m+n} \tB_{sr} \dlog x_s) = 0$, the above computation shows that there are no such poles.
\end{proof}

We remark that a variant form of \eqref{mutateStandard} also holds:
\begin{equation}
\omega' =  \omega_1 + \omega_2 \wedge  (\dlog m_+  -    \dlog x'_r). \label{mutateStandardPlus}
\end{equation}
These are equivalent, because \eqref{eq:residue} states that $\omega_2 \wedge (\dlog m_+ - \dlog m_-)=0$.

\begin{prop} \label{keepsExtending}
In the above notation, let $\omega$ be a standard form on $T$ which extends to all $n$ neighboring tori.  Then $\omega$ extends to a standard form on every cluster torus.
\end{prop}

\begin{proof}
Let $T'$ be some other cluster torus, and let $T=T^0$, $T^1$, \dots, $T^{\ell}=T'$ be a minimal length path through the cluster complex from $T$ to $T'$. We will show by induction on $\ell$ that $\omega$ extends to a standard form on $T'$. The base case $\ell=1$ is Proposition~\ref{prop:residues}.

 Let $T^{i+1}$ be related to $T^i$ by mutation in the $k_i$-th variable; note that our assumption that the path is minimal implies that $k_{i-1} \neq k_i$. Let $U$ be the torus obtained by mutating $T^{\ell-2}$ in the $k_{\ell-1}$-th variable. By our inductive assumption, $\omega$ extends to a standard form on $T^{\ell-2}$, $T^{\ell-1}$ and $U$.
 
 Before proceeding we clean up our notation. Reorder our variables so that $k_{\ell-2}=1$ and $k_{\ell-1}=2$. 
 Set $V = T^{\ell-2}$, $V^1 = T^{\ell-1}$, $V^2 = U$ and $V^{12} = T^{\ell} = T'$. 
 So $\omega$ is a standard form on $V$ which extends to $V^1$ and $V^2$, and we must show it extends to $V^{12}$. 
 Let $(x_1, x_2, \ldots, x_{n+m})$ be the cluster variables on $V$, and let $x'_1$ be the cluster variable on $V^1$ obtained by mutating $x_1$. 
 The following figure depicts the portion of the exchange graph we are considering.

\[\xymatrix{
T=T^0 \ar@{-}^(0.6){\mu_{k_0}}[r] & T^1  \ar@{-}^(0.6){\mu_{k_1}}[r] & \cdots  \ar@{-}^(0.4){\mu_{k_{\ell-3}}}[r] & V  \ar@{-}^{\mu_1}[r] & V^1   \ar@{-}^(0.4){\mu_{2}}[r] & V^{12} = T' \\
& & & V^2  \ar@{-}^{\mu_2}[u] & & \\
}\]

 Let the extended exchange matrix at $V$ be $\tB$ and let $\tB'$ be the extended exchange matrix at $V^1$. Without loss of generality, we may assume that $\tB_{12} \geq 0$.
 We also adopt abbreviations for the following $1$-forms:
\begin{align*}
\alpha_1 &= \sum_{r=3}^{n+m} \tB_{r1} \frac{dx_r}{x_r}, \qquad  \alpha^+_1 = \sum_{r=3}^{n+m} [\tB_{r1}]_+ \frac{dx_r}{x_r}, \qquad 
 \alpha_2 = \sum_{r=3}^{n+m} \tB_{r2} \frac{dx_r}{x_r}, \qquad \alpha'_2 = \sum_{r=3}^{n+m} \tB'_{r2} \frac{dx_r}{x_r}.
\end{align*}

%
 
 Write 
 \[ \omega = \omega_1 + \omega_2 \wedge \frac{d x_1}{x_1} + \omega_3 \wedge  \frac{d x_2}{x_2} + \omega_4 \wedge \frac{d x_1}{x_1} \wedge \frac{d x_2}{x_2}  \]
where the $\omega_i$ are standard forms on $V$ not containing $\dlog x_1$ or $\dlog x_2$.

From Proposition~\ref{prop:residues} and the hypothesis that $\omega$ extends to $V^2$, we have
\[ 0 = (\omega_3 + \omega_4 \frac{d x_1}{x_1}) \wedge (\alpha_2 + \tB_{12} \frac{dx_1}{x_1} )   = \omega_3 \wedge \alpha_2 + ( \tB_{12} \omega_3 - \omega_4 \wedge \alpha_2) \wedge \frac{d x_1}{x_1}  .\]
Now, the term $\omega_3 \wedge \alpha_2$ does not contain $\dlog x_1$, and every standard monomial in $ ( \tB_{12} \omega_3 - \omega_4 \alpha_2) \wedge \dlog x_1$ contains $\dlog x_1$. So each of these terms is individually zero, and we deduce that
\begin{equation}
 \omega_3 \wedge \alpha_2 = \tB_{12} \omega_3 - \omega_4 \wedge \alpha_2  = 0 \label{vanishing}
\end{equation}

From \eqref{mutateStandardPlus}, $\omega$ is given on $V^1$ by the formula:
\begin{multline*} \left( \omega_1 + \omega_3 \wedge \dfrac{d x_2}{x_2} \right) + \left( \omega_2 - \omega_4 \wedge \dfrac{dx_2}{x_2} \right) \wedge \left( \alpha_1^+ - \dfrac{dx'_1}{x'_1} \right)=\\
\left( \omega_1 + \omega_2 \wedge \alpha_1^+ - \omega_2 \wedge \dfrac{dx'_1}{x'_1} \right) + \left( \omega_3 + \omega_4 \wedge \alpha^+_1 - \omega_4 \wedge \dfrac{d x'_1}{x'_1} \right) \wedge \dfrac{d x_2}{x_2} . \end{multline*}
(We do not need to include a $[\tB_{21}]_+ \dlog x_2$ term in $d m_+/m_+$, because $\tB_{21} \leq 0$.) By Proposition~\ref{prop:residues}, we must verify that
\[  \left( \omega_3 + \omega_4 \wedge \alpha^+_1 - \omega_4 \wedge \dfrac{d x'_1}{x'_1} \right) \wedge \left(\tB'_{12} \dfrac{d x'_1}{x'_1} + \alpha'_2 \right)=0.  \]
Substituting our formulas for $\tB'$ and $\alpha'_2$, this becomes
\[ \left( \omega_3 + \omega_4 \wedge \alpha^+_1 - \omega_4 \wedge \dfrac{d x'_1}{x'_1} \right) \wedge \left(- \tB_{12} \dfrac{d x'_1}{x'_1} +\alpha_2 + \tB_{12} \alpha^+_1 \right) = 0. \]
Expanding, we need
\[ \omega_3 \wedge \alpha_2 + \left( \tB_{12} \omega_3 - \omega_4 \wedge \alpha_2 \right) \left( \alpha^+_1 - \dfrac{dx'_1}{x'_1} \right) =0.\]
This follows immediately from \eqref{vanishing}.
\end{proof}

We can now prove our last description of $H^{\ast}(\cA)_{st}$:
\begin{prop} \label{prop:neighborTori}
Let $\cA$ be a locally acyclic cluster variety of full rank.
Let $T \subset \cA$ be a cluster torus.  Then the space of standard forms on $\cA$ is naturally isomorphic to the space of standard forms on $T$ that extends to the cluster tori neighboring $T$.
%
\end{prop}

\begin{proof}
If $\omega$ is a standard form on $\cA$, it restricts to a standard form on all cluster tori, so in particular it does on $T$ and on the tori bordering $T$.

Now, suppose that $\omega$ is standard on $T$ and extends to the tori neighboring $T$. By Proposition~\ref{keepsExtending}, $\omega$ extends to all of the cluster tori. 
Let $U$ be the union of all cluster tori. By Proposition \ref{prop:uniontori} and Theorem \ref{thm:Muller},
$\cA$ is smooth and $\cA \setminus U$ is codimension $\geq 2$ in $\cA$, so $\omega$ extends to all of $\cA$; this extension $\tilde{\omega}$ is by definition a standard form on $\cA$. 
\end{proof}

Proposition~\ref{prop:neighborTori} combines with Proposition~\ref{prop:residues} to give:
\begin{cor} \label{cor:LinAlg}
Let $\cA$ be a locally acyclic cluster variety of full rank. Let $T$ be a cluster torus with extended exchange matrix $\tB$, with cluster variables $x_1,x_2,\ldots,x_{n+m}$.
Then $H^{\ast}(\cA)_{st}$ is isomorphic to the space of standard forms $\eta$ on $T$ which for $1 \leq r \leq n$, satisfies
\[ \eta_{2,r} \wedge \sum_{i=1}^{n+m} \tB_{ir} \frac{d x_i}{x_i} = 0
, \]
where $\eta = \eta_{1,r} + \eta_{2,r} \wedge \dlog x_r$ for standard forms $\eta_{1,r}, \eta_{2,r}$ not containing $\dlog x_r$.
%
\end{cor}

We note that the above condition is a system of linear equations on the coefficients of $\eta$ in the basis of standard monomials. In the authors' experience, solving these equations is the fastest way to compute $H^{\ast}(\cA)_{st}$.

\subsection{Generators and Poincar\'{e} series of the standard part}\label{ssec:generators}
There are two obvious classes of standard forms on $\cA$: the GSV forms which we introduced in Section \ref{sec:GSV}, and the forms $\dlog x_i$ for $x_i$ a frozen variable. Each of these is manifestly a $\log$-form on the initial torus, and they extend to regular differential forms on the whole of $\cA$.
Our main result in this section is that they are ring generators for the standard cohomology.

Suppose the connected components of $\Gamma = \Gamma(\tB)$ are $\Gamma_1,\ldots,\Gamma_r$.  Then a GSV-form $\gamma$ for $\Gamma_a$ is the differential form associated to the extended exchange matrix $\tB'$ consisting of the columns of $\tB$ indexed by vertices of $\Gamma_a$.  Such $\gamma$ are regular on $\cA(\tB)$ by an application of Proposition \ref{keepsExtending}.

\begin{theorem}\label{thm:standard}
Suppose that $\cA$ is locally acyclic and of full rank.  Then $H^{\ast}(\cA)_{st}$ is generated by GSV 2-forms $\gamma_1,\ldots,\gamma_r$ for the connected components $\Gamma_1,\ldots,\Gamma_r$ of $\Gamma$, and the 1-forms $\dlog y_1, \dlog y_2,\ldots,\dlog y_m$, where $y_i$ are the frozen variables. If the connected components of $\Gamma$ have cardinalities $n_1$, $n_2$, \dots, $n_r$, then the Poincar\'{e} series of $H^{\ast}(\cA)_{st}$ is
\[ P(H^{\ast}(\cA)_{st},t) = (1+t)^{m-r} \prod_{i=1}^r \left( 1+ t+ \cdots + t^{n_i} + t^{n_i+1} \right) . \]
\end{theorem}  

We invite the reader to compare this formula to the top rows of the tables in Section~\ref{sec:examples}.

%

We now state a more particular result when $A$ has principal coefficients and $\Gamma$ is connected.
\begin{prop}\label{prop:standardbasis}
Suppose that $\cA$ is locally acyclic and has principal coefficients, and $\Gamma$ is connected. Then a basis for $H^{\ast}(\cA)_{st}$ is the following list of differential forms:
\[ \gamma^j \wedge \bigwedge_{i \in I}  \frac{dy_i}{y_i} \quad \mbox{for} \quad j+\#(I) \leq n\]
where $\gamma$ is a fixed GSV form,  $I$ may be any subset of $[n]$ and the $y_i$ are the frozen variables.
\end{prop}
The proof of Proposition~\ref{prop:standardbasis} is delayed to Section \ref{sec:proofstandardbasis}. 

\begin{lemma}\label{lem:standardquotient}
Suppose that $(\Psi,\Phi,\{R_t\}): \cA' \to \cA$ is a covering map of locally acyclic cluster varieties of full rank.  Then we have $H^{\ast}(\cA)_{st} = H^{\ast}(\cA')_{st}$.  Furthermore, GSV-forms and $\dlog$-forms of $\cA$ pullback to GSV-forms and $\dlog$-forms of $\cA'$.
\end{lemma}
\begin{proof}
The map $\Psi$ sends a cluster torus $T'$ of $\cA'$ to a cluster torus $T$ of $\cA$ via a monomial map with finite abelian kernel.  Under such a map, pullback induces an isomorphism between standard forms on $T$ and standard forms on $T'$ (and in turn also an isomorphism $H^\ast(T) \cong H^\ast(T')$).  These isomorphisms are compatible for all cluster tori, and by Theorem \ref{thm:standardequiv}, we obtain an isomorphism $H^{\ast}(\cA)_{st} = H^{\ast}(\cA')_{st}$.  The last statement of the lemma also follows immediately (cf. Proposition \ref{prop:GSVpullback}).
%
%
\end{proof}

 Proposition~\ref{prop:standardbasis} implies our main result.
\begin{proof}[Proof of Theorem~\ref{thm:standard} assuming Proposition~\ref{prop:standardbasis}]
We compute the generating function for the list of differential forms in Proposition \ref{prop:standardbasis}, counted by degree.  
Letting $i = \#I$, we have
\begin{multline*}
\sum_{i+j \leq n} \binom{n}{i} t^{i+2j}  = \sum_{i=0}^n \binom{n}{i} \left(t^i + t^{i+2} + \cdots + t^{2n-i} \right) 
= \frac{1}{t^2-1} \sum_{i=0}^n \binom{n}{i} ( t^{2n+2-i} - t^i ) = \\
 \frac{1}{t^2-1} \left(  t^{n+2} (t+1)^n - (t+1)^n \right) = \frac{(t^{n+2} - 1) (t+1)^n}{(t-1)(t+1)} = (t+1)^{n-1} (t^{n+1} + t^n + \cdots + t + 1 ).
\end{multline*}
Thus the theorem holds for cluster varieties $\cA$ with principal coefficients and $\Gamma$ connected.  By taking products of the cluster varieties, it holds for any $\cA$ with principal coefficients.  By Lemma \ref{lem:standardquotient}, the theorem holds for any $\cA$ with $d$-principal coefficients, for any integer $d \geq 1$.  By taking products with a torus, the theorem holds for any $\cA$ with extended exchange matrix of the form $\left(\begin{smallmatrix}B \\d \Id_n \\ 0\end{smallmatrix}\right)$.  Finally, by Lemma \ref{lem:standardquotient} and Proposition \ref{prop:coverprincipal}, the theorem holds for any $\cA$ with full rank $\tB$.
\end{proof}

\begin{lem} \label{lem:standardIndep}
Under the hypotheses of Proposition~\ref{prop:standardbasis}, the differential forms in Proposition~\ref{prop:standardbasis} are linearly independent.
\end{lem}

\begin{proof}
We adopt the following notational shorthands: we write $\theta_i$ for $\dlog x_i$ and $\eta_i$ for $\dlog y_i$. For $I \subseteq [n]$, we write $\theta_I$ and $\eta_I$ for $\bigwedge_{i \in I} \theta_i$ and $\bigwedge_{i \in I} \eta_i$ respectively.  Let $\bgamma = \sum_{i,j=1}^n B_{ij} \theta_i \wedge \theta_j + \sum_i 2 \theta_i \wedge \eta_i$.  Thus $\bgamma$ differs from the GSV-form $\gamma$ by terms of the form $\eta_i \wedge \eta_j$.  It suffices to prove the statement with $\bgamma$ instead of $\gamma$.

If $\omega$ is a standard differential form on the torus of degree $k$ then we can write it as $\sum_{\#I + \# J = k} c_{IJ} \theta_I \wedge \eta_J$. 
We will consider the \newword{leading part} of $\omega$ to be the terms which minimize $\# I$.  That is, the leading terms are the ones with the most frozen variables and the fewest mutable variables.

The leading term of $\bgamma^r \wedge \eta_I$ is $\sum_{J \in \binom{[n] \setminus I}{r}} \theta_J \wedge \eta_{I \cup J}$.
Since $r+\#I \leq n$, the sum is not empty.  It is easy to check that these leading terms are linearly independent. 
\end{proof}

\begin{remark} We do not have a conceptual explanation for the factorization of the Poincar\'{e} series. 
When $n$ is even, it is tempting to rewrite the formula as $(t+1)^n (t^n+t^{n-2} + \cdots + t^2 + 1)$ and guess that a basis should be given by the set of differential forms of the form $\bgamma^r \wedge d \log y_{i_1} \wedge d \log y_{i_2} \wedge \cdots \wedge d \log y_{i_k}$ with $r \leq n/2$. 
However, this is wrong: in type $D_4$ with principal coefficients, we have $\bgamma^2 \wedge d \log y_1 \wedge d \log y_2 \wedge d \log y_3=0$, where $1$, $2$ and $3$ are the leaves of the $D_4$ diagram.
\end{remark}

\subsection{Standard cohomology when the quiver is a path}

The aim of this section is to prove Theorem~\ref{thm:standard} and Proposition~\ref{prop:standardbasis} when $\Gamma$ is a path.

\begin{prop}\label{prop:An}
Suppose that $\cA = \cA(\tB)$ is a cluster variety of full rank $n$ with $m$ frozen variables $y_1$, $y_2$, \dots, $y_m$ and that the graph $\Gamma(B)$ is a path. 
Then $H^{\ast}(\cA)_{st}$ is generated by the GSV-form $\gamma$ and by the forms $\dlog y_i$. The Poincar\'{e} series of $H^{\ast}(\cA)_{st}$ is given by
\[ (1+t)^{m-1} (1+t+t^2+\cdots + t^n+t^{n+1}) . \]
\end{prop}

%
%


\begin{proof}[Proof of Proposition \ref{prop:An}]
By taking products with a torus, and using Lemma \ref{lem:standardquotient} and Proposition \ref{prop:coverprincipal}, we may assume that $\tB$ has the form $\tB = \left( \begin{smallmatrix} B \\ d\,\Id_n \\ 0 \end{smallmatrix} \right)$.  By removing torus factors, we may further reduce to the case $\tB = \left( \begin{smallmatrix} B \\ d\,\Id_n\end{smallmatrix} \right)$.

We proceed by induction on the number $n$ of vertices in $\Gamma$.  
The base case $n = 1$ is  Proposition~\ref{prop:isolatedrank1}.  Let $x_1$ be the cluster variable corresponding to one of the degree one vertices of $\Gamma$, and let $x_2$ be its neighbor.  Let $U = \{x_1 \neq 0\}$ and $V = \{x_2 \neq 0\}$.  Consider the Mayer-Vietoris sequence
\[
\cdots \to H^{k-1}(U \cap V) \to H^{k}(\cA) \to H^{k}(U) \oplus H^{k}(V) \to H^{k}(U \cap V)\to H^{k+1}(\cA) \to \cdots
\]

Taking the $(k,k)$ component for the Deligne splitting, we have an exact sequence
\[
0 \to H^k(\cA)_{st} \to H^k(U)_{st} \oplus H^k(V)_{st} \to H^k(U \cap V)_{st} \to H^{k+1, (k,k)}(\cA) \to \cdots 
\]

Our first task is to show, by a direct computation, that $H^k(U)_{st} \oplus H^k(V)_{st} \to H^k(U \cap V)_{st}$ is surjective, so we can cut off this sequence.  Let $\gamma = \sum_{i,j=1}^{n+m} \widehat{B}_{ij} \dlog x_i \wedge \dlog x_j$ denote a GSV form for $\cA$.  The quiver $\Gamma_V$ of the cluster variety $V$ is a disjoint union of the singleton $\{1\}$ and the path $\{3,4,\ldots,n\}$.  Choose GSV forms 
\[ \gamma^{(1)}_V = 2\widehat{B}_{12} \frac{dx_1}{x_1} \wedge \frac{dx_2}{x_2} + \text{ other terms}\] and 
\[ \gamma^{(2)}_V = 2\widehat{B}_{23} \frac{dx_2}{x_2} \wedge \frac{dx_3}{x_3}+ 2\widehat{B}_{34}\frac{dx_3}{x_3} \wedge \frac{dx_4}{x_4}+ \cdots + 2\widehat{B}_{n-1,n} \frac{dx_{n-1}}{x_{n-1}}\wedge \frac{dx_n}{x_n} + \text{ other terms}\] for these components so that $\gamma|_V = \gamma^{(1)}_V + \gamma^{(2)}_V$.  Here, the ``other terms" involve the frozen variables $y_1,y_2,\ldots,y_m$; the variable $x_1$ does not appear in $\gamma^{(2)}_V$ while the variables $x_3,\ldots,x_n$ do not appear in $\gamma^{(1)}_V$.
%


The form $\gamma$ restricts to a GSV-form on $U \cap V$.  By induction, $H^{\ast}(U \cap V)_{st}$ is spanned by 
\begin{equation}
\gamma^r \wedge \left(\frac{d x_1}{x_1}\right)^{\epsilon_1} \wedge \left(\frac{d x_2}{x_2}\right)^{\epsilon_2}  \wedge \bigwedge_{i \in I}  \frac{dy_i}{y_i} \label{spanninglist}
\end{equation}
where $I$ is a subset of the frozen variables, and $\epsilon_1$, $\epsilon_2 \in \{0,1\}$.

We also induct on $k$, so we may assume that $H^{k-1}(U)_{st} \oplus H^{k-1}(V)_{st} \to H^{k-1}(U \cap V)_{st}$ is surjective. (The base case $k=0$ is trivial.) 
Therefore, any class in $H^k(U \cap V)_{st}$ which is in $\sum_i (\dlog y_i) H^{k-1}(U \cap V)_{st}$ is inductively known to be in the image of $H^k(U)_{st} \oplus H^k(V)_{st}$. 
Modulo $\sum_i (\dlog y_i) H^{k-1}(U \cap V)_{st}$, there are only two forms in the list~\eqref{spanninglist} which we must consider: 
\[ 
\begin{cases} 
\gamma^{k/2} \ \mbox{and}\ \dlog x_1 \wedge \dlog x_2 \wedge \gamma^{k/2-1} & \mbox{if $k$ is even,} \\
\dlog x_1 \wedge \gamma^{(k-1)/2} \ \mbox{and}\ \dlog x_2 \wedge \gamma^{(k-1)/2} & \mbox{if $k$ is odd.}  \\
\end{cases} \]
The forms $\gamma^{k/2}$ and $\dlog x_2 \wedge \gamma^{(k-1)/2}$ are clearly images of classes in $H^{\ast}(V)$.
%

We have
\[
\frac{dx_1}{x_1} \wedge\frac{dx_2}{x_2} \wedge \gamma^{k/2-1}  = \pm \frac{1}{2 \widehat{B}_{12}} \gamma^{(1)}_V \gamma^{k/2-1}  \mod \sum_i \frac{dy_i}{y_i} H^{\ast}(U \cap V)
\]
and $\gamma^{(1)}_V\wedge \gamma^{k/2-1} \in H^{\ast}(V)$.  Similarly 
\[
\frac{dx_1}{x_1} \wedge (\gamma^{(2)}_{V})^{(k-1)/2}   = \frac{dx_1}{x_1} \wedge \gamma^{(k-1)/2} \mod \sum_i \frac{dy_i}{y_i} H^{\ast}(U \cap V)
\]
and $\dlog x_1 \wedge \gamma^{(k-1)/2} \in H^{\ast}(U)$. We have now shown that $H^{\ast}(U)_{st} \oplus H^{\ast}(V)_{st} \to H^{\ast}(U \cap V)_{st}$ is surjective.

So we have a short exact sequence:
\[
0 \to H^k(\cA)_{st} \to H^k(U)_{st} \oplus H^k(V)_{st} \to H^k(U \cap V)_{st} \to 0 .
\]

Now, by induction, the dimensions of $H^k(U)_{st}$, $H^k(V)_{st}$ and $H^k(U \cap V)_{st}$ are the coefficients of $t^k$ in $(1+t)^{n+1-1}(1+ t+ \cdots + t^{n-1+1})$, $(1+t)^{n+1-2} (1+t+t^2) (1+t+ \cdots + t^{n-2+1})$ and $(1+t)^{n+2-1} (1+t+ \cdots + t^{n-2+1})$ respectively. So the dimension of $H^k(\cA)$ is the coefficient of $t^k$ in
\begin{multline*} 
(1+t)^n (1+t+\cdots + t^n) + (1+t)^{n-1} (1+t+t^2) (1+t+\cdots+t^{n-1}) - (1+t)^{n+1} (1+t+\cdots+t^{n-1}) \\ = (1+t)^{n-1} (1+t+\cdots + t^{n+1})
\end{multline*} 

Finally, let $R$ be the subring of $H^{\ast}(\cA)$ generated by the GSV form and the $\dlog y_i$. By  Lemma~\ref{lem:standardIndep}, the degree $k$ part of $R$ has dimension at least the coefficient of $t^k$ in $ (1+t)^{n-1} (1+t+\cdots + t^{n+1})$. So we conclude that $H^{\ast}(\cA)$ is generated by the GSV form and the $\dlog y_i$, and that it has Poincar\'{e} series $(1+t)^{n-1} (1+t+\cdots + t^{n+1})$.
\end{proof}

\begin{prop} \label{prop:pathReallyFullRank}
Continue all the hypotheses of Proposition~\ref{prop:An}.
Suppose in addition that $\tB$ is really full rank, then $H^{\ast}(\cA)_{st} = H^{\ast}(\cA)$.
\end{prop}

\begin{proof}
We prove our result by induction on $n$; the base case $n=1$ is Corollary~\ref{cor:isolatedreallyfullrank}.  Applying the reductions in the first paragraph of the proof of Proposition~\ref{prop:An}, we can assume that $\tB$ has principal coefficients.

Then $U$, $V$ and $U \cap V$ are each a product of path algebras associated with paths where $\tB$ has really full rank so, by induction, we can assume that $H^{\ast}(U)_{ns}$, $H^{\ast}(V)_{ns}$ and $H^{\ast}(U \cap V)_{ns}$ are all zero. For $q \leq k-2$, in the Mayer-Vietoris sequence, the term $H^{k, (q,q)}(\cA)$ is sandwiched between $H^{k-1, (q,q)}(U \cap V)$ and $H^{k, (q,q)}(U) \oplus H^{k, (q,q)}(V)$, which are both zero by induction. In the case $q=k-1$, we have an exact sequence 
\[ \begin{array}{rcl@{}c@{}lcl}
& & H^{k-1, (k-1,k-1)}(U) &\oplus& H^{k-1, (k-1,k-1)}(V)& \to& H^{k-1, (k-1,k-1)}(U \cap V) \\ \to H^{k, (k-1, k-1)}(\cA) &\to& H^{k, (k-1, k-1)}(U) &\oplus&  H^{k, (k-1, k-1)}(V) & & \\
\end{array}.\]
The last term is zero by induction, and the map on the top line is surjective as we established earlier in the proof, so $H^{k, (k-1,k-1)}(\cA)=0$ as desired.
\end{proof}

\begin{cor} \label{cor:pathbasis}
Continue the hypotheses of Proposition~\ref{prop:An}. 
Suppose that $\tB = \left( \begin{smallmatrix} B \\ \mathrm{Id} \\ C \end{smallmatrix} \right)$ is obtained from the principal coefficient exchange matrix by adding additional frozen rows. 
Let $x_1$, $x_2$, \ldots, $x_n$ be the cluster variables, let $y_1$, $y_2$, \ldots, $y_n$ be the frozen variables from the principal coefficient subalgebra, and let $v_1$, $v_2$, \ldots, $v_{\ell}$ denote the additional frozen variables. Then
\[ \gamma^r \bigwedge_{i \in I}  \frac{d y_i }{y_i} \bigwedge_{j \in J}  \frac{d v_j} {v_j} \quad \mbox{for} \quad r+\#(I)\leq n\]
for $I \subset [n]$, and $J \subset [\ell]$ forms a basis for $H^{\ast}(\cA)$.
\end{cor}

\begin{proof}
By Proposition~\ref{lem:standardIndep}, the given forms are linearly independent. The computation in the proof that  Proposition~\ref{prop:standardbasis} implies Theorem~\ref{thm:standard}, shows that the number of these forms matches the Betti numbers we computed in Proposition~\ref{prop:An}.
\end{proof}

\subsection{Proof of Proposition \ref{prop:standardbasis}} \label{sec:proofstandardbasis}
We once again adopt the following notational shorthands: we write $\theta_i$ for $\dlog x_i$ and $\eta_i$ for $\dlog y_i$. For $I \subseteq [n]$, we write $\theta_I$ and $\eta_I$ for $\bigwedge_{i \in I} \theta_i$ and $\bigwedge_{i \in I} \eta_i$ respectively.  Let $\bgamma = \sum_{i,j=1}^n B_{ij} \theta_i \wedge \theta_j + \sum_i 2 \theta_i \wedge \eta_i$.  It suffices to prove the claim with the GSV form $\gamma$ replaced by the reduced GSV form $\bgamma$.

We showed in Lemma~\ref{lem:standardIndep} that the proposed basis is linearly independent, so we must now show it spans.  Recall that we defined the leading part of a standard differential form in the proof of Lemma \ref{lem:standardIndep}.  Let $\omega = \sum c_{IJ} \theta_I \eta_J$ be a standard differential form on $\cA$. We will show that $\omega$ can be written as a linear combination of the forms $\bgamma^r \wedge \eta_I$, with $r+\# I \leq n$. We may assume that $\omega$ is homogeneous, say of degree $d$.

Let the leading part of $\omega$ (in the sense of Proposition~\ref{ssec:generators}) be $\omega_0$; suppose that $\omega_0$ has degree $r$ in the mutable terms and degree $d-r$ in the frozen terms. 

\begin{lemma}
If $\theta_I \wedge \eta_J$ occurs with nonzero coefficient in $\omega_0$, then $I \subseteq J$.
\end{lemma}
\begin{proof}
Suppose otherwise, with $i \in I \setminus J$. Note that no term of the form  $\theta_{I \setminus j} \wedge \eta_{J \cup i}$ can occur in $\omega$ as, if it did, it would lie in the leading part rather than $\theta_I \wedge \eta_J$. Without such terms, $\omega$ cannot satisfy \eqref{eq:residue} of Proposition \ref{prop:residues} for $r = i$.  More specifically, writing $\omega = \omega_1 + \omega_2 \wedge \theta_i$, we see that the term $\theta_I \wedge \eta_J$ in $\omega$ contributes a term $\omega_{I \setminus i} \eta_{J \cup i}$ to the form $\alpha = \omega_2 \wedge \sum_{s=1}^{2n} \tB_{si} \dlog x_s$. 
But the only other way $\theta_{I \setminus i} \eta_{J \cup i}$ can occur in $\alpha$ is if $\omega$ contained the term $\theta_{I \setminus j} \wedge \eta_{J \cup i}$ for some $j \neq i$.
\end{proof}

So we can write $\omega_0$ as
\[ \omega_0 = \sum_{I \cap J = \emptyset} d_{IJ} \  \theta_I \wedge \eta_{I \cup J} \]
for some coefficients $d_{IJ}$. We will need to be careful about sign conventions for the wedge products: We assume that the $\dlog x_i$ for $i \in I$ are ordered in the same way within $\theta_I$ and $\eta_I$ and that the $\dlog x_j$ for $j \in J$ are ordered within $\eta_{I \cup J}$ after the $\dlog x_i$ for $i \in I$.

\begin{lemma}\label{lem:dIJ}
With the above notation, we have $d_{IJ} = d_{I' J}$, for any $I$ and $I'$ disjoint from $J$.
\end{lemma}

\begin{proof}
Our proof is by induction on $\# (I \setminus I')$. The base case, $\# (I \setminus I')=0$, means that $I=I'$ so the statement is trivial.

Let $I \cap I' = \{ j_1, j_2, \ldots, j_s \}$. Let $I \setminus I' = \{ i_1, i_2, \ldots, i_{r-s} \}$ and $I' \setminus I = \{ i'_1, i'_2, \ldots, i'_{r-s} \}$.

For two indices $k$ and $\ell$ in $[n]$, let $\delta(k, \ell)$ be the length of the shortest path through $\Gamma$ from $k$ to $\ell$. (Here is where we are using that $\Gamma$ is connected.) 
Reorder $I$ and $I'$ such that $\delta(i_1, i'_1)$ is minimal among all pairs $\delta(i_p, i'_q)$.
Let $\Gamma' \subseteq \Gamma$ be a minimal length path from $i_1$ to $i'_1$. 
In particular, none of the interior vertices of $\Gamma'$ are in $I \setminus I'$ or $I' \setminus I$. Also, there are no edges connecting two vertices in $\Gamma'$ other than the edges in the path $\Gamma'$.

We first present our argument in the case that $\Gamma'$ does not contain any of the elements of $I \cap I'$.
Consider the cluster variety $\cA'$ obtained from freezing all of the variables not on the path $\Gamma'$, with extended exchange matrix $\tB'$ obtained from taking the columns of $\tB$ indexed by $\Gamma'$.

By Corollary~\ref{cor:pathbasis} applied to $\cA'$, we may write
\[
\omega|_{\cA'} = \beta_1  + \bgamma' \wedge \beta_2 + (\bgamma')^2 \wedge \delta
\]
where $\beta_1$ and $\beta_2$ do not involve $\theta_i$ for $i \in \Gamma'$, and in addition $\beta_2$ has no term using all the $\{\eta_i \mid i \in \Gamma'\}$.  Note that the terms $\theta_I \wedge \eta_{I \cup J}$ and $\theta_{I'} \wedge \eta_{I' \cup J}$ of interest to us can only occur in $\bgamma' \wedge \beta_2$ since they use exactly one of the $\{\theta_i \mid i \in \Gamma'\}$.  Consider a monomial $\theta_K \wedge  \eta_{K \cup L}$ using exactly one of the $\{\theta_i \mid i \in \Gamma'\}$.  Then $\theta_K \wedge  \eta_{K \cup L}$ belongs to the leading part of $\bgamma' \theta_{K \setminus i} \wedge \eta_{K \cup L \setminus i}$, and  $\theta_K \wedge  \eta_{K \cup L}$ belongs to the leading part of only this term in $\bgamma' \wedge \beta_2$.  Furthermore, all monomials occurring in the leading part of $\bgamma' \wedge \beta_2$ use exactly one of the $\{\theta_i \mid i \in \Gamma'\}$.  In particular, the leading parts of $\bgamma' \wedge \beta_2$ are terms in $\omega_0$, and cannot be canceled by terms from $\beta_1$ or $(\bgamma')^2 \wedge \delta$.

Since our term $\theta_I\eta_{I \cup J}$ belongs to $\omega_0$, we see that $\beta_2$ cannot contain a term of the form $\theta_{I \setminus \{i_1,j\}} \wedge \eta_{I \cup J}$, for otherwise $\omega_0$ would contain $\theta_K \wedge \eta_{K \cup L}$ where $|K| < r$.  Thus the term $\theta_I  \wedge \eta_{I \cup J}$ can only come from a term $\theta_{I\setminus i_1} \wedge  \eta_{I \cup J \setminus \{ i_1 \} }$ in $\bgamma \wedge \beta_2$.  But $\theta_{I'} \wedge \eta_{I' \cup J}$ occurs in $\bgamma \wedge \theta_{I\setminus\{ i_1 \} } \wedge  \eta_{I \cup J \setminus\{ i_1 \}}$ with the same coefficient, so $d_{IJ} = d_{I'J}$.

Now we consider the case that some of the elements of $I \cap I'$ lie on the path $\Gamma'$. Re-indexing as necessary, let's say that $\Gamma'$ starts at $i_1$ and passes through $\ell_1$, $\ell_2$, \ldots, $\ell_t$ before ending at $i'_1$.  Repeatedly use the freezing argument of the preceding argument, first on the subpath of $\Gamma'$ from $\ell_t$ to $i'_1$, then on the subpath from $\ell_{t-1}$ to $\ell_t$, and so forth, with the final application being on the subpath from $i_1$ to $\ell_1$. The effect is to show that, without changing the coefficient, we can change $I$ to $I  \cup \{ i'_1 \} \setminus \{ \ell_t \}$, to $I \cup \{ i'_1 \} \setminus \{ \ell_{t-1} \} $ and so forth, until we finally reach $I \cup \{ i'_1 \} \setminus \{ i_1 \}$. This has one more element in common with $I'$, so now we are done by induction as above.
\end{proof}

By Lemma \ref{lem:dIJ}, we can express $\omega_0$ in the form
\[ \omega_0 = \sum_{J} d_J \left( \sum_{I \in \binom{[n] \setminus J}{r}} \theta_I \wedge \eta_I\right) \wedge \eta_J \]
for some constants $d_J$. Note that $r + \#J \leq n$, since $r = \#I$ and $I \subset [n] \setminus J$.

Take $\omega' = ((-1)^{r-1}/2^r)\sum_J d_J \bgamma^r \eta_J$. Then $\omega'$ has the same leading term as $\omega$ does. So $\omega - \omega'$ has a leading term which has lower degree in the mutable variables than $\omega$ does. By induction, $\omega - \omega'$ is a linear combinations of terms of the form $\bgamma^k \eta_J$ with $k+\#J \leq n$. Since $\omega'$ is defined to be a linear combination of such terms as well, we see that $\omega$ is a linear combination of terms of the form $\bgamma^k \eta_J$ with $k + \#J \leq n$.  This completes the proof of Proposition~\ref{prop:standardbasis}. \hfill \qedsymbol

\section{Point counts for cluster varieties over finite fields}
\label{sec:pointcounts}

\subsection{Cluster varieties in finite characteristic}
Cluster varieties are naturally defined over $\ZZ$.  In this section, we consider cluster varieties in finite characteristic.  We let $\oF_p$ denote the algebraic closure of the finite field $\FF_p$, and let $\cA_{\oF_p}$ denote the base change of $\cA_\ZZ$ to $\oF_p$.  Write $\cA(\FF_q)$ for the set of $\FF_q$-points of $\cA$.  We refer the reader to \cite{BMRS} for more on cluster varieties in finite characteristic.  Note that in \cite{Mul}, local acyclicity is defined for cluster algebras over $\ZZ$. 

\begin{thm}\label{thm:finitesmooth}
Suppose that $\cA_\ZZ$ is locally acyclic and of really full rank.  Then $\cA_{\oF_p}$ is smooth for all primes $p$.
\end{thm}
\begin{proof}
In \cite[Section 7]{Mul}, Muller shows that $\cA_\QQ$ is regular by computing Jacobian matrices of the natural presentations of isolated cluster varieties.  The really full rank condition ensures that these Jacobian matrices also have the expected rank over a field of finite characteristic.
\end{proof}
\subsection{Dirichlet characters}
A \defn{Dirichlet character} of modulus $n$ is a function $\chi: \ZZ \to \CC^\ast$ satisfying:
\begin{enumerate}
\item for all $a \in \ZZ$, we have $\chi(a+n) = \chi(a)$,
\item if $\gcd(a,n) > 1$, then $\chi(a) = 0$, 
\item if $\gcd(a,n) = 1$, then $\chi(a) \neq 0$, and
\item for all $a,b, \in \ZZ$, we have $\chi(ab) = \chi(a)\chi(b)$.
\end{enumerate}
The product of two Dirichlet characters $\chi_1$ and $\chi_2$ of modulus $n_1$ and $n_2$ is a Dirichlet character of modulus $\lcm(n_1,n_2)$.  If $\chi$ is a Dirichlet character of modulus $n$, we write $\chi^{-1}$ for the Dirichlet character of modulus $n$, satisfying $\chi^{-1}(a) = 1/\chi(a)$ for $\gcd(a,n) = 1$.

For any $n$, we write $\Dir(n)$ for the set of Dirichlet characters modulo $n$. Note that $\#\Dir(n) = \phi(n)$, where $\phi$ is the Euler totient function.
If $d$ divides $n$, we also consider the elements of $\Dir(d)$ as functions on $\ZZ/n\ZZ$ by the composition $\ZZ/n \ZZ \to \ZZ/d \ZZ \to \CC$. 
Let $\Dir^{\ast}(n)$ be the multiset $\bigsqcup_{d |n} \Dir(d)$. Note that $\# \Dir^{\ast}(n) = \sum_{d|n} \phi(d) = n$.

The following lemma is standard:
\begin{lem} \label{charSum}
For any positive integer $d$ and any integer $q$, we have
\[ \sum_{\chi \in \Dir(d)} \chi(q) = \begin{cases} \phi(d) & q \equiv 1 \bmod n \\ 0 & \mbox{otherwise} \end{cases} . \]
\end{lem}

We will want the variant:
\begin{lem} \label{charSum*}
For any positive integer $n$ and any integer $q$, we have
\[ \sum_{\chi \in \Dir^*(n)} \chi(q) = \gcd(q-1,n) . \]
\end{lem}

\begin{proof}
The sum is  $\sum_{d|n} \sum_{\chi \in \Dir(d)} \chi(q) = \sum_{d|\gcd(n,q-1)} \phi(d) = \gcd(q-1,n)$.
\end{proof}

\subsection{Point counts and Frobenius eigenvalues}\label{ssec:eigenvalues}

The set $\{ f: \{ \mbox{prime powers} \} \to \CC\}$ of all complex valued functions on the set of prime powers has a ring structure obtained by pointwise addition and multiplication.  Let $\Lambda$ denote the subring generated over $\ZZ$ by the functions $q \mapsto q$ and the functions $q \mapsto \chi(q)$, for Dirichlet characters $\chi$.  Since the product of two Dirichlet characters is again a Dirichlet character, every element $f \in \Lambda$ has the form
\[
f(q) = \sum_r a_r \chi_r(q) q^{b_r}
\]
where $a_r \in \ZZ$ and $b_r \in \ZZ_{\geq 0}$ are integers.  The aim of this section is to prove the following two results.

\begin{theorem}\label{thm:eigenvalues}
Suppose that $\cA$ satisfies the Louise property and is of full rank.  Let $\ell$ be a prime.  Then there exists a list of nonnegative integers $s_1,s_2,\ldots,s_r$, and Dirichlet characters $\chi_1,\chi_2,\ldots,\chi_r$, such that for sufficiently large primes $p$, the eigenvalues of Frobenius on the $\ell$-adic cohomology $H^\ast(\cA_{\overline{\FF_{p}}},\QQ_\ell)$ are given by
\[
\chi_1(p)p^{s_1}, \ldots, \chi_r(p)p^{s_r}.
\]
If, in addition, $\cA$ is of really full rank, then for all primes $p \neq \ell$, the eigenvalues of Frobenius are all powers of $p$.
\end{theorem}

\begin{theorem}\label{thm:pointcount}
Suppose that $\cA$ satisfies the Louise property and is of full rank.  Then for sufficiently large primes $p$, and $q= p^ a$ a prime power,
the function $q \mapsto \#\cA(\FF_q)$ agrees with an element of the ring $\Lambda$.  
If, in addition, $\cA$ is of really full rank, then for all prime powers $q$, the function $q \mapsto \#\cA(\FF_q)$ is a polynomial in $q$.
\end{theorem}

Some work has been done on point counts of cluster varieties.  For instance, see \cite{Cha, Cha2, Mor}.

\begin{proof}[Proof of  Theorem \ref{thm:pointcount} from Theorem \ref{thm:eigenvalues}]
Let $\cA$ have rank $n$ and $m$ frozen variables.  The variety $\cA_\QQ$ is regular \cite{Mul}.  It follows that, for large primes $p$, the variety $\cA_{\oF_p}$ is smooth.  We assume that $p$ is sufficiently large that this is the case.  
 By the Grothendieck--Lefschetz fixed point formula \cite[(1.5.1)]{DelWeil}, we have
\begin{equation}\label{eq:Groth}
\#\cA(\FF_{p^a}) = \sum_{k=n+m}^{2n+2m} (-1)^k {\rm Tr}(F^a,H^k_{c}(\cA_{\oF_p},\QQ_\ell)) = \sum_{k=n+m}^{2n+m} (-1)^k \sum_{i=1}^{\dim H^k_{c}(\cA) }  \alpha_i^a
\end{equation}
where the $\alpha_i$ are the eigenvalues of the Frobenius $F$ on the compactly supported $\ell$-adic cohomology $H^k_{c}(\cA_{\oF_p},\QQ_\ell)$.  
Applying Poincar\'{e} duality (\cite[Theor\`eme 2.8]{DelWeil}), we have a perfect pairing between compactly supported $\ell$-adic cohomology and usual $\ell$-adic cohomology:
\[
H^k_c(\cA_{\oF_p},\QQ_\ell) \times H^{2n+2m-k}(\cA_{\oF_p},\QQ_\ell) \to \QQ_\ell(-n-m).
\]   In particular, if the eigenvalues of Frobenius on $H^k_c(\cA_{\oF_p},\QQ_\ell)$ are $\alpha_1,\alpha_2,\ldots,\alpha_r$, then the eigenvalues of Frobenius on $H^{2n+2m-k}(\cA_{\oF_p},\QQ_\ell)$ are $p^{n+m}/\alpha_1$, \dots, $p^{n+m}/\alpha_r$.  Let the eigenvalues of Frobenius on $H^k(\cA_{\oF_p},\QQ_\ell)$ be given by $\epsilon_{k1} p^{w_{k1}}$, $\epsilon_{k2} p^{w_{k2}}, \dots$ and let $p$ be sufficiently large. Then, using Theorem~\ref{thm:eigenvalues}, we obtain
\begin{equation}\label{eq:GrothDual}
\#\cA(\FF_{q}) =\sum_{k=0}^{n+m} (-1)^{k} \sum_{r=1}^{\dim H^k(\cA)} \epsilon_{kr}^{-a} p^{a(n+m-w_{kr})} = \sum_{k=0}^{n+m} (-1)^k \sum_{r=1}^{\dim H^k(\cA)} \chi^{-1}_{kr}(q) q^{n+m-w_{kr}}
\end{equation}
for some characters $\chi_{kr}$.
 Theorem \ref{thm:pointcount} then follows in the full rank case.


By Theorem \ref{thm:finitesmooth}, $\cA_{\overline{\FF_q}}$ is smooth in all characteristics when $\tB$ has really full rank, so Poincar\'{e} duality holds in all characteristics.  The last statement of Theorem \ref{thm:eigenvalues} thus gives the last statement of Theorem \ref{thm:pointcount}.  \end{proof}

Deligne's theory of weights relates the exponents $w_{ki}$ in \eqref{eq:GrothDual} to a weight filtration $W_\bullet H^k$ on $H^k(\cA_{\oF_p},\QQ_\ell)$. Specifically, Frobenius preserves this filtration, and the action of Frobenius on $\Gr^{w}_W:=W_{w}/W_{w-1}$ is by eigenvalues of the form $\epsilon p^{w/2}$, where $|\epsilon|=1$.  Assume $p$ is large enough that $\cA_{\oF_p}$ is smooth.
A comparison theorem between etale and Betti cohomology (see \cite[Theorem 2.2]{KL}, \cite[Theorem 20.5 and 21.5]{Milne}, and \cite{FK}) gives
\begin{equation}\label{eq:comparison}
\dim \Gr^{2w}_W(H^k(\cA_{\oF_p},\QQ_\ell))= \dim H^{k,(w,w)}(\cA,\CC)
\end{equation}
and $\dim \Gr^{2w-1}_W(H^k(\cA_{\oF_p},\QQ_\ell))=0$.
Therefore, for each $k$ and $w$, we have $\dim H^{k, (w,w)}$ many characters $\chi_{kw1}$, $\chi_{kw2}$, \dots such that
\begin{equation}\label{eq:GrothDualWeights}
\#\cA(\FF_{q}) = \sum_{k=0}^{n+m} (-1)^{k} \sum_w q^{n+m-w} \sum_{r=1}^{\dim H^{k,(w,w)}(\cA)} \chi^{-1}_{kwr}(q) .
\end{equation}
%

We use the curious Lefschetz property to simplify \eqref{eq:GrothDualWeights}. 
\begin{proposition}
Suppose that $\cA$ satisfies the Louise property and is of full rank. Let $p$ be large enough that $\cA_{\oF_p}$ is smooth.  Let the eigenvalues of Frobenius on $\Gr^{2w}_W(H^k(\cA_{\oF_p},\QQ_\ell))$ be $\chi_{kw1}(p) p^w,\chi_{kw2}(p) p^w,\ldots,\chi_{kwr}(p) p^w$ where $1 \leq r \leq \dim H^{k,(w,w)}(\cA)$.  Then for a prime power $q = p^a$, we have
\begin{equation}\label{eq:GrothFinal}
\#\cA(\FF_{q}) = \sum_{k=0}^{n+m} (-1)^{n+m-k} \sum_w q^w \sum_{r=1}^{\dim H^{k,(w,w)}(\cA)} \chi_{kwr}(q).
\end{equation}
\end{proposition}

\begin{proof}
By Proposition \ref{prop:rowspanenough}, both~\eqref{eq:GrothDualWeights} and~\eqref{eq:GrothFinal} multiply by $q-1$ when we add an extra row to $\tB$ in the integer span of the previous rows.
So we may assume that $m+n$ is even, and we may apply Theorem \ref{thm:curious}.  Since $n+m$ is even, we have $(-1)^k = (-1)^{n+m-k}$.  Let $[\gamma] \in H^2(\cA_{\oF_p},\QQ_\ell)$ be such that $(\cA_{\oF_p},[\gamma])$ satisfies curious Lefschetz.  
Now, choose Frobenius eigenbases for the spaces $W(s,\ell,0)$ in \eqref{primitiveH}. For each such basis element $\alpha$, look at the Frobenius eigenvectors $\alpha$, 
$[\gamma] \alpha$, $[\gamma]^2 \alpha$, \dots, $[\gamma]^{\ell} \alpha$. The corresponding eigenvalues are of the form $p^i \chi$ for $(m+n)/2-\ell \leq i \leq (m+n)/2+\ell$.
The contributions to the RHS of \eqref{eq:GrothDualWeights} and of \eqref{eq:GrothFinal} are $(-1)^{s+\ell} \chi^{-1} \sum_{i=(m+n)/2-\ell}^{i=(m+n)/2+\ell} p^{n+m-i}$ and $(-1)^{s + \ell} \chi \sum_{i=(m+n)/2-\ell}^{i=(m+n)/2+\ell} p^i$ respectively.  The sums of powers of $p$ are equal. Since the characteristic polynomial of Frobenius has rational coefficients, if $p^i \chi$ is an eigenvalue of Frobenius, then so is the complex conjugate $p^i \chi^{-1}$.  This proves that \eqref{eq:GrothDualWeights} and \eqref{eq:GrothFinal} are equivalent.
\end{proof}

\subsection{Point counts and eigenvalues in rank 1}

\begin{lemma}\label{lem:isolated1}
Let $\cA$ be the rank one cluster variety with extended exchange matrix $\tB = (0,d)^T$, where we assume that $d \geq 0$.  Suppose $q = p^a$ is an odd prime power.
Then
\[
\#\cA(\FF_q) = q^2 +(c-2)q + 1
\]
where $c$ is the number of $d$-th roots of $-1$ in $\FF_q^\times$, given by
\[
c = \gcd(q-1,2d) - \gcd(q-1,d).
\]

\end{lemma}
\begin{proof}
We need to count the number of solutions $(x,x',y) \in \FF_q \times \FF_q \times \FF_q^{\times}$ satisfying
\[
xx' = y^d + 1.
\]
For a fixed non-zero value of the RHS, there are $q-1$ choices for $(x,x')$.  If the RHS is zero, then there are $2q-1$ choices for $(x,x')$.  The RHS is zero exactly when $y^d = -1$ in $\FF_q^{\times}$.  
\end{proof}

Let $X(d)$ be the multiset of Dirichlet characters obtained from the multiset $\bigsqcup_{c\mid 2d,\ c \notdivides d} \Dir(c)$ by deleting the trivial character of modulus $2d$.

\begin{proposition}\label{prop:eigen}
Let $\cA$ be the rank one cluster variety with extended exchange matrix $\tB = \left( \begin{smallmatrix} 0 \\ d \end{smallmatrix} \right)^T$ where $d \geq 0$.
Let $p$ be an odd prime not dividing $d$, and let $\ell$ be a prime distinct from $p$. Then Frobenius acts on $H^0(\cA_{\oF_p}, \QQ_{\ell})$ and $H^1(\cA_{\oF_p}, \QQ_{\ell})$ by multiplication by $1$ and $p$ respectively.  It acts on $H^2(\cA_{\oF_p}, \QQ_{\ell})$ with one eigenvalue equal to $p^2$ and the remaining eigenvalues equal to $ \chi(p)p$ for $\chi$ ranging over $X(n)$.  

If $d = 1$, the statement also holds with $p = 2$.
\end{proposition}

\begin{proof}
We first remark that $\cA_{\oF_p}$ is smooth whenever $p$ does not divide $d$.  For brevity, we write $\cA$ instead of $\cA_{\oF_p}$ in the rest of the proof.  By \eqref{eq:comparison} and Proposition \ref{prop:isolatedrank1}, the Betti numbers of $H^\ast(\cA, \QQ_{\ell})$ are $(1,1,d)$.

Let $T = (\oF_p^\times)^2$ be a cluster torus.  Frobenius acts on $H^i(T,\QQ_\ell)$ by $p^i$, and the restriction map $H^{\ast}(\cA,\QQ_\ell) \to H^{\ast}(T,\QQ_\ell)$ has one-dimensional image in each of $H^0$, $H^1$ and $H^2$.  This establishes the claim for $H^0$ and $H^1$, and accounts for the eigenvalue $p^2$ in $H^2$. 

Let the Frobenius eigenvalues on the kernel of $H^{2}(\cA,\QQ_\ell) \to H^{2}(T,\QQ_\ell)$  be $\epsilon_1p $, $\epsilon_2 p$, \dots, $\epsilon_{d-1}p $. 
From \eqref{eq:GrothDual}, we have
\[ \# \cA(\FF_{p^a}) = p^{2a} - p^a + 1 + p^a \sum_{i=1}^{d-1} \epsilon_j^{-a}. \]
On the other hand, from Lemma~\ref{lem:isolated1}, the number of such points is
\[ p^{2a} - p^a + 1+ p^a (\gcd(p^a-1,2d) - \gcd(p^a-1,d)-1)  = p^{2a} - p^a + 1 + p^a \sum_{\chi \in X(d)} \chi(p)^a . \]
So $\sum_j \epsilon_j^{-a} = \sum_{\chi \in X(d)} \chi(p)^a$ for all $a$.  Since the multiset $X(d)$ is invariant under $\chi \mapsto \chi^{-1}$, we deduce that the multiset $(\epsilon_1, \epsilon_2, \ldots, \epsilon_{d-1})$ is $X(d)$.
\end{proof}

\begin{example}
Let $\tB = \left( \begin{smallmatrix} 0 \\ 3 \end{smallmatrix} \right)$.  Then the Frobenius eigenvalues of $H^\ast(\cA_{\oF_p},\QQ_\ell)$ are given by
\[
\begin{array}{|r|lll|}
\hline
& H^0& H^1 & H^2 \\
\hline
k-p=0 & 1 & p & p^2 \\
1 & & & p,\ \left(\tfrac{-3}{p}\right) p\\
\hline
\end{array}\]
where $(\tfrac{-3}{p}) $ denotes a Legendre symbol. Compare to Table~\ref{tab:isolated}.
\end{example}

\subsection{Proof of Theorem \ref{thm:eigenvalues}}
By Proposition \ref{prop:eigen}, the statement of Theorem \ref{thm:eigenvalues} holds for the cluster variety $\cA(\tB)$, where $\tB = (0,d)^T$ for any $d \in \ZZ$.  By taking products, it thus holds for any cluster variety with extended exchange matrix of the form $\tB =  \left( \begin{smallmatrix} 0 \\ d\,\Id_n \\ 0 \end{smallmatrix} \right)$.

Let $\cA(\tB)$ be an isolated cluster variety of full rank.  By Proposition \ref{prop:coverprincipal}, there is a covering map $(\Psi,\Phi,\{R_t\}):\cA' \to \cA$, where $\cA'$ has extended exchange matrix $\tB' = \left( \begin{smallmatrix} 0 \\ d\,\Id_n \\ 0 \end{smallmatrix} \right)$ for some nonnegative integer $d$.  Let $H \subseteq \Aut(\cA')$ be the finite abelian group such that $\cA = \cA' \sslash H$.  Then we also have $\cA_{\oF_p}= \cA'_{\oF_p} \sslash H$.  The actions of Frobenius and of $H$ on the $\ell$-adic cohomology $H^\ast(\cA'_{\oF_p},\QQ_\ell)$ commute.  It follows that the eigenvalues of Frobenius acting on $H^\ast(\cA_{\oF_p},\QQ_\ell)$ is a subset of the eigenvalues of Frobenius acting on $H^\ast(\cA'_{\oF_p},\QQ_\ell)$.  Thus Theorem \ref{thm:eigenvalues} holds for isolated cluster varieties of full rank.

The Mayer-Vietoris exact sequence for $\ell$-adic cohomology is a long exact sequence of modules of the Galois group.  Thus if $\cA = U \cup V$ and the statement of Theorem \ref{thm:eigenvalues} holds for $U, V,$ and $U\cap V$, then it also holds for $\cA$.  By the definition of the Louise property, we see that Theorem \ref{thm:eigenvalues} holds for all cluster varieties $\cA$ satisfying the Louise property and of full rank.

Now suppose that $\cA$ satisfies the Louise property and has really full rank.  By Theorem \ref{thm:finitesmooth}, the variety $\cA$ is smooth in all characteristics.  The same argument as above shows that the eigenvalues of Frobenius must be powers of $p$.  This completes the proof of Theorem \ref{thm:eigenvalues}.

\begin{example}
Let
\[ \tB = \begin{pmatrix} 0 & 1 & 1 & 0 \\ -1 & 0 & 1 &0 \\ -1 & -1 & 0 & 1 \\ 0 & 0 & -1& 0 \end{pmatrix}.
\]
This is the matrix on right hand side of Table~\ref{tab:triangle}. The eigenvalues of Frobenius on the cohomology groups are
\[\begin{array}{|r|ccccc|}
\hline
& H^0 & H^1 & H^2 & H^3 & H^4 \\
\hline
k-p=0 & 1 & 0  & p^2 & 0 & p^4 \\
1 & & & & p^2 &  \\
\hline
\end{array}\]
so by \eqref{eq:GrothFinal}, we have $\#\cA(\FF_q) = (1+q^2+q^4)-q^2 =q^4 +1$. 
We can decompose $\cA$ into strata according to which of the cluster variables in the initial seed are nonzero, and get
\[ \cA = \GG_m^4 \sqcup (\GG_m^2 \times \AA^1) \sqcup (\GG_m^2 \times \AA^1) \sqcup (\GG_m^2 \times \AA^1) \sqcup (\GG_m^2 \times \AA^1) \sqcup \AA^2 \sqcup \AA^2 \]
Here $\GG_m = \AA^1 \setminus \{ 0 \}$.  We confirm our computation:
\[ \#\cA({\FF_q}) = (q-1)^4 + 4 (q-1)^2 q + 2 q^2 = q^4+1. \]
\end{example}


\begin{thebibliography}{xxxx}
\bibitem[BMRS]{BMRS} A.~Benito, G.~Muller, J.~Rajchgot, and K.~Smith, Singularities of locally acyclic cluster algebras. Algebra Number Theory 9 (2015), 913--936.
\bibitem[BFZ]{CA3} A. Berenstein, S.~Fomin, and A.~Zelevinsky, Cluster algebras. III. Upper bounds and double Bruhat cells. Duke Math. J. 126 (2005), no. 1, 1--52.
\bibitem[Cha11]{Cha} F.~Chapoton, On the number of points over finite fields on varieties related to cluster algebras. Glasg. Math. J. 53 (2011), no. 1, 141--151. 
\bibitem[Cha15]{Cha2} F.~Chapoton, On some varieties associated with trees. Michigan Math. J. 64 (2015), no. 4, 721--758. 
\bibitem[Del71]{Del} P.~Deligne, Th\'eorie de Hodge: II. Publ. Math. Inst. Hautes \'Etudes Sci., No. 40 (1971), 5--57.
\bibitem[Del74]{DelWeil} P.~Deligne, La conjecture de Weil. I. Publ. Math. Inst. Hautes \'Etudes Sci. No. 43 (1974), 273--307.

\bibitem[Dub]{Dubouloz} A.~Dubouloz, On the cancellation problem for algebraic tori, Annales de l'institut Fourier, \textbf{66} no. 6 (2016), p. 2621--2640.

\bibitem[FG]{FG} V.~Fock and A.~Goncharov, Moduli spaces of local systems and higher Teichm\"uller theory. Publ. Math. Inst. Hautes \'Etudes Sci. No. 103 (2006), 1--211.

\bibitem[FST]{FST} S.~Fomin, M.~Shapiro, and D.~Thurston, Cluster algebras and triangulated surfaces. I. Cluster complexes. Acta Math. 201 (2008), no. 1, 83--146.

\bibitem[FZ]{CA1}  S.~Fomin and A.~Zelevinsky, Cluster algebras. I. Foundations. J. Amer. Math. Soc. 15 (2002), no. 2, 497--529.

\bibitem[Fra]{Fra} C.~Fraser, Quasi-homomorphisms of cluster algebras,  Adv. in Appl. Math. \textbf{81} (2016), 40--77.

\bibitem[FK]{FK} E.~Freitag and R.~Kiehl,  \'Etale cohomology and the Weil conjecture. 
Ergebnisse der Mathematik und ihrer Grenzgebiete (3) 13. Springer-Verlag, Berlin, 1988. xviii+317 pp.

\bibitem[GL19]{GL} P.~Galashin and T.~Lam, Positroid varieties and cluster algebras, preprint, 2019; {\tt arXiv:1906.03501}.

\bibitem[GL20]{GL2} P.~Galashin and T.~Lam, Positroids, knots, and $q,t$-Catalan numbers, preprint, 2020; {\tt arXiv:2012.09745}.

\bibitem[GSV]{GSV} M.~Gekhtman, M.~Shapiro, and A.~Vainshtein, Cluster algebras and Poisson geometry. Mathematical Surveys and Monographs, 167. American Mathematical Society, Providence, RI, 2010. xvi+246 pp.

\bibitem[GY]{GY} K.R.~Goodearl and M.T.~Yakimov, The Berenstein–Zelevinsky quantum cluster algebra conjecture. J. Eur. Math. Soc. 22 (2020), 2453--2509.

\bibitem[GHKK]{GHKK}  M.~Gross, P.~Hacking, S.~Keel, and M.~Kontsevich, Canonical bases for cluster algebras,  JAMS \textbf{31} (2018), no. 2, 497--608.

\bibitem[HW]{HW} T.~Harima and J.~Watanabe, The strong Lefschetz property for Artinian algebras with non-standard grading. J. Algebra 311 (2007), no. 2, 511--537.
\bibitem[HR]{HR} T.~Hausel and F.~Rodriguez-Villegas, Mixed Hodge polynomials of character varieties. With an appendix by Nicholas M. Katz. Invent. Math. 174 (2008), no. 3, 555--624. 

\bibitem[Ing]{Ing} G.~Ingermanson, Cluster Algebras of Open Richardson Varieties, U.Michigan Ph.D. Thesis, 2019.

\bibitem[KL]{KL} D.~Kazhdan and G.~Lusztig,  Representations of Coxeter groups and Hecke algebras. Invent. Math. 53 (1979), no. 2, 165--184. 

\bibitem[KLS]{KLS} A.~Knutson, T.~Lam, and D.~Speyer, Positroid varieties: juggling and geometry. Compositio Mathematica 149 (2013), 1710--1752.

\bibitem[LP]{LP} T.~Lam and P.~Pylyavskyy, Laurent phenomenon algebras. Cambridge Journal of Mathematics 4 (2016), 121--162.

\bibitem[LS]{LS} T.~Lam and D.~Speyer, Cohomology of cluster varieties. II. Acyclic case, in preparation.

\bibitem[Lec]{Lec} B.~Leclerc, Cluster structures on strata of flag varieties,  \emph{Adv. Math.} \textbf{300} (2016), 190--228.

\bibitem[McD]{McDaniel} C.~McDaniel, The strong Lefschetz property for coinvariant rings of finite reflection groups.
J. Algebra 331 (2011), 68--95.
\bibitem[Mil]{Milne} J.~Milne, Lectures on Etale Cohomology, version 2.21, \\ \url{http://www.jmilne.org/math/CourseNotes/LEC.pdf}.
\bibitem[Mor]{Mor} S.~Morier-Genoud, Counting Coxeter's friezes over a finite field via moduli spaces, preprint, 2019; {\tt arXiv:1907.12790}.

\bibitem[Mul12]{Mul3} G.~Muller, The Weil--Petersson form on an acyclic cluster variety. IMRN 2012 (16), 3680--3692.
\bibitem[Mul13]{Mul} G.~Muller, Locally acyclic cluster algebras. Adv. Math. 233 (2013), 207--247. 
\bibitem[MS]{MS} G.~Muller and D.~Speyer, Cluster Algebras of Grassmannians are Locally Acyclic, Proc. Amer. Math. Soc. \textbf{144} (2016), no. 8, 3267--3281.
\bibitem[PS]{PS} C.~Peters and J.~Steenbrink, Mixed Hodge structures. Ergebnisse der Mathematik und ihrer
Grenzgebiete. 3. Folge., vol. 52, Springer-Verlag, Berlin, 2008.

\bibitem[RSW]{RSW} S.~Riche, W.~Soergel, and G.~Williamson, Modular Koszul duality. Compos. Math. 150 (2014), no. 2, 273--332.

\bibitem[Sco]{Scott} J.~Scott, Grassmannians and cluster algebras. Proc. London Math. Soc. (3) 92 (2006), no. 2, 345--380.

\bibitem[SSBW]{SSBW} K. Serhiyenko, M. Sherman-Bennett, and L. Williams, Cluster structures in Schubert varieties in the Grassmannian. Proc. Lond. Math. Soc. (3), 119(6):1694--1744, 2019.

\bibitem[Voi]{Voisin} C.~Voisin, Hodge theory and complex algebraic geometry. I. Translated from the French original by Leila Schneps. Cambridge Studies in Advanced Mathematics, 76. Cambridge University Press, Cambridge, 2002. x+322 pp.

\end{thebibliography}
\end{document}